%% file: sample_paper.tex
\documentclass[twoside]{article}

\usepackage[preprint]{aistats2027}

\input{math_commands.tex}

\usepackage{hyperref}
\usepackage{cleveref}
\usepackage{url}
\usepackage{booktabs}
\usepackage{longtable}

\usepackage{amsthm}
\usepackage{amsmath}
\usepackage{mathtools}
\usepackage{algorithm, algorithmic}
\usepackage{amssymb}
\usepackage{thmtools}
\usepackage{thm-restate}

\newtheorem{lemma}{Lemma}
\newtheorem{theorem}{Theorem}
\newtheorem{corollary}{Corollary}
\newtheorem{assumption}{Assumption}
\newtheorem{definition}{Definition}
\newtheorem{remark}{Remark}
\newtheorem{example}{Example}
\newtheorem{proposition}{Proposition}

\crefname{assumption}{Assumption}{Assumptions}
\crefname{theorem}{Theorem}{Theorems}
\crefname{corollary}{Corollary}{Corollaries}

\newenvironment{proofsketch}
{\begin{proof}[Proof sketch]}
{\end{proof}}
\usepackage{xcolor}
\definecolor{myred}{RGB}{180,0,0}

\newcommand{\conf}{\operatorname{conf}}
\newcommand{\Ind}[1]{\mathbf{1}\{#1\}}
\newcommand{\bmu}{\boldsymbol{\mu}}

\newcommand{\Prob}[1]{\mathbb{P}\left(#1\right)}
\usepackage{nicefrac}
\newcommand{\Exp}[1]{{\mathbb{E}}\left[#1\right]}
\newcommand{\ExpSub}[2]{{\mathbb{E}}_{#1}\left[#2\right]}

\DeclareMathOperator{\supp}{supp}
\newcommand{\regretdyn}{\mathcal{R}}
\renewcommand{\eqref}[1]{(\ref{#1})}

\usepackage{xcolor}

\newcommand{\tceil}[1]{\left\lceil #1 \right\rceil}
\newcommand{\tfloor}[1]{\left\lfloor #1 \right\rfloor}
\newcommand{\cI}{\mathcal{I}}

\newcommand{\ba}{\bm{a}}
\newcommand{\bs}{\bm{s}}
\newcommand{\bw}{\bm{w}}
\newcommand{\bu}{\bm{u}}

\newcommand{\cW}{\mathcal{W}}

\newcommand{\algname}[1]{\ifmmode\text{\scshape #1}\else{\scshape #1}\fi}

\newcommand{\eqdef}{\mathrel{\mathop:}=}

\usepackage{natbib}

\hypersetup{
  colorlinks   = true, %
  urlcolor     = blue, %
  linkcolor    = {red!75!black}, %
  citecolor   = {blue!50!black} %
}

\hypersetup{
  pdftitle={Improved Analysis of Adaptive Task Allocation and Non-Stationary Processing Times},
  pdfauthor={Grigory Begunov, Ilgam Latypov, Yuriy Dorn, Alexander Tyurin}
}

\begin{document}
\runningtitle{Improved Analysis of Adaptive Task Allocation}
\runningauthor{Begunov, Latypov, Dorn, Tyurin}

\twocolumn[

\aistatstitle{Improved Analysis of Adaptive Task Allocation \\ and Non-Stationary Processing Times}

\aistatsauthor{
  \begin{tabular}{@{}cccc@{}}
    Grigory Begunov & Ilgam Latypov & Yuriy Dorn & Alexander Tyurin \\
    {\normalfont AXXX\textsuperscript{1}, MSU\textsuperscript{2}} &
    {\normalfont IAI MSU\textsuperscript{3}} &
    {\normalfont IAI MSU\textsuperscript{3}} &
    {\normalfont AXXX\textsuperscript{1}, Applied AI Institute\textsuperscript{4}} \\
    \scalebox{0.98}{\normalfont\small\texttt{grigorii.begunov@math.msu.ru}} &
    \scalebox{0.98}{\normalfont\small\texttt{i.latypov@iai.msu.ru}} &
    \scalebox{0.98}{\normalfont\small\texttt{dornyv@my.msu.ru}} &
    \scalebox{0.98}{\normalfont\small\texttt{alexandertiurin@gmail.com}}
  \end{tabular}
}

\aistatsaddress{} ]

\begingroup
\renewcommand{\thefootnote}{}
\makeatletter
\renewcommand{\@makefntext}[1]{\noindent#1}
\makeatother
\footnotetext[0]{
  \textsuperscript{1}AXXX, Moscow, Russia. \\
  \textsuperscript{2}Lomonosov Moscow State University, Moscow, Russia. \\
  \textsuperscript{3}Institute for Artificial Intelligence, Lomonosov Moscow State University,
  Moscow, Russia. \\
  \textsuperscript{4}Applied AI Institute, Moscow, Russia.
}
\endgroup

\begin{abstract}
  We study the task allocation problem, where our goal is to develop an online algorithm that efficiently distributes tasks among workers based on their computational performance. This task is challenging since computation times are \emph{random} and \emph{non-stationary} in practice. Recent works have introduced different online algorithms to address this problem. One of them is Adaptive Task Allocation (ATA), a bandit-style algorithm. However, previous theoretical guarantees leave substantial room for improvement, and existing methods are designed only for stationary computation times. Using new proof techniques, our work establishes an improved analysis of ATA with better gap-dependent and gap-independent regret guarantees, improving upon all previous approaches. We also introduce a new method, Sliding Window Adaptive Task Allocation (SW-ATA), designed for more practical non-stationary environments. Finally, we prove new lower bounds in both the stationary and non-stationary settings.
\end{abstract}
\section{Introduction}
In this work, we consider the task management problem. 
Assume that there are $n$ workers or nodes with random computational delays and unknown time distributions, which can perform tasks in parallel. 
At each iteration, they need to perform $B$ tasks. One of the recent applications 
of this framework is neural network training, where at each iteration $B$ gradients need to be computed by $n$ heterogeneous devices. For instance, \citet{tyurin2023optimal} designed an optimal SGD-like method that computes $B$ stochastic gradients asynchronously across the devices.
However, modern asynchronous methods \citep{dutta2018slow,tyurin2023optimal,islamov2024asgrad,maranjyan2025ringmaster} may be wasteful in terms of total worker utilization time, and they require \emph{all devices all the time} to work in every iteration. Moreover, all these methods greedily aggregate the calculated tasks, which is impractical because, in modern frameworks, task allocation must be known a priori before each iteration to support highly optimized communication, model-parallel, and memory-efficient implementations \citep{sergeev2018horovod,shoeybi2019megatron,rajbhandari2020zero}.

This motivates the main question: instead of greedily waiting for and aggregating tasks, can we distribute them efficiently through adaptive task allocation?

The challenge is that an efficient allocation depends on worker speeds, which must be learned
from observed task durations.
Additionally, in real systems, workers' computation times can change during execution: workers can temporarily slow down because of resource contention, background jobs, network congestion, or garbage collection \citep{dean2013tail,peng2018optimus}.
Such side effects are common in shared clusters and can turn a normally fast worker into a straggler for
several iterations. As a result, an allocation that was good at the beginning may become inefficient later.
This motivates the non-stationary processing setting, where the algorithm must adapt to changes in worker
processing times.

Our goal is to provide strong theoretical guarantees for an efficient task allocation method that adapts to heterogeneous and changing worker speeds.
We aim to reduce the total worker utilization time caused by stragglers by learning task allocations online and analyzing an approach in which the allocation is known before every iteration.

\subsection{Problem Setup}
\label{sec:setup}
We consider the task-allocation problem with non-stationary and random processing times. 
In each round $k \in [K]$, we distribute $B$ tasks among
$n$ workers. Worker $i$'s delay at iteration $k$ has an unknown distribution with mean $\mu_{i, k}>0$.
We first use an allocation algorithm to choose an allocation $\ba_k \in \mathcal A$ at iteration $k$. The allocation set is defined as
$\mathcal A = \left\{
      \ba \in \mathbb Z_{\geq 0}^n:\,\sum_{i = 1}^n a_{i} = B\right\}$. 
Once the tasks are completed, we observe $a_{i, k}$ observations $\{X_{i, k}^{(u)}\}_{u = 1}^{a_{i, k}}$ 
from each worker and observe the completion time
$C_k(\ba_k) := \max_{i \in \supp(\ba_k)} \sum_{u = 1}^{a_{i, k}} X_{i, k}^{(u)}$,
where $\supp(\ba):=\{i\in[n]:a_i>0\}$ denotes the set of workers assigned at least one task.
Our main goal is to minimize the total expected completion time 
\begin{align}
\label{eq:eZJWZzDcwN}
\textstyle C_K \eqdef \sum\limits_{k = 1}^K \Exp{C_k(\ba_k)},
\end{align}
where we choose $\ba_k$ in an online fashion.
Let $\ba_k^* \in \argmin_{\ba \in \mathcal A} \Exp{C_k(\ba)}$ denote an allocation that minimizes the expected completion time in round $k$.
The minimum total expected completion time is $C_K^* = \sum_{k = 1}^K \Exp{C_k(\ba_k^*)}$. Note that finding $\ba_k^*$ directly is infeasible because
$\Exp{C_k(\ba)}$ depends on the unknown delay distribution of $X_{i, k}.$ 

\textbf{Reduction to regret minimization.} To solve the problem, similarly to \citep{maranjyan2025ata}, 
we introduce the proxy loss $\ell_k(\ba) \eqdef \max_{i\in[n]}a_i\mu_{i, k}$ and instead of minimizing (\ref{eq:eZJWZzDcwN}), we minimize
$\ell_k(\ba)$. We then define the expected regret of the proxy loss 
$$\textstyle \regretdyn := \sum\limits_{k=1}^K \Exp{\ell_k(\ba_k) - L_k^*},$$
where we define an arbitrary minimizer of the proxy loss $\bar \ba_k \in \argmin_{\ba \in \mathcal A} \ell_k(\ba),$ the minimum of the proxy loss $L_k^* := \ell_k(\bar \ba_k),$ and $\bar \regretdyn := \sum_{k=1}^K \left(\ell_k(\ba_k) - L_k^*\right).$ 
\emph{Crucially, in Theorems~\ref{thm:stationary_prox_connection} and \ref{thm:oRtZJgVnZogTrbAAt}, we show that we can bound $C_K$ using the expected regret $\regretdyn$.}

\subsection{Assumptions on delays}\label{sec:delays}
We assume that delays $X_{i, k}^{(u)}$ are drawn from a sub-exponential distribution. 
\begin{definition}
\label{def:subexp}
A random variable $X$ is sub-exponential if $\|X\|_{\psi_1} < + \infty$, where
$\|X\|_{\psi_1} := \inf\left\{t>0 : \mathbb{E}\left[\exp\left(\nicefrac{|X|}{t}\right)\right]\le 2 \right\}$.
\end{definition}

\begin{assumption}\label{ass:drift-exp}
The random variables $X_{i,k}^{(u)}$ are sub-exponential, positive almost surely (a.s.), and mutually independent
for all $i \in [n]$, $k \in [K]$, and $u \in [B]$.
For each fixed $i \in [n]$ and $k \in [K]$, the variables $X_{i,k}^{(u)}$, $u \in [B]$, are identically distributed and satisfy
$\mathbb{E}[X_{i,k}^{(u)}] > 0$. Let $\mu_{i,k} \eqdef \mathbb{E}[X_{i,k}^{(u)}]$ and
$\alpha_{i,k} \eqdef \|X_{i,k}^{(u)} - \mu_{i,k}\|_{\psi_1}$.
\end{assumption}
We then define $\alpha_{\max} = \max_{i \in [n], k \in [K]} \alpha_{i, k}$
and $\tilde \alpha = \max_{i, k} \|X_{i, k}^{(u)}\|_{\psi_1}$.
Note that for a positive random variable $X$ satisfying Definition~\ref{def:subexp}, we have $\|X - \mu\|_{\psi_1} \leq 2 \|X\|_{\psi_1}$. Therefore, $\alpha_{\max} \leq 2 \tilde \alpha$.

Sub-exponential random variables form a broader class than the sub-Gaussian class and include
distributions with heavier tails, such as exponential and gamma
distributions \citep{vershynin2019high}. Such distributions are natural models for positive random
delays \citep{mitliagkas2016asynchrony, dutta2018slow}.

\subsection{Related work}
\begin{table*}[t]
\caption{\emph{Top table}: Bounds on the total expected completion time $C_K = \sum_{k=1}^K \mathbb{E}[C_k(\ba_k)]$ in the stationary case.
\emph{Bottom table}: Bounds on the expected proxy-loss regret $\regretdyn := \sum_{k=1}^K \Exp{\ell(\ba_k) - L^*}$.
For comparison, we assume $0 < X_{i,k}^{(u)} \leq 1$ a.s., $K \geq \max\left\{B, \nicefrac{n}{B}\right\}$, $n \geq 2$, and $B \geq 2$.
Abbr.: $n$ is the \# of workers; $B$ is the batch size; and $K$ is the \# of rounds;
$C_K^*$ is the optimal total expected completion time; $\mu_i$ is worker $i$'s mean processing time; $b_i$ and $\Delta_i$ are the over-allocation threshold
and suboptimality gap defined in Definition~\ref{def:suboptimality_gaps}.
The upper bounds on regret $\regretdyn^{\text{ATA}}$, $\regretdyn^{\text{Tak21}}$, $\regretdyn^{\text{Gen}}$, and
$\regretdyn^{\text{Ours}}$ are defined in the bottom table.}
\label{table:stationary_comparison}
\centering
\setlength{\tabcolsep}{4pt}
\setlength{\arrayrulewidth}{0.3pt}
\fontsize{8}{10}\selectfont
\renewcommand{\arraystretch}{1.8}
\begin{minipage}[t]{0.75\linewidth}
\centering
\vspace{0.1em}
\begin{tabular*}{\linewidth}{c@{\extracolsep{\fill}}c}
\toprule
\textbf{Results} & \textbf{Bound on} $C_K$ \\
\midrule[\heavyrulewidth]
{\renewcommand{\arraystretch}{1}
\begin{tabular}[c]{@{}c@{}}Previous \\ \citep{maranjyan2025ata} \\
\citep{takemura2021near} \\ \citep{nakamura2025multi}\end{tabular}}
& $\left(1 + 4\textcolor{myred}{\eta}\ln B\right)
   \left(C_K^* + \textcolor{myred}{\min\{\regretdyn^{\text{ATA}}, \regretdyn^{\text{Tak21}}, \regretdyn^{\text{Gen}}\}}\right)$ \\

\midrule
Ours (Theorem~\ref{thm:stationary_prox_connection})
& $\left(1 + 4\varphi\ln \min\left\{n, B\right\}\right)
   \left(C_K^* + \regretdyn^{\text{Ours}}\right)$ \\
\bottomrule
\end{tabular*}
\end{minipage}
\par\vspace{0.2em}
{The coefficients $\eta$ and $\varphi$ are defined in
Section~\ref{sec:ata_large_batch}; $\varphi \leq \eta$ by definition.\\
For fixed $n$ and worker distributions,
Proposition~\ref{prop:large_batch_phi} shows that
$\varphi = \mathcal O(B^{-1/2})$ as $B \to \infty$, while $\eta$ remains constant.
\par}
\par\vspace{0.8em}
\begin{minipage}[t]{\linewidth}
\centering
\begin{tabular*}{\linewidth}{l@{\extracolsep{\fill}}ccc}
\toprule
\multicolumn{2}{c}{} & \multicolumn{2}{c}{\textbf{Bound on proxy-loss regret} ($\boldsymbol{\regretdyn}$)} \\
\cmidrule{3-4}
\textbf{Method} & \textbf{Notation} & \textbf{Gap-independent} & \textbf{Gap-dependent} \\
\midrule[\heavyrulewidth]
ATA~\citep{maranjyan2025ata}
& $\regretdyn^{\text{ATA}}$
& --
& $\mathcal O\left(\sum_{i=1}^n \textcolor{myred}{\frac{B\mu_i-L^*}{\Delta_i b_i}}
   \frac{\log K}{\Delta_i}\right)$ \\

\midrule
Tak21~\citep{takemura2021near}
& $\regretdyn^{\text{Tak21}}$
& $\widetilde{\mathcal O}\left(\textcolor{myred}{n}\sqrt{BK}\right)$
& -- \\

\midrule
GenCTS~\citep{nakamura2025multi}
& $\regretdyn^{\text{Gen}}$
& --
& $\Omega \left(\textcolor{myred}{B\log(\min\{n,B\})}
   \sum_{i=1}^n\frac{\log K}{\Delta_i}\right)$\textsuperscript{(b)} \\

\midrule
ATA (\textbf{ours})\textsuperscript{(a)} (Theorem~\ref{thm:ata_regret})
& $\regretdyn^{\text{Ours}}$
& $\widetilde{\mathcal O}\left(\min\left\{\sqrt{n B K}, \sqrt{K(B+n)}
   + \frac{n}{\min_{i \in [n]}\mu_i}\right\}\right)$
& $\mathcal O\left(\sum_{i=1}^n\frac{\log K}{\Delta_i}\right)$ \\

\midrule
\midrule
Lower bound (\textbf{ours}) (Theorem~\ref{thm:lower_bound})
& --
& $\Omega\left(\sqrt{K \left(\frac{B}{n} + \frac{n}{B}\right)}\right)$
& -- \\
\bottomrule
\end{tabular*}

\vspace{0.2em}
{\fontsize{8}{10}\selectfont\raggedright
\textsuperscript{(a)} $\regretdyn^{\text{Ours}} :=
\min\{\text{gap-independent bound},\,\text{gap-dependent bound}\}$
from the ATA (ours) row. 
\textsuperscript{(b)} For a clearer comparison, we display a lower bound on the leading term in $\regretdyn^{\text{Gen}}$.
In the table we present only the terms that depend on $K$, omitting constants.
\par}
\end{minipage}
\end{table*}

The adaptive allocation problem considered here generalizes the stationary version introduced in \citep{maranjyan2025ata} and extends it to a more practical setting in which the distribution may vary from iteration to iteration. 

When $B = 1$ and the means are stationary, this problem can be reduced to the classical multi-armed bandit problem.
The literature on multi-armed bandits is extensive; we refer to
\citet{lattimore2019bandit} and \citet{slivkins2019introduction} for
introductions to the subject. A classical finite-time analysis of stochastic
bandits and the UCB1 algorithm was given by \citet{auer2002finite}.

The combinatorial semi-bandit framework is closely related to our problem \citep{chen2013combinatorial,kveton2015tight,chen2016combinatorial}.
In the standard formulation, an action selects a subset of base arms and is represented by a binary vector. In our setting, an action is an integer allocation
$\ba$, allowing multiple tasks to be assigned to the same worker. Each task provides an independent observation of that worker's processing time: assigning
$a_i$ tasks to worker $i$ yields $a_i$ observations.

When $B > 1$, assigning multiple tasks to a worker corresponds to repeated pulls of the same arm, connecting our problem to bandit settings that allow repeated pulls.
Related work includes adaptive task allocation \citep{maranjyan2025ata} and the more general frameworks of \citet{takemura2021near,nakamura2025multi}.

Non-stationary bandits study learning in changing environments; our problem reduces to this setting when $B = 1.$ One approach uses sliding-window confidence bounds, as studied by
\citet{garivier2011upper}. For multi-armed bandits with
mean variation $V_K=\sum_{k=1}^{K-1}\|\bmu_{k+1}-\bmu_k\|_\infty$,
\citet{besbes2014stochastic} establish
a regret rate of $\widetilde{\mathcal O}((nV_K)^{1/3}K^{2/3})$.
\citet{cheung2019learning} obtain matching upper and lower bounds for non-stationary linear bandits. 
\citet{chen2021combinatorial} extend this line of work to combinatorial
semi-bandits with nonlinear rewards and provide both gap-dependent and
gap-independent guarantees. Although their framework accommodates Lipschitz-continuous objectives,
it requires the random variables to take values in $[0, 1]$.

\subsection{Contributions}\label{sec:motivation}
(i) In this work, we consider the adaptive task allocation problem from Section~\ref{sec:setup} in both stationary and non-stationary computational settings. In Section~\ref{sec:new}, we develop a new method called SW-ATA (Algorithm~\ref{alg:sw_ata}), designed to minimize~(\ref{eq:eZJWZzDcwN}), which uses a sliding window mechanism to adapt to non-stationary environments.

(ii) We first present our analysis in the stationary setting, where we improve on the proxy-loss regret bounds of all previous approaches \citep{takemura2021near,nakamura2025multi,maranjyan2025ata} (see Table~\ref{table:stationary_comparison} and Section~\ref{sec:stat}). In particular, using a new counting argument, we reduce the dependence on the gaps in the bound of~\citet{maranjyan2025ata} from $\nicefrac{1}{\Delta_i^2}$ to $\nicefrac{1}{\Delta_i}$ and improve on the bound for GenCTS \citep{nakamura2025multi}, which is at least $B\log(\min\{n,B\})$ times larger than ours. Additionally, we obtain a gap-independent guarantee that describes the algorithm's worst-case performance and is better than that of~\citet{takemura2021near} by a factor of $\sqrt{n}$. Compared with previous approaches, we not only improve the regret bound but also provide a tighter bound on the total expected completion time (see Table~\ref{table:stationary_comparison}) by replacing $\eta$ with a sharper quantity $\varphi$ that decreases at the rate
$\mathcal O\left(B^{-\nicefrac{1}{2}}\right)$ as the batch size $B$ grows. 
Thus, we provide the first guarantees showing that $C_K/C_K^*$ can be arbitrarily close to $1$
for sufficiently large $B$ and $K$, while also improving the regret bound.

(iii) Next, in Section~\ref{sec:sw_ata}, to the best of our knowledge, we provide the first theoretical guarantees for the more challenging non-stationary computational setting. We analyze the regret of SW-ATA using our refined counting argument.

(iv) To the best of our knowledge, we provide the first regret lower bounds for the task allocation problem, helping to clarify the complexity gap. While \algname{SW-ATA} does not match these bounds, it comes closest among all current methods. See Table~\ref{table:stationary_comparison} and Section~\ref{sec:lower_bounds}.

\vspace{-3pt}
\section{SW-ATA: A New Algorithm in a Non-Stationary Environment}
\label{sec:new}
\vspace{-2pt}
In this section, we present a new algorithm called Sliding Window Adaptive Task Allocation (SW-ATA) (Algorithm~\ref{alg:sw_ata}).
The idea is to extend the method in \citep{maranjyan2025ata}: we estimate worker delay means and, following the principle which is called optimism in the face of uncertainty, we assign tasks according to the optimal allocation with respect to lower confidence bounds.

The algorithm requires a parameter $\alpha \geq \alpha_{\max}$.
In non-stationary environments, we implement a sliding-window mechanism in order to "forget" outdated statistics. 
We introduce an additional parameter, the window size $W\in\{1,\dots,K\}$. 
We then, for each round $k$, define the window of previous observations by
$\cW_k := \{q:\max(1,k-W)\leq q\leq k-1\}$ and define the number of accumulated observations for
each worker $i$ as $N_{i,k}^{(W)} := \sum_{q\in\cW_k}a_{i,q}$.
For $N_{i,k}^{(W)}\geq1$, we define the sliding-window empirical mean
and the confidence radius:
\vspace{-2pt}
\begin{align*}
\textstyle \hat\mu_{i,k}^{(W)}
\eqdef \frac{1}{N_{i,k}^{(W)}} \sum\limits_{q\in\cW_k}\sum\limits_{u=1}^{a_{i,q}}X_{i,q}^{(u)}
\end{align*} 
\vspace{-2pt}
and
\vspace{-2pt}
\begin{align}\label{eq:lksjnfkjhfd}
\textstyle \conf_W(i,k) \eqdef 2\alpha\max\left\{ \sqrt{\frac{\log(2Bk^3)}{N_{i,k}^{(W)}}}, 
   \frac{\log(2Bk^3)}{N_{i,k}^{(W)}}\right\}.
\end{align}
Then, the lower confidence bound (LCB) is \begin{align}
   \label{eq:kAGPwqp}
   s_{i,k} \eqdef \left(\hat\mu_{i,k}^{(W)}-\conf_W(i,k)\right)_+.
\end{align}
For $N_{i,k}^{(W)}=0$, the empirical mean and the confidence radius are undefined, and we set $s_{i,k}=0$.

It is important to choose an appropriate window size $W:$ a small $W$ may yield inaccurate sliding-window empirical means. A large $W$ reduces the variance of the estimates but may include outdated observations, introducing estimation bias. For example, a worker may perform well initially but later become a straggler.

It is convenient to define the following quantities. 
\begin{definition}[Suboptimality Gaps]\label{def:suboptimality_gaps}
The minimal number of tasks that implies over-allocation 
of the $i$th worker at the $k$th iteration is $b_{i, k} := \min\left\{b \in [B + 1]: b \mu_{i, k}>L_k^* \right\}$.
We then define the minimal suboptimality gap 
$\Delta_i^{\min} := \min_{k \in [K]: b_{i, k} \leq B}\left\{\mu_{i, k}-\nicefrac{L_k^*}{b_{i, k}}\right\} > 0$ with the standard convention that $\min$ of an empty set is $+\infty.$
\end{definition}
We call a regret bound gap-dependent if it depends on the suboptimality gaps, and gap-independent otherwise.

\begin{algorithm}[h]
\caption{SW-ATA}\label{alg:sw_ata}
\begin{algorithmic}[1]
\STATE \textbf{Input:} number of workers $n$, number of tasks $B$, $\alpha>0$, window length $W$
\FOR{$k=1,\dots,K$}
    \STATE Compute LCBs $s_{i,k}$ for all $i\in[n]$ using (\ref{eq:kAGPwqp})
    \STATE Set $\mathcal Z_k=\{i\in[n]:s_{i,k}=0\}$
    \IF{$\mathcal Z_k=\varnothing$}
        \STATE Set $\ba_k= \text{RAS}(\bs_k; B)$ (Algorithm~\ref{alg:ras}) 
    \ELSE
        \STATE Set $a_{i,k}=0$ for $i\notin\mathcal Z_k$ and split $B$ tasks uniformly\footnotemark among $i\in\mathcal Z_k$
    \ENDIF
    \STATE Allocate $a_{i,k}$ tasks to each worker $i \in [n]$
    \STATE Observe $X_{i,k}^{(u)}$ for $u=1,\dots,a_{i,k}$ after the completion of the tasks
\ENDFOR
\end{algorithmic}
\end{algorithm}
\footnotetext{Set $a_{i,k} = q + 1$ for the first $r$ workers of $\mathcal Z_k$, and $a_{i,k} = q$ for the remaining $|\mathcal Z_k| - r$ workers, where $q = \tfloor{B / |\mathcal Z_k|}$ and $r = B - |\mathcal Z_k| q$.}

\section{Stationary Processing Times: Improved Analysis}
\label{sec:stat}
In this section, we temporarily consider the stationary setting.
We show that our new theoretical results provide better guarantees than those in \citep{takemura2021near,nakamura2025multi,maranjyan2025ata}. We use this assumption only in this section.
\begin{assumption}\label{ass:stationary-exp}
In addition to Assumption~\ref{ass:drift-exp}, for each $i \in [n]$,
the random variables $X_{i,k}^{(u)}$, $k \in [K]$, $u \in [B]$, are identically distributed.
\end{assumption}

Note that ATA is a special case of SW-ATA when $W=K$. In this setting, we simplify the notation by omitting the window index $W$:
the number of accumulated samples becomes $N_{i,k}:=\sum_{q=1}^{k-1}a_{i,q}$,
and we write $\conf(i,k) \equiv \conf_W(i,k)$ for the confidence radius.
We also simplify quantities that are constant due to stationarity,
writing $\mu_i$, $\alpha_i$, $\ell(\ba)$, $L^*$, $b_i$, and $\Delta_i$
instead of $\mu_{i,k}$, $\alpha_{i,k}$, $\ell_k(\ba)$, $L_k^*$, $b_{i,k}$, and $\Delta_i^{\min}$, respectively.
Additionally, due to Assumption~\ref{ass:stationary-exp}, the optimal allocations $\bar \ba_k$ and $\ba_k^*$ are independent of
$k$, so we omit their index. Finally, we define $C^* \eqdef \Exp{C_1(\ba^*)}$ and note that $C_K^* = K C^*$.

\subsection{Better regret bound}\label{sec:refined_ata_bound}
In the following theorem, we provide both a gap-independent and a gap-dependent bound. 

\begin{theorem}[Regret of ATA]\label{thm:ata_regret}
Let Assumption~\ref{ass:stationary-exp} be satisfied, and let
$\beta_K := 4\log(2BK^3)$.
Then, ATA (Alg.~\ref{alg:sw_ata} with $W=K$) with parameter
$\alpha \geq \alpha_{\max}$ satisfies the gap-dependent bound
\begin{align*}
\textstyle \regretdyn &\textstyle\leq \mathcal O\Big(n\max_{i \in [n]}\left\{B\mu_i-L^*\right\}
+ \alpha^2 \beta_K \sum_{i = 1}^n \frac{1}{\Delta_i} \\
&\textstyle\quad + \alpha \beta_K \sum_{i = 1}^n
   \left(\log \frac{B\mu_i-L^*}{\alpha B}\right)_+
\Big).
\end{align*}
Moreover, its gap-independent regret is
\begin{align*}
\textstyle \regretdyn
&\textstyle\leq \widetilde{\mathcal O}\Big(
   n\max_{i \in [n]}\left\{B\mu_i-L^*\right\}
   + \min\Big\{\alpha\sqrt{nBK}, \\
&\textstyle \qquad \alpha\sqrt{K(B+n)}
   + \frac{\alpha^2 n}{\min_{i \in [n]}\mu_i}
   \Big\}\Big).
\end{align*}
\end{theorem}
Our analysis improves the gap-dependent proxy-loss regret bounds of
\citet{maranjyan2025ata} and \citet{nakamura2025multi} and the gap-independent bound of
\citet{takemura2021near}.
Table~\ref{table:stationary_comparison} compares these bounds under the common assumption
$0 < X_{i,k}^{(u)} \leq 1$ a.s., highlighting the improved dependence on the batch size,
the number of workers, and the suboptimality gaps.

Let us compare the old and new gap-dependent bounds for ATA in more detail. 
Both bounds include a constant term $\mathcal O\left(n\max_{i \in [n]}\left\{B\mu_i-L^*\right\}\right)$, 
so below we compare terms that depend on $K$.
The original ATA bound by \citet{maranjyan2025ata} is
$\Theta\left(\tilde \alpha^2 \sum_{i = 1}^n \left(\nicefrac{B\mu_i-L^*}{\Delta_i b_i}\right) \cdot \nicefrac{\log K}{\Delta_i}\right)$. 
Let us show that our bound is tighter.
When we set $\alpha = 2 \tilde \alpha$ as an input to the algorithm (which is not a tight choice but makes the comparison easier),
Theorem~\ref{thm:ata_regret} yields regret of the order 
$\mathcal O \left(\tilde \alpha^2 \beta_K \sum_{i = 1}^{n} \frac{1}{\Delta_i}
\right)$. Indeed, by Lemma~\ref{lemma:expectation_via_orlicz}, we have 
$\ln \nicefrac{B\mu_i - L^*}{\tilde \alpha B} \leq \ln \nicefrac{\mu_i}{\tilde \alpha} \leq \ln \ln 2 \leq 0$.
We then show that our bound is never worse. Indeed, for $i$ such that $\Delta_i = +\infty$,
the corresponding terms in both bounds $\left(\nicefrac{B\mu_i-L^*}{\Delta_i b_i}\right) \cdot \nicefrac{\log K}{\Delta_i}$
and $\nicefrac{\log K}{\Delta_i}$ are equal to $0$.
If $\Delta_i \neq +\infty$, then $b_i \leq B$ and 
$\nicefrac{B\mu_i-L^*}{\Delta_i b_i} = \nicefrac{B\mu_i-L^*}{b_i \mu_i-L^*} 
   = 1 + \nicefrac{\left(B - b_i\right) \mu_i}{b_i \mu_i-L^*} \geq 1$. 
Thus, the new bound is never worse and becomes strictly better if there exists 
a worker $i$ such that $b_i < B$.
\begin{example}
Consider $n=3$ workers with
$\mu_1=\mu_2=\nicefrac{2}{B}$ and $\mu_3=\nicefrac{1-\varepsilon}{2}$,
where $B \geq 4$ is even and $0<\varepsilon<\nicefrac{1}{2B}$.
In this case, the optimal allocation assigns $\nicefrac{B}{2}-1$ tasks to the first two workers
and $2$ tasks to the last one.
Thus, $\bar a_1\mu_1=1-\nicefrac{2}{B}$, $\bar a_3\mu_3=1-\varepsilon$, and
$L^*=1-\varepsilon$.
Moreover, $\Delta_1=\nicefrac{2\varepsilon}{B}$ and
$\mu_1-\nicefrac{L^*}{B}=\nicefrac{1+\varepsilon}{B}\geq\nicefrac{1}{B}$,
while $\Delta_3=\nicefrac{1-\varepsilon}{6}\geq\nicefrac{1}{12}$.
Therefore, the previous gap-dependent term is
$\Omega\left(\tilde \alpha^2 B\log K/\varepsilon^2\right)$,
while the corresponding term in our bound is
$\mathcal O\left(\tilde \alpha^2 B\log K/\varepsilon\right)$.
Thus, the new guarantee can be arbitrarily better as $\varepsilon\to0$.
\end{example}

The results by \citet{takemura2021near, nakamura2025multi} can also be applied and used in our problem since $\ell(\ba,\bmu)=\max_{i\in[n]} a_i \mu_i$ is $B$-Lipschitz in $\bmu$.
However, the framework of~\citet{takemura2021near} yields a gap-independent bound only under $\sigma$-sub-Gaussian noise.
The recent work of~\citet{nakamura2025multi} directly considers integer actions and multiple observations from the same arm.
For distributions supported on $[0,1]$, their GenCTS analysis gives the gap-dependent bound in Table~\ref{table:stationary_comparison}. Our new result provides better guarantees for our problem.
\begin{remark}\label{remark:gap_comparison}
In our work, $\Delta_i$ measures how accurately the mean processing time of worker $i$ must be estimated to stop over-allocation of this worker. At the same time, \citet{nakamura2025multi} define and use
$\Delta_i^{\mathrm{MP}}
:= \min_{\ba \in \mathcal A \setminus \{\bar \ba\}, a_i \geq 1}
\left(\ell(\ba)-L^*\right),$
which depends on the entire allocation, as is common in combinatorial bandit
analyses~\citep{chen2016combinatorial}.
In Table~\ref{table:stationary_comparison}, we use the bound $\Delta_i^{\mathrm{MP}} \leq b_i \Delta_i$ to simplify the comparison; thus, the original guarantees of~\citet{nakamura2025multi}, which depend on $\Delta_i^{\mathrm{MP}}$, may be even worse.
\end{remark}

\subsection{Proof sketch}\label{sec:proof_sketch}
Setting $\alpha = 2 \tilde \alpha$ as an input to the algorithm for a simpler presentation, we introduce the proof idea; see
Appendix~\ref{sec:main_regret_analysys} for the formal proof for arbitrary $\alpha \geq \alpha_{\max}$.
First, we briefly outline the proof idea from \citep{maranjyan2025ata}.
They fix a worker $i$ and consider the set of rounds
$\mathcal S_i := \left\{k \in [K]: a_{i, k} \mu_{i} > L^*, \,
a_{i, k} \mu_{i} = \ell(\ba_k)\right\}$.
They then bound
\begin{align}\label{eq:oisjfiojiorwje}
\textstyle\bar \regretdyn \leq \sum_{i = 1}^n \sum_{k \in \mathcal S_i}
\left(a_{i, k} \mu_{i} - L^*\right),
\end{align}
since if the algorithm makes a mistake in the $k$th round, then there exists a worker $i$
such that $a_{i, k} \mu_{i} = \ell(\ba_k)$.
\citet{maranjyan2025ata} show that $|\mathcal S_i| \leq O\left(\nicefrac{\tilde \alpha^2 \ln K}{b_i \Delta_i^2}\right)$,
where $|\cdot|$ denotes cardinality.
Combining this with $a_{i, k} \mu_{i} - L^* \leq B \mu_{i} - L^*$,
they obtain the regret bound
$\mathcal O\left(\tilde \alpha^2 \sum_{i = 1}^n
\left(\nicefrac{B\mu_i-L^*}{\Delta_i b_i}\right)
\cdot \nicefrac{\log K}{\Delta_i}\right)$.
This analysis can be loose for large $B$, since it replaces the actual allocation
$a_{i,k}$ of the over-allocated worker by the worst-case value $B$.

We instead analyze the sum in~(\ref{eq:oisjfiojiorwje}) more accurately using our technical lemma:
\begin{restatable}[Decomposition by confidence levels]{lemma}{GAPSUMMATION}
\label{lemma:weighted_gap_summation}
Let $a_1,\ldots,a_m$ and $\Delta_1,\ldots,\Delta_m$ be non-negative sequences.
Let $0 < \delta_1 < \cdots < \delta_J$ be the distinct positive values among
$\Delta_1,\ldots,\Delta_m$, and set $\delta_0 = 0$. Then,
$\sum_{k=1}^m a_k \Delta_k
= \sum_{j=1}^J \left(\delta_j-\delta_{j-1}\right)
\sum_{k:\,\Delta_k\geq\delta_j} a_k$.
\end{restatable}

The key difference is that we sum the actual loss
$a_{i,k}\mu_i-L^*$ incurred in each round instead of upper bounding it by
$B\mu_i-L^*$.

\textbf{Step 1.} We write~(\ref{eq:oisjfiojiorwje}) as
$\bar \regretdyn \leq \sum_{i = 1}^n \sum_{k \in \mathcal S_i}
a_{i, k} \Delta_{i, k}$,
where $\Delta_{i, k} \eqdef \mu_{i} - \nicefrac{L^*}{a_{i, k}}$.
For every $i$ such that $\Delta_i$ from Definition~\ref{def:suboptimality_gaps} is finite,
we have $\Delta_{i, k} \geq \Delta_i > 0$ for all $k \in \mathcal S_i$.
For all other workers $i$, the set $\mathcal S_i$ is empty.

\textbf{Step 2.} We then fix a worker $i$ and define
$\delta_1, \dots, \delta_J$ as the increasing sequence of distinct values
in $\left\{\Delta_{i, k}\right\}_{k \in \mathcal S_i}$.
Additionally, we set $\delta_0 = 0$.
By Lemma~\ref{lemma:weighted_gap_summation}, we have
$\sum_{k \in \mathcal S_i} a_{i, k} \Delta_{i, k}
= \sum_{j=1}^J \left(\delta_j-\delta_{j-1}\right)
\sum_{k\in \mathcal S_i:\,\Delta_{i, k}\geq\delta_j} a_{i, k}$.

\textbf{Step 3.}
By Lemmas~\ref{lemma:confidence_radius_bound} and~\ref{lemma:stable_samples},
we bound $\sum_{k\in \mathcal S_i:\,\Delta_{i, k}\geq\delta_j} a_{i, k}$ by
the number of samples collected until the confidence radius becomes
smaller than $\delta_j$. This number of samples is
$\mathcal O\left(\nicefrac{\tilde \alpha^2 \beta_K}{\delta_j^2}\right)$; therefore,
\begin{align*}
&\textstyle \sum_{k \in \mathcal S_i} a_{i, k} \Delta_{i, k}
   \leq \mathcal O\left(
   \tilde \alpha^2 \beta_K
   \sum_{j=1}^J \frac{\delta_j-\delta_{j-1}}{\delta_j^2}\right) \\
&\textstyle \leq \mathcal O\left(\frac{\tilde \alpha^2 \beta_K}{\delta_1}\right)
   \leq \mathcal O\left(\frac{\tilde \alpha^2 \beta_K}{\Delta_i}\right).
\end{align*}
Summing over the workers establishes the bound
$\mathcal O\left(\sum_{i = 1}^n \nicefrac{\tilde \alpha^2 \beta_K}{\Delta_i}\right)$.
\begin{remark}
In Appendix~\ref{sec:knapsack}, we present this proof idea in a slightly more general form,
which may be applied to other objectives.
We illustrate this by analyzing a UCB-type algorithm for the knapsack problem.
\end{remark}

\subsection{Better bound on the total expected completion time}\label{sec:ata_large_batch}
In this section, we present a new theorem that refines the connection between the expected regret $\regretdyn$ and the total expected completion time. 
\begin{theorem}\label{thm:stationary_prox_connection}
Suppose that Assumption~\ref{ass:stationary-exp} is satisfied and $n \geq 2$. Then,
\begin{align*}
C_K &\leq \left(1 + 4 \varphi \ln \min\left\{n, B\right\}\right)
   \left(C_K^* + \regretdyn\right),
\end{align*}
where
$\varphi = \max_{i \in [n]} \left\{\min\left\{\frac{\alpha_{i}}{\mu_{i}},
   \frac{\alpha_{i}}{\sqrt{\mu_{i}}}\sqrt{\frac{K}{\left(K C^* + \regretdyn\right)}}\right\}\right\}.$
\end{theorem}
Together, Theorems~\ref{thm:stationary_prox_connection} and~\ref{thm:ata_regret} provide the full guarantee, including a bound on $C_K$.

Compare this with the corresponding bound from~\citep{maranjyan2025ata}:
\begin{align*}
C_K
&\leq \left(1 + 4 \eta \ln B\right) \cdot \left(C_K^* + \regretdyn\right), \text{where }
\eta = \max_{i \in [n]} \frac{\alpha_{i}}{\mu_{i}}.
\end{align*}
In the previous section, we compared the bounds on $\regretdyn$; here we highlight
the difference in the coefficients $\eta$ and $\varphi$.
First, from the definition of $\varphi$ it is clear that $\varphi \leq \eta$. 
To simplify further comparison, notice that
\begin{align}\label{eq:lksflkjdsnflkjsn}
\textstyle \frac{1}{\sqrt{C^* + \nicefrac{\regretdyn}{K}}}
\leq \frac{1}{\sqrt{\ell(\ba^*)}}
\leq \frac{1}{\sqrt{\ell(\bar \ba)}},
\end{align}
which implies
\begin{align}
\textstyle \varphi \leq \max_{i \in [n]} 
   \left\{\min\left\{\frac{\alpha_i}{\mu_{i}},
   \frac{\alpha_i}{\sqrt{\mu_{i} \ell(\bar \ba)}}\right\}\right\}.
\end{align}
Note that it is important that in the definition of $\varphi$ we take the minimum first.
Indeed, the minimum of two bounds can be small for each worker
while both quantities $\max_{i \in [n]} \nicefrac{\alpha_{i}}{\mu_{i}}$
and $\max_{i \in [n]} \nicefrac{\alpha_{i}}{\sqrt{\mu_{i}\ell(\bar \ba)}}$ can be huge. 
We then show why the difference between $\eta$ and $\varphi$ may be crucial in settings where 
the batch size $B$ is large.
\begin{proposition}\label{prop:large_batch_phi}
Suppose that Assumption~\ref{ass:stationary-exp} holds and
$B \geq \left(\max_{i\in[n]}\mu_i\right)\sum_{i=1}^n\mu_i^{-1}$. Then,
\begin{align*}
\textstyle \varphi \leq \frac{1}{\sqrt{B}} \left(\sum\limits_{i=1}^n\frac{1}{\mu_i}\right)^{1/2}
   \max\limits_{i\in[n]}\frac{\alpha_i}{\sqrt{\mu_i}}
\leq \eta.
\end{align*}
In particular, for fixed $n$ and worker distributions, $\varphi=\mathcal O(B^{-1/2})$ as $B\to\infty$,
while $\eta$ remains constant.
\end{proposition}
\begin{proof}
Since $\ell(\bar \ba)$ is the first moment
by which workers with fixed times have computed at least $B$ tasks, we get
$\textstyle \sum\limits_{i = 1}^n \frac{\ell(\bar \ba)}{\mu_i}
\geq \sum\limits_{i = 1}^n \tfloor{\frac{\ell(\bar \ba)}{\mu_i}} \geq B$ and 
$\textstyle \ell(\bar \ba) \geq B \left(\sum_{i = 1}^n \frac{1}{\mu_i}\right)^{\!-1}
\geq \max\limits_{i\in[n]}\mu_i$.
Combining this with~(\ref{eq:lksflkjdsnflkjsn}) proves the claim. 
\end{proof}

\section{Non-Stationary Environments}\label{sec:sw_ata}
In this section, we provide an analysis of SW-ATA in a non-stationary environment. 
To control non-stationarity, we introduce the following quantities.
\begin{definition}[Variation budget]\label{ass:variation_budget}
Let $V_K \in \R$ be such that $\sum_{k=1}^{K-1}\|\bmu_{k+1}-\bmu_k\|_\infty\leq V_K$.
Additionally, define $M_K \eqdef \max_{k \in [K], i \in [n]} \mu_{i,k}$ and $m_K \eqdef \min_{k \in [K], i \in [n]} \mu_{i,k}$.
\end{definition}
We focus on the nontrivial regime $V_K = \Theta\left(K^{\gamma}\right)$, where $-\nicefrac{1}{2} < \gamma < 1$.
Our bounds show that when $V_K \leq \mathcal O\left(\nicefrac{1}{\sqrt{K}}\right)$ (small variation), we can use a stationary algorithm. When $V_K \geq \tilde \Omega\left(K\right)$, the regret is $\mathcal O(K)$. This behavior is expected because the rate $\widetilde{\mathcal O}(V_K^{1/3}K^{2/3})$ is common in the bandit literature, and algorithms cannot learn from observations that quickly become irrelevant.

\subsection{SW-ATA guarantees}
In this section, all window lengths are rounded up to the nearest integer.

\begin{theorem}[Regret of SW-ATA]
Let Assumption~\ref{ass:drift-exp} be satisfied, let $\beta_K := 4\log(2BK^3)$, and suppose that
$\frac{\alpha}{B} \sqrt{\frac{\beta_K}{K}\max\left\{nB,B+n\right\}} \leq V_K \leq \frac{\alpha K}{B}$.
With $W=\max\left\{1,\sqrt{\nicefrac{K\alpha}{BV_K}}\right\}$,
SW-ATA with parameter $\alpha\geq\alpha_{\max}$ satisfies the gap-dependent bound
\begin{align*}
\textstyle \regretdyn
&\textstyle\leq \widetilde{\mathcal O}\Big(
   \sqrt{\alpha BV_KK}
   \Big(
   \sum_{i = 1}^n\frac{\alpha}{\Delta_i^{\min}}
   + n\left(\log\frac{M_K}{\alpha}\right)_+ \\
&\textstyle\quad + \frac{nBM_K}{\alpha}
   \Big)\Big).
\end{align*}
With $W=\max\left\{1,
\left(\nicefrac{K^2n\alpha^2\beta_K}{BV_K^2}\right)^{\frac{1}{3}}\right\}$,
the gap-independent regret satisfies
\begin{align*}
\textstyle \regretdyn
&\textstyle \leq \widetilde{\mathcal O}\Big(
   \left(\alpha^2nB^2\right)^{\frac{1}{3}}
   K^{\frac{2}{3}}V_K^{\frac{1}{3}} \\
&\textstyle\quad+ \left(\alpha n^2B\right)^{\frac{1}{3}}
   K^{\frac{1}{3}}V_K^{\frac{2}{3}} \cdot \frac{BM_K}{\alpha}\Big).
\end{align*}
Alternatively, choosing $W=\max\left\{1,
\left(\nicefrac{K^2(B+n)\alpha^2\beta_K}{B^2V_K^2}\right)^{\frac{1}{3}}\right\}$
yields
\begin{align*}
\textstyle\regretdyn
&\textstyle\leq \widetilde{\mathcal O}\Big(
   \left(\alpha^2B(B+n)\right)^{\frac{1}{3}}
   K^{\frac{2}{3}}V_K^{\frac{1}{3}} \\
&\textstyle\quad + \left(\frac{\alpha n^3B^2}{B+n}\right)^{\frac{1}{3}}
   K^{\frac{1}{3}}V_K^{\frac{2}{3}}
   \left(\frac{\alpha}{m_K}+\frac{BM_K}{\alpha}\right)\Big).
\end{align*}
\end{theorem}
For the setting where $\alpha = 1$, $B = 1$, and $M_K = \mathcal O\left(1\right)$, omitting constant factors,
the asymptotically leading term is $\widetilde{\mathcal O}((nV_K)^{1/3}K^{2/3})$ as in \citet{besbes2014stochastic}.
\begin{proofsketch}
The approach is similar to that in the stationary setup, except for several details.
Our proof in the stationary setup relies on the inequality
$a_{i,k}\mu_{i} - L^* \leq 2 a_{i,k} \conf_W(i,k)$, which connects the regret with the confidence radius. 
However, in a non-stationary environment, the analogue of this bound is the inequality
$a_{i,k}\mu_{i,k} - L_k^* \leq 2 a_{i,k} \conf_W(i,k) + 2 B\rho_k$, where 
$\rho_k := \sum_{q\in\cW_k} \|\bmu_{q+1}-\bmu_q\|_\infty$ is the variation inside the current window (local variation). 
This additional term is expected, since the algorithm may make a mistake for two reasons:
it estimates the expectation poorly (a small number of observations), or there is high local variation,
which implies that the estimated quantity is very loosely related to the parameters of the current round. 
So, we split rounds into two groups: rounds with small and high variation. In the first group, $\rho$ is small and 
we proceed as in the stationary setup. To estimate the regret on the other rounds, we use the fact that the total variation 
is bounded by $V_K$. See the complete proof in Section~\ref{sec:main_regret_analysys}.
\end{proofsketch}

\subsection{Proxy-loss connection}
Below, we provide a generalization of Theorem~\ref{thm:stationary_prox_connection}
that connects the expected completion time achieved by the algorithm with $C_K^*$
via the expected regret $\regretdyn$. 
\begin{restatable}{theorem}{PROXCONNECTION}
\label{thm:oRtZJgVnZogTrbAAt}
Suppose that Assumption~\ref{ass:drift-exp} is satisfied and $n \geq 2$. Then,
\begin{align*}
C_K
\leq \left(1 + 4 \varphi \ln \min\left\{n, B\right\}\right)
   \left(C_K^* + \regretdyn\right),
\end{align*}
where $\varphi = \max\limits_{i \in [n], k \in [K]} 
   \left\{\min\left\{\frac{\alpha_{i, k}}{\mu_{i, k}},
   \frac{\alpha_{i, k}}{\sqrt{\mu_{i, k}}} \cdot \sqrt{\frac{K}{C_K^* + \regretdyn}}\right\}\right\}.$
\end{restatable}
We obtain the same conclusion as in Section~\ref{sec:ata_large_batch}: the quantity $\varphi$ tends to zero as $B$ grows.
\section{Experiments}
\begin{figure}[t]
\centering
\includegraphics[width=0.75\linewidth]{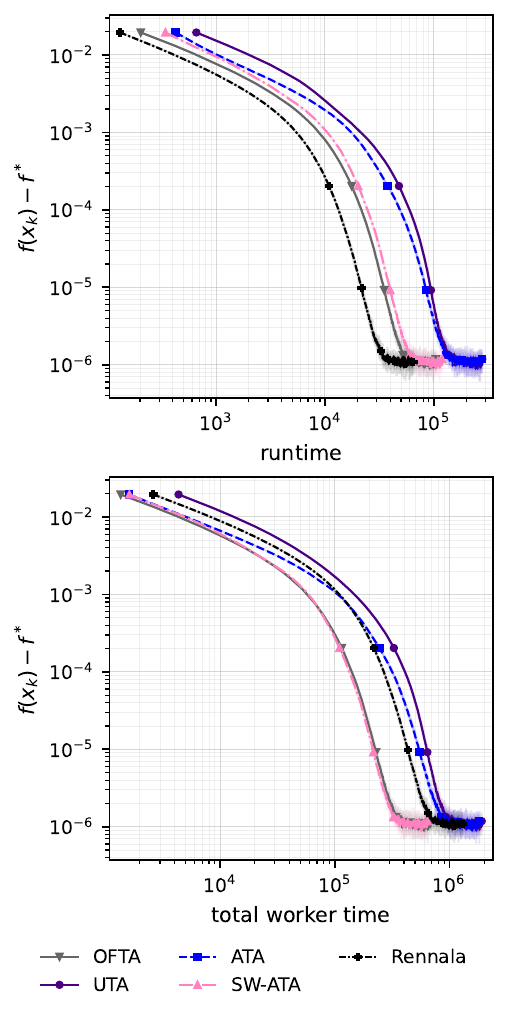}
\caption{The drifting instance at $n = 20$: the suboptimality $f(x_k) - f^*$
against the runtime $C_k = \sum_{q=1}^k C_q(\ba_q)$ \textbf{(top)} and against the
\textit{total worker time} $\sum_{i=1}^n T_{i,k}$ of \eqref{eq:paper_worker_time} \textbf{(bottom)}. Median over $15$ trajectories with a $95\%$
confidence interval. \algname{SW-ATA} closely follows the OFTA,
\algname{ATA} and the uniform baseline lag behind, and \algname{Rennala} is the fastest in runtime at visibly higher worker time.}
\label{fig:paper_sgd_drift_n20_mini}
\end{figure}
\label{sec:experiments}

We complement the theoretical guarantees with a numerical study on synthetic
instances. In this section we validate the proposed method on a distributed
SGD problem by simulating a scenario with $n$ workers with drifting delays.
The parameter tuning, the
remaining pool sizes, and further experimental details are deferred to
\Cref{app:paper_extra}.

\paragraph{Optimization task.}
We consider distributed stochastic optimization: at every step, the
workers compute a batch of stochastic gradients of the objective in
parallel, and the iterate is updated once with their average:
$x_{k+1} = x_k - \gamma\,\bar g_k$, where $\bar g_k$ is the average of the
$B$ gradients of round $k$ and $\gamma>0$ is the step size. The
objective is the convex quadratic
$f(x) = \tfrac12 x^\top A x - b^\top x$ with
$A = \tfrac14\,\mathrm{tridiag}(-1,2,-1)$, $b = \tfrac14(-1,0,\dots,0)$
and dimension $d = 50$, and the
gradient noise is isotropic Gaussian with
$\Exp{\|g - \nabla f\|^2} = \sigma^2$, $\sigma = 0.01$. In experiments we
use the constant step size $\gamma = 1$, which lies well inside the
stability region, since $\lambda_{\max}(A) < 1$ gives the threshold
$2/\lambda_{\max}(A) > 2$.

\paragraph{Environment.}
The mean delay of worker $i$ at iteration $k$ follows the slowing sinusoid
$\mu_{i,k} = 5\sin\bigl(2\pi\sqrt{k/P_0} + \phi_i\bigr) + 6$
with $P_0 = 1600$ and the phases $\phi_i = 2\pi(i-1)/n$ spread over the
pool, so the oscillation slows down over the horizon, and the set of fast
workers rotates. The resulting variation
budget grows as $V_K = \Theta(\sqrt K)$: its measured value at $K = 10^4$
is $V_K \approx 77$ at every pool size used below, with growth exponent
$\approx 0.5$. A task of worker
$i$ at iteration $k$ takes $\mu_{i,k}/2 + \mathrm{Exp}(\mu_{i,k}/2)$ time
units; that is, the delays have mean $\mu_{i,k}$ and standard deviation
$\mu_{i,k}/2$.

\paragraph{Parameters.}
At each iteration, the allocation algorithm distributes a batch of $B = 32$
stochastic gradients at the current point over the $n$ workers.
The horizon is $K = 10^{4}$ iterations, and the pools are $n = 10, 20, 40$.
Every reported curve is the median
over $15$ trajectories, with a $95\%$ confidence interval.

We run \algname{SW-ATA} with the window prescribed by
\cref{thm:general_ns_gap_independent}, evaluated at the reference budget
$V_K = \sqrt K$ as described in \Cref{app:paper_param_selection}. According to this selection, we use the window $W = 59$ at $n = 20$,
and $W = 47$ and $W = 74$ at $n = 10$ and $n = 40$.
The competitors are \algname{OFTA} (Optimal Fixed Task
Allocation), which plays the exact optimal
allocation of the current means, the \algname{UTA} (Uniform Task
Allocation),
\algname{ATA} of \citet{maranjyan2025ata}, and the queue
mechanics of \citet{tyurin2023optimal} (\algname{Rennala}).
The uniform baseline \algname{UTA} spreads the batch uniformly over the whole pool of workers.
In \algname{Rennala} each worker starts processing its own queue of tasks, and the round ends as soon as the first $B$ tasks are
completed.

\Cref{fig:paper_sgd_drift_n20_mini} summarizes the experimental results
for $n=20$: it plots the suboptimality $f(x_k)-f^*$, where $f^*$ is the
optimal value of the objective, against the runtime
$C_k := \sum_{q=1}^k C_q(\ba_q)$, whose expectation at $k = K$ is the total
expected completion time $C_K$ of \Cref{sec:setup}, and against the total
worker time $\sum_{i=1}^n T_{i,k}$, where $T_{i,k}$ is the total time
worked by worker $i$ up to iteration $k$. For the simplex methods,
$T_{i,k}$ is the sum of the realized delays of the assigned tasks, and
for \algname{Rennala} it equals $C_k$, since every worker stays busy for
the whole round.
The proposed \algname{SW-ATA} significantly outperforms \algname{ATA} and \algname{UTA}, staying close to the oracle \algname{OFTA} in runtime. \algname{Rennala} converges fastest, but uses more worker time than \algname{OFTA} and \algname{SW-ATA}; \algname{SW-ATA} is the most economical learning method. The average regret and the average iteration time, together with the results at the remaining pool sizes, are deferred to \Cref{app:paper_extra:drift}.

\section{Conclusion and Future Work}
In this paper, we consider the stationary and non-stationary task allocation problems.
We improve the previous analysis under Assumption~\ref{ass:stationary-exp} in two important ways:
first, by refining the proxy-loss regret bound (Section~\ref{sec:refined_ata_bound}), and second, by strengthening
the connection between the proxy loss and the true completion time (Section~\ref{sec:ata_large_batch}).
Introducing SW-ATA, we extend ATA to non-stationary environments (Assumption~\ref{ass:drift-exp}), which better reflect practical systems. Our theory shows that SW-ATA achieves sublinear proxy-loss regret, while, similarly to the stationary case,
the connection between the proxy loss and the true completion time becomes tighter as the batch size grows
(Section~\ref{sec:sw_ata}).
Our experiments further show that SW-ATA performs comparably to the asynchronous baseline while using
noticeably less computation (Figure~\ref{fig:paper_sgd_drift_n20_mini}).
Several directions remain for future work. In particular, our upper bounds and our new lower bounds~(Section~\ref{sec:lower_bounds}) do not match
in their dependence on $n$ and $B$. Designing a new method that matches the lower bound, or developing a tighter lower-bound construction,
remains an open question.

\bibliographystyle{plainnat}
\bibliography{sample_paper}

\clearpage
\appendix

\thispagestyle{empty}

\onecolumn
\newpage
\tableofcontents
\newpage

\section{Notation}
\label{sec:notation}
Table~\ref{tab:notation} summarizes the general mathematical notation used throughout the paper.

\begingroup
\small
\renewcommand{\arraystretch}{1.2}
\begin{longtable}{@{}p{0.25\textwidth}p{0.71\textwidth}@{}}
\caption{General mathematical notation.}
\label{tab:notation}\\
\toprule
\textbf{Symbol} & \textbf{Meaning} \\
\midrule
\endfirsthead
\multicolumn{2}{l}{\tablename~\thetable{} (continued)}\\
\toprule
\textbf{Symbol} & \textbf{Meaning} \\
\midrule
\endhead
\midrule
\multicolumn{2}{r}{Continued on next page}\\
\endfoot
\bottomrule
\endlastfoot
$[m]$ & The set $\{1,\ldots,m\}$ for a positive integer $m$; in particular, $[K]=\{1,\ldots,K\}$. \\
$\mathbb R$, $\mathbb Z$ & The sets of real numbers and integers, respectively. \\
$\mathbb R^n$ & Real vectors with $n$ coordinates. \\
$\mathbb Z_{\geq0}^n$ & Vectors of $n$ nonnegative integers, including zero coordinates. \\
$\boldsymbol{x}$, $x_i$ & A vector and its $i$th coordinate. \\
$\supp(\boldsymbol{x})$ & Support of a vector: $\{i:x_i\neq0\}$; for nonnegative vectors, equivalently $\{i:x_i>0\}$. \\
$\|\boldsymbol{x}\|_\infty$ & Maximum norm: $\max_i|x_i|$. \\
$\lfloor x\rfloor$, $\lceil x\rceil$ & Largest integer not exceeding $x$ and smallest integer not less than $x$, respectively. \\
$(x)_+$ & Positive part: $\max\{0,x\}$. We use $(\log x)_+=\log\max\{1,x\}$ for $x\geq0$. \\ 
$\mathbf{1}\{A\}$ & Indicator: $1$ when event or condition $A$ holds, and $0$ otherwise. \\
$\mathbb P(A)$ & Probability of event $A$. \\
$\mathbb E[X]$ & Expectation of a random variable $X$. \\
$f\lesssim g$ & Inequality $f\leq c g$ for a positive universal constant $c$. \\
$\mathcal O(g)$, $\Omega(g)$ & Asymptotic upper and lower bounds up to positive constant factors. \\
$\Theta(g)$ & Matching asymptotic upper and lower bounds up to positive constant factors. \\
$\widetilde{\mathcal O}(g)$ & An asymptotic upper bound that additionally suppresses logarithmic factors. \\
\end{longtable}
\endgroup

\section{Technical Lemmas}
Our proof relies on the following lemma:
\GAPSUMMATION*

\begin{proof}
We define a decreasing sequence $\tilde \delta_j := \delta_{J + 1 - j}$. Then, regrouping terms and applying Abel's transformation, we get
\begin{align*}
\sum_{k=1}^m a_k \Delta_k
= \sum_{j=1}^J \tilde \delta_j \sum_{k:\, \Delta_k = \tilde \delta_j} a_k
&= \tilde \delta_J \sum_{k:\, \Delta_k \geq \tilde \delta_J} a_k
   + \sum_{j=1}^{J - 1} \left(\tilde \delta_j - \tilde\delta_{j+1}\right) \sum_{k:\, \Delta_k \geq \tilde \delta_j} a_k \\
&= \delta_1 \sum_{k:\, \Delta_k \geq \delta_1} a_k
   + \sum_{j=2}^{J} \left(\delta_j - \delta_{j-1}\right) \sum_{k:\, \Delta_k \geq \delta_j} a_k.
\end{align*}
\end{proof}

\begin{lemma}\label{lemma:sub_exp_technical}
Let $0 < \delta_1, \dots, \delta_J$ ($J \geq 1$) be an increasing sequence of distinct values, and additionally
set $\delta_0 = 0$. Then,
\begin{align*}
\sum_{j=1}^J \frac{\delta_j-\delta_{j-1}}{\delta_j} \leq 1 + \log \frac{\delta_{J}}{\delta_1}.
\end{align*}
\end{lemma}
\begin{proof}
For $J = 1$, the inequality is trivial. Assume $J \geq 2$.
Since the sequence $\left\{\delta_j\right\}$ is increasing, we have
\begin{align*}
\int_{\delta_{j-1}}^{\delta_{j}} \frac{dx}{x}
\geq \left(\min_{x \in \left[\delta_{j-1}, \delta_{j}\right]} \frac{1}{x}\right)
   \cdot \left(\delta_j-\delta_{j-1}\right)
= \frac{\delta_j-\delta_{j-1}}{\delta_j},
\quad \text{for all} \quad j \in \{2, \dots, J\}.
\end{align*}
This implies 
\begin{align*}
\sum_{j=1}^J \frac{\delta_j-\delta_{j-1}}{\delta_j}
= 1 + \sum_{j=2}^J \frac{\delta_j-\delta_{j-1}}{\delta_j}
\leq 1 + \int_{\delta_{1}}^{\delta_{J}} \frac{dx}{x}
= 1 + \log \frac{\delta_{J}}{\delta_1}.
\end{align*}
\end{proof}

\begin{lemma}\label{lemma:number_of_tasks}
Let $\left\{\mu_i\right\}_{i = 1}^n$ be a sequence of positive values. 
Assume $0 < B \in \mathbb N$ and define
 $t^* = \min_{t > 0} \left\{t: \sum_{i = 1}^n \tfloor{\frac{t}{\mu_i}} \geq B\right\}$. 
Then, 
\begin{align*}
\sum_{i = 1}^n \frac{t^*}{\mu_i} \leq B + n - 1.
\end{align*}
\end{lemma}
\begin{proof}
Define a non-decreasing function $f(t) = \sum_{i = 1}^n \tfloor{\frac{t}{\mu_i}}$ and a point
$t' \in \argmax_{t \in (0, t^*)} f(t)$.
Such a $t'$ exists since $f(t)$ is piecewise constant and it has a finite number of discontinuity points
on $(0, t^*)$.
Note that $\frac{t^*}{\mu_i} \leq \tfloor{\frac{t'}{\mu_i}} + 1$
for all $i \in [n]$.
Indeed, if $\frac{t^*}{\mu_i} >\tfloor{\frac{t'}{\mu_i}} + 1$, then 
there exists $t'' = \mu_i\left(\tfloor{\frac{t'}{\mu_i}} + 1\right)$ 
such that $f(t') < f(t'') < f(t^*)$. Since $f(t)$ is non-decreasing, we have $t'' < t^*$. 
Combining these inequalities, we obtain a contradiction with the definition of $t'$. Thus,
\begin{align*}
\sum_{i = 1}^n \frac{t^*}{\mu_i} \leq \sum_{i = 1}^n \tfloor{\frac{t'}{\mu_i}} + n 
\leq B + n - 1.
\end{align*}
The last inequality follows since 
$f(t') = \sum_{i = 1}^n \tfloor{\frac{t'}{\mu_i}} < B$.
\end{proof}

The following proposition states that there generally exists at least one $i$ such that $\Delta_i^{\min} \neq +\infty$.
\begin{restatable}{proposition}{INFDELTAS}\label{proposition:inf_deltas}
   If $n \geq 2$ and $\Delta_i^{\min} = +\infty$ for all $i \in [n]$, then $B = 1$, and $\mu_{1, k} = \dots = \mu_{n, k}$ for all $k$ in $[K]$.
\end{restatable}
\begin{proof}
If $\Delta_i^{\min} = +\infty$, then $b_{i,k} = B + 1$ for every $k \in [K]$.
Hence, $B\mu_{i,k} \leq L_k^*$. On the other hand, assigning all $B$ tasks to
worker $i$ is feasible, and therefore $L_k^* \leq B\mu_{i,k}$. Thus, $L_k^* = B\mu_{i,k}$
for all $i \in [n]$ and $k \in [K]$, which implies
$\mu_{1,k} = \dots = \mu_{n,k}$ for every $k \in [K]$.

Suppose now that $B \geq 2$. Since $n \geq 2$, assigning $B-1$ tasks to one
worker and one task to another yields proxy loss $(B-1)\mu_{i,k} < B\mu_{i,k}
= L_k^*$, which contradicts the optimality of $L_k^*$. Hence, $B=1$.
\end{proof}
\section{Regret Analysis}\label{sec:main_regret_analysys}
In this section we prove a technical bound from which gap-dependent and gap-independent bounds for both stationary 
and non-stationary settings follow. First, we introduce useful quantities.
\subsection{Auxiliary Definitions}
For $N_{i,k}^{(W)}\geq1$, define the window-average mean
$\bar\mu_{i,k}^{(W)} = \frac{1}{N_{i,k}^{(W)}} \sum_{q\in\cW_k}a_{i,q}\mu_{i,q}$.
SW-ATA estimates $\bar\mu_{i,k}^{(W)}$ by computing $\hat\mu_{i,k}^{(W)}$.
We then define the bad events
\begin{align}\label{eq:bad_events}
\mathcal E_{i,k} =
\left\{
N_{i,k}^{(W)}\geq1,\quad
\left|\hat\mu_{i,k}^{(W)}-\bar\mu_{i,k}^{(W)}\right|
> \conf_W(i,k)
\right\},
\qquad \mathcal E_k=\bigcup_{i=1}^n\mathcal E_{i,k}.
\end{align}
The bad event $\mathcal E_k$ implies that on the $k$th iteration we were "unlucky" and the constructed confidence intervals 
do not contain $\bar\mu_{i,k}^{(W)}$ for at least one $i \in [n]$. In Section~\ref{sec:confidence_interval}
we show that our choice of $\conf_W(i, k)$ ensures that the expected number of such rounds is less than $2n$.

Additionally, we define the local variation $\rho_k=\sum_{q\in\cW_k} \|\bmu_{q+1}-\bmu_q\|_\infty$.
Intuitively, if on round $k$ the local variation is high, then the allocation received from the algorithm 
may be bad even if we estimate the window-average mean accurately, since $\bar\mu_{i,k}^{(W)}$ may not be close to 
the actual mean of the $i$th worker in the $k$th round.
\begin{lemma}[Window variation]\label{lemma:window-var}
Let the variation budget in Definition~\ref{ass:variation_budget} be satisfied. Then, $\sum_{k=1}^K \rho_k\leq W V_K$.
\end{lemma}
\begin{proof}
For any $t$, the term $\|\bmu_{t+1}-\bmu_t\|_\infty$ appears in at most $W$ windows. Hence,
\begin{align*}
\sum_{k=1}^K \rho_k \leq W\sum_{t=1}^{K-1}\|\bmu_{t+1}-\bmu_t\|_\infty \leq W V_K.
\end{align*}
\end{proof}

Assume that in the $k$th iteration, the good event holds, 
but the algorithm makes a mistake. The following lemma connects
the regret caused by a worker to the confidence radius and local variation. 

\begin{lemma}\label{lemma:confidence_radius_bound}
Let Assumption~\ref{ass:drift-exp} and the variation budget in Definition~\ref{ass:variation_budget} be satisfied.
Consider the allocations returned by SW-ATA. 
If the round is good ($\neg \mathcal E_{k}$ holds) and $N_{i,k}^{(W)} > 0$, then
\begin{align*}
a_{i,k}\mu_{i,k} - L_k^* \leq 2 a_{i,k} \conf_W(i,k) + 2 B\rho_k.
\end{align*}
\end{lemma}
\begin{proof}
On the event $\neg\mathcal E_{i,k}$, we have
\begin{align}\label{eq:oiwureoiwnnds}
a_{i,k}\bigl(\bar\mu_{i,k}^{(W)}-2\conf_W(i,k)\bigr)
\leq a_{i,k}s_{i,k}.
\end{align}
We also show that $s_{j,k} \leq \mu_{j,k} + \rho_k$ for all $j$.
Consider two cases. If $N_{j,k}^{(W)} > 0$, then, by Corollary~\ref{cor:confidence_interval},
we have $s_{j,k} \leq \bar\mu_{j,k}^{(W)} \leq \mu_{j,k} + \rho_k$.
Otherwise, we have $N_{j,k}^{(W)} = 0$, which implies $s_{j,k} = 0 \leq \mu_{j,k} + \rho_k$.
Then,
\begin{align*}
a_{i,k}s_{i,k}
\leq \max_{j\in[n]} a_{j,k} s_{j,k}
\leq \max_{j\in[n]} \bar a_{j,k} s_{j,k}
\leq \max_{j\in[n]} \bar a_{j,k} \left(\mu_{j,k} + \rho_k\right).
\end{align*}
Here, the second inequality holds because $\ba_{k}$ is an optimal allocation returned by RAS
(or the uniform allocation among workers with $s_{j,k} = 0$, which is also optimal).
Therefore, we obtain
\begin{align}\label{eq:iwfnmdskmfd}
a_{i,k}s_{i,k}
\leq
\max_{j\in[n]} \bar a_{j,k}\mu_{j,k}
+
\max_{j\in[n]} \bar a_{j,k}\rho_k
\leq
L_k^*+B\rho_k.
\end{align}
Combining inequalities~(\ref{eq:oiwureoiwnnds}) and~(\ref{eq:iwfnmdskmfd}) with
$\bar\mu_{i,k}^{(W)} \geq \mu_{i,k}-\rho_k$, we get
\begin{align*}
a_{i,k}\left(\mu_{i,k}-\rho_k-2\conf_W(i,k)\right)
\leq a_{i,k}\left(\bar\mu_{i,k}^{(W)}-2\conf_W(i,k)\right) &\leq L_k^*+B\rho_k.
\end{align*}
Using $a_{i,k} \leq B$, we obtain
\begin{align*}
a_{i,k}\mu_{i,k} - L_k^* - 2 B\rho_k \leq 2 a_{i,k} \conf_W(i,k).
\end{align*}
\end{proof}
Let us fix a trajectory.
To consider rounds in which a worker may be overallocated because of high local variation and 
insufficient information (a large confidence interval), we split rounds as follows:
\begin{align}\label{eq:stable_definition}
\text{stable rounds} \quad
   &\mathcal S = \left\{k:\, \neg \mathcal E_k \text{ holds and }
   \ell_k(\ba_k) > L_k^* \text{ and }
   2 B\rho_k < \varepsilon \left(\ell_k(\ba_k) - L_k^*\right)\right\} \\
\text{and unstable rounds} \quad
   &\mathcal U = \left\{k:\,\neg \mathcal E_k \text{ holds and }
   \ell_k(\ba_k) > L_k^* \text{ and }
   2 B\rho_k \geq \varepsilon \left(\ell_k(\ba_k) - L_k^*\right)\right\}. \nonumber
\end{align}
\subsection{Proofs}
We now explain the proof idea.
Note that each mistaken round is either bad, stable, or unstable.
We bound the regret on unstable rounds using Lemma~\ref{lemma:window-var}.
To bound the regret on stable rounds, we split the horizon into consecutive blocks of length $W$, where
the last block may contain fewer elements.
We fix a worker $i$ and a block $q$, denote the last stable round in the block by $k'$, and
use the fact that $N_{i,k'}^{(W)}$ accumulates all corresponding observations from the preceding rounds in this block.
Using Lemmas~\ref{lemma:weighted_gap_summation},~\ref{lemma:confidence_radius_bound},
and~\ref{lemma:stable_samples},
we control the regret in each block. Multiplying by $\tceil{\frac{K}{W}}$ bounds the total regret caused by stable rounds.
Finally, we bound the expected number of bad rounds using 
Lemma~\ref{lemma:sufficient-beta} and charge them with the maximum possible regret.

\begin{theorem}\label{thm:master_bound}
Let Assumption~\ref{ass:drift-exp} and the variation budget in Definition~\ref{ass:variation_budget} be satisfied.
Then, the following upper bound on the expected regret of SW-ATA with window size $W \in [K]$ holds
for every $\varepsilon \in \left(0, 1\right)$ and $\theta > 0$:
\begin{align}\label{eq:ksndhfiuwehfjdk}
\regretdyn
&\leq 2 n \max_{i, k} \left\{B\mu_{i,k} - L_k^*\right\} + \frac{2 B V_K W}{\varepsilon} + K \theta \\
& \quad + \frac{16 \alpha^2 \beta_K}{(1 - \varepsilon)^2}
   \min\left\{\tceil{\frac{K}{W}} \sum \limits_{i = 1}^{n} \frac{1}{\Delta_i^{\min}}, 
      \tceil{\frac{K}{W}} \frac{n B}{\theta},
      \sum \limits_{q = 1}^{\tceil{\frac{K}{W}}} \left(1 + \varepsilon + \frac{L_{\tau_q}^*}{\theta}\right) 
         \sum \limits_{i = 1}^{n} \frac{1}{\mu_{i, \tau_q}}\right\} \nonumber \\
&\quad + \tceil{\frac{K}{W}} \frac{4\alpha \beta_K}{1 - \varepsilon} \sum_{i = 1}^n
   \left(\log \frac{\max_{k \in [K]} \left\{B\mu_{i,k} - L_k^*\right\}}{\alpha B}\right)_+ \nonumber \\
&\quad + 2 \tceil{\frac{K}{W}} \sum_{i = 1}^n \max_{k \in [K]} \left\{B \mu_{i,k} - L_k^*\right\}, \nonumber 
\end{align}
where $\beta_K = 4 \log(2BK^3)$, $\tau_q = (q - 1)W + 1$, and $(\log x)_+ := \log \max\left\{1, x\right\}$.
\end{theorem}

\begin{proof}
We fix a trajectory. Then, we split the rounds into the following types: bad rounds,
stable rounds, and unstable rounds; see~(\ref{eq:stable_definition}). Since in a non-mistaken round,
$\ell_k(\ba_k) - L_k^* = 0$, we have
\begin{align}
&\sum \limits_{k = 1}^K \left(\ell_k(\ba_k) - L_k^*\right)
   = \sum \limits_{k = 1}^K \Ind{\ell_k(\ba_k) > L_k^*} \left(\ell_k(\ba_k) - L_k^*\right) \nonumber \\
&\leq \sum \limits_{k = 1}^K
   \left(\Ind{\mathcal E_k} + \Ind{\neg \mathcal E_k, \text{$k$th round is stable}}
      + \Ind{\neg \mathcal E_k, \text{$k$th round is unstable}}\right)
   \left(\ell_k(\ba_k) - L_k^*\right) \nonumber \\
&= \sum \limits_{k = 1}^K \Ind{\mathcal E_k} \left(\ell_k(\ba_k) - L_k^*\right)
   + \sum \limits_{k \in \mathcal U} \left(\ell_k(\ba_k) - L_k^*\right)
   + \sum \limits_{k \in \mathcal S} \left(\ell_k(\ba_k) - L_k^*\right) \nonumber \\
&\leq \max_{i, k} \left\{B\mu_{i,k} - L_k^*\right\} \cdot \sum \limits_{k = 1}^K \Ind{\mathcal E_k}
   + \sum \limits_{k \in \mathcal U} \left(\ell_k(\ba_k) - L_k^*\right)
   + \sum \limits_{k \in \mathcal S} \left(\ell_k(\ba_k) - L_k^*\right). \label{eq:decomposition}
\end{align}

Consider the last two terms separately. By the definition of unstable rounds, we have
$\ell_k(\ba_k) - L_k^* \leq \frac{2 B\rho_k}{\varepsilon}$.
Using this inequality and Lemma~\ref{lemma:window-var}, we obtain
\begin{align}\label{eq:drift_term_bound}
\sum \limits_{k \in \mathcal U} \left(\ell_k(\ba_k) - L_k^*\right)
\leq \sum \limits_{k \in \mathcal U}\frac{2 B\rho_k}{\varepsilon}
\leq \frac{2 B V_K W}{\varepsilon}.
\end{align}

Next, we bound the regret on the stable rounds. Define a sequence of indices $i_k$, $k \in [K]$, specifying
which worker is the bottleneck. If there are several such workers, we choose one arbitrarily
(e.g., the worker with the smallest index).
\begin{align*}
\sum \limits_{k \in \mathcal S} \left(\ell_k(\ba_k) - L_k^*\right)
= \sum \limits_{k \in \mathcal S} \left(a_{i_k,k} \mu_{i_k,k} - L_k^*\right).
\end{align*}
Then, we split the rounds into two types: rounds with small regret, which means
$\ell_k(\ba_k) - L_k^* < \theta$, and rounds such that $\ell_k(\ba_k) - L_k^* \geq \theta$.
The total regret caused by rounds of the first type is less than $K \theta$.
Therefore,
\begin{align}\label{eq:oieriiojjfowoj}
\sum \limits_{k \in \mathcal S} \left(\ell_k(\ba_k) - L_k^*\right)
\leq K \theta
   + \sum \limits_{k \in \mathcal S: \ell_k(\ba_k) - L_k^* \geq \theta} a_{i_k,k}
      \left(\mu_{i_k,k} - \frac{L_k^*}{a_{i_k,k}}\right).
\end{align}
It remains to bound the last term. Denote $\Delta_{i,k}:= \mu_{i,k} - \frac{L_k^*}{a_{i,k}}$
for $a_{i,k} > 0$.
Then, we split the horizon into consecutive blocks of length $W$, where
the last block may contain fewer elements.
Define the subset of stable rounds in the $q$th block with large loss:
\begin{align*}
\mathcal{S}_i^\theta(q, W):= \left\{k \in \mathcal{S}:\, (q - 1) W + 1 \leq k \leq q W
   \text{ and } i_k = i \text{ and } \ell_k(\ba_k) - L_k^* \geq \theta\right\},
\end{align*}
where $q \in \left[\tceil{\frac{K}{W}}\right]$. Thus,
\begin{align}\label{eq:ioewirojfwo}
\sum \limits_{k \in \mathcal S:\,\ell_k(\ba_k) - L_k^* \geq \theta} a_{i_k,k} \Delta_{i_k,k}
= \sum \limits_{i = 1}^{n} \sum \limits_{q = 1}^{\tceil{\frac{K}{W}}}
   \sum \limits_{k \in \mathcal{S}_i^\theta(q, W)} a_{i,k} \Delta_{i,k}.
\end{align}

Fix a worker $i$ and a block $q$.
If $\mathcal{S}_i^\theta(q, W)$ is empty, its contribution is zero, so we only consider nonempty sets.
Let $0 < \delta_1, \dots, \delta_J$ be an increasing sequence of distinct
values of the sequence $\Delta_{i,k}$, $k \in \mathcal{S}_i^\theta(q, W)$, and additionally
set $\delta_0 = 0$.
First, assume that $\mathcal{S}_i^\theta(q, W)$ contains at least two elements.
Using Lemma~\ref{lemma:weighted_gap_summation}, we obtain
\begin{align*}
\sum \limits_{k \in \mathcal{S}_i^\theta(q, W)} a_{i,k} \Delta_{i,k}
   = \sum_{j=1}^J \left(\delta_j-\delta_{j-1}\right)
   \sum_{k \in \mathcal{S}_i^\theta(q, W):\,\Delta_{i,k}\geq\delta_j} a_{i,k}.
\end{align*}
Lemma~\ref{lemma:stable_samples} yields
\begin{align}\label{eq:iksdfkldnskfd}
\sum \limits_{k \in \mathcal{S}_i^\theta(q, W)} a_{i,k} \Delta_{i,k}
&\leq \sum_{j=1}^J \left(\delta_j-\delta_{j-1}\right)
   \left(\max\left\{\frac{4\alpha^2 \beta_K}{(1 - \varepsilon)^2\delta_j^2},
   \frac{\alpha \beta_K}{(1 - \varepsilon)\delta_j}\right\} + B\right).
\end{align}

Define the function $g(x) = 
   \max\left\{\frac{4\alpha^2 \beta_K}{(1 - \varepsilon)^2 x^2},
   \frac{\alpha \beta_K}{(1 - \varepsilon)x}, B\right\}$.
Denote the last index $j \in \left\{0, 1, \dots, J\right\}$ 
such that $\delta_j \leq \frac{4 \alpha}{1 - \varepsilon}$ as $j'$.
Since $\delta_0 = 0$, $j'$ is always defined.
Additionally, we define $\mathcal J_1 = \left\{j:\,1 \leq j \leq j'\right\}$ and
$\mathcal J_2 = \left\{j:\,j' + 1 \leq j \leq J\right\}$.
For $j \in \mathcal J_1$, the quadratic term is at least the linear term, and for $j \in \mathcal J_2$,
the linear term is larger. Bounding the contribution of $B$ by $2B\delta_J$, we obtain
\begin{align*}
\sum \limits_{k \in \mathcal{S}_i^\theta(q, W)} a_{i,k} \Delta_{i,k}
&\leq 2 \sum_{j \in \mathcal J_1} \left(\delta_j-\delta_{j-1}\right)
   \frac{4\alpha^2 \beta_K}{(1 - \varepsilon)^2\delta_j^2} \\
&\quad + 2 \sum_{j \in \mathcal J_2} \left(\delta_j-\delta_{j-1}\right)
   \frac{\alpha \beta_K}{(1 - \varepsilon)\delta_j} + 2B\delta_J.
\end{align*}
Here, we use the convention that the sum over an empty set equals zero.
Using Lemma~\ref{lemma:sub_exp_technical}, we bound the sub-exponential term when $\mathcal J_2$ is nonempty.
\begin{align}\label{eq:ksdnfsdfksdj}
\sum_{j \in \mathcal J_2} \left(\delta_j-\delta_{j-1}\right)
   \frac{\alpha \beta_K}{(1 - \varepsilon)\delta_j}
&\leq \left(1 + \log \frac{\delta_J}{\delta_{j' + 1}}\right)
   \frac{\alpha \beta_K}{1 - \varepsilon} \nonumber \\
&\leq \left(1 + \log \frac{\delta_J}{4\alpha}\right)
   \frac{\alpha \beta_K}{1 - \varepsilon} \nonumber \\
&\leq 2\left(\log \frac{\delta_J}{\alpha}\right)_+
   \frac{\alpha \beta_K}{1 - \varepsilon}.
\end{align}
The second inequality follows from $\delta_{j' + 1} > \frac{4\alpha}{1 - \varepsilon} > 4\alpha$.
For the last inequality, note that $\mathcal J_2 \neq \varnothing$ implies $\delta_J/\alpha > 4$.
We then use $1 + \log x \leq 2\log x$ for $x \geq \mathrm e$ and monotonicity of the logarithm.
If $\mathcal J_2$ is empty, its contribution is zero, and the last bound remains valid.
Then, we bound the sub-Gaussian term. First, assume that $\left|\mathcal J_1\right| \geq 2$.
\begin{align*}
\sum_{j=1}^{j'} \left(\delta_j-\delta_{j-1}\right)
   \left(\frac{4\alpha^2 \beta_K}{(1 - \varepsilon)^2\delta_j^2}\right)
&\leq \frac{4\alpha^2 \beta_K}{(1 - \varepsilon)^2} \sum_{j=2}^{j'}
   \left(\frac{1}{\delta_{j-1}} - \frac{1}{\delta_j}\right)
   + \frac{4\alpha^2 \beta_K}{(1 - \varepsilon)^2 \delta_1} \\
&\leq \frac{4\alpha^2 \beta_K}{(1 - \varepsilon)^2}
   \left(\frac{1}{\delta_1} - \frac{1}{\delta_{j'}}\right)
   + \frac{4\alpha^2 \beta_K}{(1 - \varepsilon)^2 \delta_1} \\
&\leq \frac{8\alpha^2 \beta_K}{(1 - \varepsilon)^2 \delta_1}.
\end{align*}
If $\mathcal J_1$ contains fewer than $2$ elements, then the bound is straightforward.
Combining these inequalities, we obtain
\begin{align}\label{eq:kksmklslkf}
\sum \limits_{k \in \mathcal{S}_i^\theta(q, W)} a_{i,k} \Delta_{i,k} 
&\leq 2 \sum_{j=1}^J \left(\delta_j-\delta_{j-1}\right) g(\delta_j) \nonumber \\
&\leq \frac{16 \alpha^2 \beta_K}{(1 - \varepsilon)^2\delta_1}
+ \frac{4\alpha \beta_K}{1 - \varepsilon}
   \left(\log \frac{\delta_J}{\alpha}\right)_+ + 2B\delta_J,
\end{align}
where $(x)_+ = \max\left\{0, x\right\}$.

Above, we considered the case when $\mathcal{S}_i^\theta(q, W)$ contains at least two elements.
If $\mathcal{S}_i^\theta(q, W)$ contains exactly one element, the bound follows directly:
\begin{align*}
\sum \limits_{k \in \mathcal{S}_i^\theta(q, W)} a_{i,k} \Delta_{i,k}
\leq B \delta_J.
\end{align*}
Then, we bound $\delta_1$ and $\delta_J$ globally:
\begin{align*}
\delta_J &= \max_{k \in \mathcal{S}_i^\theta(q, W)}
\left\{\mu_{i,k} - \frac{L_k^*}{a_{i,k}} \right\}
\leq \max_{k \in \mathcal{S}_i^\theta(q, W)}
\left\{\mu_{i,k} - \frac{L_k^*}{B} \right\} \\
\text{and} \quad \delta_1 &= \min_{k \in \mathcal{S}_i^\theta(q, W)} \left\{\mu_{i,k} - \frac{L_k^*}{a_{i,k}}\right\}
\geq \Delta_i^{\min} = \min_{k \in [K]:\,b_{i,k} \leq B}
\left\{\mu_{i,k} - \frac{L_k^*}{b_{i,k}}\right\}.
\end{align*}
Here, worker $i$ is the bottleneck in every round $k \in \mathcal{S}_i^\theta(q, W)$, and the regret is positive.
Thus, $a_{i,k}\mu_{i,k} > L_k^*$ implies $b_{i,k} \leq a_{i,k} \leq B$.
By the definition of $\delta_1$, we also have $\delta_1 \geq \nicefrac{\theta}{B}$.
Moreover, the bound on $\delta_J$ and monotonicity imply
\begin{align*}
\left(\log \frac{\delta_J}{\alpha}\right)_+
\leq \left(\log \frac{\max_{k \in [K]} \left\{B\mu_{i,k} - L_k^*\right\}}{\alpha B}\right)_+.
\end{align*}
Additionally, we bound 
$\left(\min_{k \in \mathcal{S}_i^\theta(q, W)} \left\{\mu_{i,k} - \frac{L_k^*}{a_{i,k}}\right\}\right)^{-1}$
using $\ell_k(\ba_k) - L_k^* \geq \theta$ and Lemma~\ref{lemma:delta_sum_bound}
to obtain a tighter instance-independent bound.
Denoting the first round in the block by $\tau_q$ and combining these bounds, we get
\begin{align}\label{eq:iwifnkskmdknf}
\frac{1}{\delta_1} &\leq \min\left\{\frac{1}{\Delta_i^{\min}}, \frac{B}{\theta},
      \left(1 + \varepsilon + \frac{L_{\tau_q}^*}{\theta}\right) \frac{1}{\mu_{i, \tau_q}}\right\}.
\end{align}
Combining~(\ref{eq:ioewirojfwo}),~(\ref{eq:kksmklslkf}), and~(\ref{eq:iwifnkskmdknf}), we get
\begin{align}\label{eq:oijewnfnksdk}
\sum \limits_{k \in \mathcal S:\,\ell_k(\ba_k) - L_k^* \geq \theta} a_{i_k,k} \Delta_{i_k,k}
&\leq \frac{16 \alpha^2 \beta_K}{(1 - \varepsilon)^2}
   \min\left\{\tceil{\frac{K}{W}} \sum \limits_{i = 1}^{n} \frac{1}{\Delta_i^{\min}}, 
      \tceil{\frac{K}{W}} \frac{n B}{\theta},
      \sum \limits_{q = 1}^{\tceil{\frac{K}{W}}} \left(1 + \varepsilon + \frac{L_{\tau_q}^*}{\theta}\right) 
         \sum \limits_{i = 1}^{n} \frac{1}{\mu_{i, \tau_q}}\right\} \nonumber \\
&\quad + \tceil{\frac{K}{W}} \frac{4\alpha \beta_K}{1 - \varepsilon} \sum_{i = 1}^n
   \left(\log \frac{\max_{k \in [K]} \left\{B\mu_{i,k} - L_k^*\right\}}{\alpha B}\right)_+ \nonumber \\
&\quad + 2 \tceil{\frac{K}{W}} \sum_{i = 1}^n \max_{k \in [K]} \left\{B \mu_{i,k} - L_k^*\right\}.
\end{align}

Combining~(\ref{eq:oieriiojjfowoj}),~(\ref{eq:drift_term_bound}),
~(\ref{eq:oijewnfnksdk}), and~(\ref{eq:decomposition}), we deduce
\begin{align*}
\sum \limits_{k = 1}^K \left(\ell_k(\ba_k) - L_k^*\right)
&\leq \max_{i, k} \left\{B\mu_{i,k} - L_k^*\right\}
\cdot \sum \limits_{k = 1}^K \Ind{\mathcal E_k}
+ \frac{2 B V_K W}{\varepsilon} + K \theta \\
& \quad + \frac{16 \alpha^2 \beta_K}{(1 - \varepsilon)^2}
   \min\left\{\tceil{\frac{K}{W}} \sum \limits_{i = 1}^{n} \frac{1}{\Delta_i^{\min}}, 
      \tceil{\frac{K}{W}} \frac{n B}{\theta},
      \sum \limits_{q = 1}^{\tceil{\frac{K}{W}}} \left(1 + \varepsilon + \frac{L_{\tau_q}^*}{\theta}\right) 
         \sum \limits_{i = 1}^{n} \frac{1}{\mu_{i, \tau_q}}\right\} \nonumber \\
&\quad + \tceil{\frac{K}{W}} \frac{4\alpha \beta_K}{1 - \varepsilon} \sum_{i = 1}^n
   \left(\log \frac{\max_{k \in [K]} \left\{B\mu_{i,k} - L_k^*\right\}}{\alpha B}\right)_+ \nonumber \\
&\quad + 2 \tceil{\frac{K}{W}} \sum_{i = 1}^n \max_{k \in [K]} \left\{B \mu_{i,k} - L_k^*\right\}.
\end{align*}

Note that all terms except $\sum \limits_{k = 1}^K \Ind{\mathcal E_k}$ are deterministic. 
By Lemma~\ref{lemma:sufficient-beta}, we have
$\Exp{\sum \limits_{k = 1}^K \Ind{\mathcal E_k}} = \sum_{k=1}^K \Prob{\mathcal E_k} \leq 2n$, which completes the proof.
\end{proof}

\begin{lemma}\label{lemma:stable_samples}
Consider $\varepsilon \in \left(0, 1\right)$ and the corresponding set of
stable rounds $\mathcal{S}_i^\theta(q, W)$ from the proof of
Theorem~\ref{thm:master_bound}, with an arbitrary $\theta > 0$.
Define the subset $\mathcal D_{\delta} = \left\{k \in \mathcal{S}_i^\theta(q, W):\,\Delta_{i,k}\geq\delta\right\}$.
Then, for every $\delta > 0$,
\begin{align*}
\sum_{k \in \mathcal D_{\delta}} a_{i,k}
\leq \max\left\{
\frac{4\alpha^2 \beta_K}{(1 - \varepsilon)^2\delta^2},
\frac{\alpha \beta_K}{(1 - \varepsilon)\delta}\right\} + B,
\end{align*}
where $\beta_K = 4 \log(2B K^3)$.
\end{lemma}
\begin{proof}
Consider two cases separately. First, if there are fewer than two rounds in
$\mathcal D_{\delta}$, then
\begin{align*}
\sum_{k \in \mathcal D_{\delta}} a_{i,k} \leq B.
\end{align*}
Next, assume that $\mathcal D_{\delta}$ contains at least two rounds.
Let $k'$ be the index of the last round in $\mathcal D_{\delta}$.
Since $\mathcal D_{\delta} \subseteq \mathcal{S}_i^\theta(q, W)$ is a subset of the $q$th block, these observations are accumulated in
$N_{i,k'}^{(W)}$, which implies
\begin{align}\label{eq:oifndkslnfk}
\sum_{k \in \mathcal D_\delta} a_{i,k}
&= \sum_{k \in \mathcal D_\delta:\,k<k'} a_{i,k}+a_{i,k'}
\leq N_{i,k'}^{(W)}+B.
\end{align}
Since $\mathcal D_{\delta}$ contains at least two rounds,
there exists $k'' < k'$ such that $k'' \in \cW_{k'}$, the regret in this round is positive, and
$a_{i,k''} > 0$, which implies $N_{i,k'}^{(W)} > 0$.
Then, we combine Lemma~\ref{lemma:confidence_radius_bound} with
the definition of stable rounds and obtain
\begin{align*}
(1 - \varepsilon) \left(a_{i,k'}\mu_{i,k'} - L_{k'}^*\right)
\leq 2 a_{i,k'} \conf_W(i,k').
\end{align*}
Since $a_{i,k'} > 0$, we have
\begin{align*}
(1 - \varepsilon)
\left(\mu_{i,k'} - \frac{L_{k'}^*}{a_{i,k'}}\right)
\leq 4\alpha\max\left\{\sqrt{\frac{\log(2Bk'^3)}{N_{i,k'}^{(W)}}}, \frac{\log(2Bk'^3)}{N_{i,k'}^{(W)}}
\right\}.
\end{align*}
Therefore, we have
\begin{align*}
N_{i,k'}^{(W)}
&\leq \max \left\{
\frac{16\alpha^2\log(2Bk'^3)}{(1 - \varepsilon)^2 \left(\mu_{i,k'} - \frac{L_{k'}^*}{a_{i,k'}}\right)^2},
\frac{4\alpha\log(2Bk'^3)}{(1 - \varepsilon) \left(\mu_{i,k'} - \frac{L_{k'}^*}{a_{i,k'}}\right)},
\right\}.
\end{align*}
Given that $k' \leq K, W \leq K$ and $\mu_{i,k'} - \frac{L_{k'}^*}{a_{i,k'}} \geq \delta$, we obtain
\begin{align*}
N_{i,k'}^{(W)}
\leq \max\left\{
\frac{4\alpha^2 \beta_K}{(1 - \varepsilon)^2\delta^2},
\frac{\alpha \beta_K}{(1 - \varepsilon)\delta}\right\}.
\end{align*}
Substituting this into~(\ref{eq:oifndkslnfk}) proves the claim.
\end{proof}

\begin{lemma}\label{lemma:delta_sum_bound}
Consider $\varepsilon \in \left(0, 1\right)$,
a block $q$, and the corresponding set of
stable rounds $\mathcal{S}_i^\theta(q, W)$ from the proof of
Theorem~\ref{thm:master_bound}, for an arbitrary $\theta > 0$.
For any $i \in [n]$ with $\mathcal{S}_i^\theta(q, W) \neq \varnothing$, choose any
$p_i \in \argmin_{k \in \mathcal{S}_i^\theta(q, W)} \left\{\mu_{i,k} - \frac{L_k^*}{a_{i,k}}\right\}$.
Then,
\begin{align*}
\frac{1}{\Delta_{i,p_i}}
\leq \left(1 + \varepsilon + \frac{L_{\tau}^*}{\theta}\right) \frac{1}{\mu_{i, \tau}},
\end{align*}
where $\Delta_{i,k}= \mu_{i,k} - \frac{L_k^*}{a_{i,k}}$ and $\tau$ is the index of the first round in block $q$.
\begin{proof}
By the definition of the local variation $\rho_{p_i}$, we get 
\begin{align*}
\mu_{i, \tau} \leq \mu_{i, p_i} + \rho_{p_i}
\quad \Rightarrow \quad a_{i, p_i} \mu_{i, \tau} \leq a_{i, p_i} \mu_{i, p_i} + a_{i, p_i} \rho_{p_i}
\quad \Rightarrow \quad a_{i, p_i} \mu_{i, \tau} \leq \ell_{p_i}(\ba_{p_i}) + B \rho_{p_i}.
\end{align*}
In the last inequality, we use $a_{i, p_i} \leq B$ and $a_{i, p_i} \mu_{i, p_i} = \ell_{p_i}(\ba_{p_i})$.
Note that 
\begin{align*}
\max_{j \in [n]} \bar{a}_{j, p_i} \mu_{j, p_i} &\leq \max_{j \in [n]} \bar{a}_{j, \tau} \mu_{j, p_i}
&&\text{by the optimality of } \bar{\ba}_{p_i}, \\
&\leq \max_{j \in [n]} \bar{a}_{j, \tau} (\mu_{j, \tau} + \rho_{p_i})
&&\text{by the definition of } \rho_{p_i}.
\end{align*}
This implies that $L_{p_i}^* \leq L_{\tau}^* + B \rho_{p_i}$.
Combining these inequalities, we deduce
\begin{align}\label{eq:iowkldsmlkfdslk}
a_{i, p_i} \mu_{i, \tau} \leq L_{\tau}^* + (\ell_{p_i}(\ba_{p_i}) - L_{p_i}^*) + 2 B \rho_{p_i}
\leq L_{\tau}^* + (1 + \varepsilon)(\ell_{p_i}(\ba_{p_i}) - L_{p_i}^*).
\end{align}
In the last inequality, we use the fact that the round $p_i$ is stable. 
Using inequality~(\ref{eq:iowkldsmlkfdslk}) and $\ell_{p_i}(\ba_{p_i}) - L_{p_i}^* \geq \theta$, we 
complete the proof:
\begin{align*}
\frac{1}{\Delta_{i,p_i}} = \frac{a_{i,p_i}}{a_{i,p_i} \mu_{i,p_i} - L_{p_i}^*}
\leq \frac{L_{\tau}^* + (1 + \varepsilon)(\ell_{p_i}(\ba_{p_i}) - L_{p_i}^*)}
   {\mu_{i, \tau} (\ell_{p_i}(\ba_{p_i}) - L_{p_i}^*)}
\leq \left(1 + \varepsilon + \frac{L_{\tau}^*}{\theta}\right) \frac{1}{\mu_{i, \tau}}.
\end{align*}
\end{proof}
\end{lemma}

\subsection{Gap-dependent and gap-independent bounds}\label{sec:regret_bounds}
In this section we deduce from Theorem~\ref{thm:master_bound} a gap-dependent bound (Theorem~\ref{thm:general_ns_gap_dependent}) and 
gap-independent bounds (Theorems~\ref{thm:general_ns_gap_independent} and~\ref{thm:general_ns_positive_mean}).
The difference between the last two theorems is that the bound in Theorem~\ref{thm:general_ns_positive_mean} has better dependence on the parameters $B, n$;
however, it depends on the quantity $(\min_{i \in [n], k \in [K]} \mu_{i, k})^{-1}$ and may be worse 
when $\min_{i \in [n], k \in [K]} \mu_{i, k}$ is small.
Note that each theorem compares regimes of small and high variation $V_K$ separately and the stationary bounds follow 
directly from the regimes with small variation. 
Throughout this section all window lengths in the statements of the theorems are rounded up to the nearest integer.

For convenience, define $G_i := \max_{k \in [K]} \left\{B\mu_{i,k} - L_k^*\right\}$ 
and $G_{\max} := \max_{i,k} \left\{B\mu_{i,k} - L_k^*\right\}$.

\begin{theorem}[Gap-dependent bound]\label{thm:general_ns_gap_dependent}
Let the assumptions of Theorem~\ref{thm:master_bound} be satisfied.
Suppose first that $V_K \geq \frac{\alpha}{B K}$.
Then, choosing $W = \max\left\{1, \sqrt{\frac{K \alpha}{B V_K}}\right\}$ yields
\begin{align}
\regretdyn
&\lesssim n G_{\max} + \sqrt{\alpha B  V_K K }
   \left(\alpha \beta_K \sum_{i = 1}^n \frac{1}{\Delta_i^{\min}}
   + \beta_K \sum_{i = 1}^n \left(\log \frac{G_i}{\alpha B}\right)_+
   + \frac{1}{\alpha}\sum_{i = 1}^n G_i\right)
   + B V_K. \label{eq:gap_ns_bound}
\end{align}
Otherwise, choosing $W = K$, we obtain
\begin{align}
\regretdyn
&\lesssim n G_{\max} + \alpha
   \left(\alpha \beta_K \sum_{i = 1}^n \frac{1}{\Delta_i^{\min}}
   + \beta_K \sum_{i = 1}^n \left(\log \frac{G_i}{\alpha B}\right)_+\right). \label{eq:gap_stationary_bound}
\end{align}
\end{theorem}

\begin{proof}
Taking the first term in the minimum in Theorem~\ref{thm:master_bound} and setting
$\varepsilon = \frac{1}{2}$, we obtain a bound that holds for every $\theta > 0$.
Letting $\theta \to 0$, we obtain
\begin{align*}
\regretdyn
&\lesssim n G_{\max} + B V_K W
   + \frac{\alpha K}{W}
   \left(\alpha \beta_K \sum_{i = 1}^n \frac{1}{\Delta_i^{\min}}
   + \beta_K \sum_{i = 1}^n \left(\log \frac{G_i}{\alpha B}\right)_+
   + \frac{1}{\alpha}\sum_{i = 1}^n G_i\right).
\end{align*}
Here and below, we use $\tceil{\frac{K}{W}} \leq \frac{2K}{W}$ for $W \leq K$.
Then, we balance the terms $B V_K W$ and $\frac{\alpha K}{W}$ by choosing $W = \sqrt{\frac{K\alpha}{B V_K}}$.
Suppose that $V_K \geq \frac{\alpha}{B K}$.
Then, $\sqrt{\frac{K\alpha}{B V_K}} \leq K$.
For $W = \max\left\{1,\sqrt{\frac{K\alpha}{B V_K}}\right\}$, we have
$B V_K W \leq B V_K + \sqrt{\alpha B  V_K K }$ and
$\frac{\alpha K}{W} \leq \sqrt{\alpha B  V_K K }$.
Substituting these inequalities, we get
\begin{align*}
\regretdyn
&\lesssim n G_{\max} + \sqrt{\alpha B  V_K K }
   \left(1 + \alpha \beta_K \sum_{i = 1}^n \frac{1}{\Delta_i^{\min}}
   + \beta_K \sum_{i = 1}^n \left(\log \frac{G_i}{\alpha B}\right)_+
   + \frac{1}{\alpha}\sum_{i = 1}^n G_i\right)
   + B V_K.
\end{align*}
Next, we show $1 \leq \alpha \beta_K \sum_{i = 1}^n \frac{1}{\Delta_i^{\min}} + \frac{1}{\alpha}\sum_{i = 1}^n G_i$.
Indeed, if $\Delta_i^{\min} = +\infty$ for all $i$, then by Proposition~\ref{proposition:inf_deltas}, $B = 1$ and any allocation is optimal;
therefore $\regretdyn = 0$. 
Assume that there exists $i$ such that $\Delta_i^{\min} \neq +\infty$. For this $i$, we have 
\begin{align*}
G_i = \max_{k \in [K]} \left\{B\mu_{i, k} - L_{k}^*\right\} \geq  B \max_{k \in [K]:\, b_{i, k} \leq B} \left\{\mu_{i, k} - \frac{L_{k}^*}{B}\right\}  
\geq B \max_{k \in [K]:\, b_{i, k} \leq B} \left\{\mu_{i, k} - \frac{L_{k}^*}{b_{i, k}}\right\} \geq B \Delta_i^{\min}.
\end{align*}
For this $i$ and $K > 1$, we have 
\begin{align*}
\frac{\alpha \beta_K}{\Delta_i^{\min}} + \frac{G_i}{\alpha} \geq 
   2 \sqrt{\frac{\alpha \beta_K}{\Delta_i^{\min}} \cdot \frac{G_i}{\alpha}}
\geq 2 \sqrt{B \beta_K} \geq 1.
\end{align*}
Therefore, we prove~(\ref{eq:gap_ns_bound}).
Now, suppose that $V_K < \frac{\alpha}{B K}$.
Choosing $W = K$, we have $B V_K K < \alpha$. Using 
$1 \leq \alpha \beta_K \sum_{i = 1}^n \frac{1}{\Delta_i^{\min}} + \frac{1}{\alpha}\sum_{i = 1}^n G_i$
and $\sum_{i = 1}^n G_i \leq n G_{\max}$, we prove the claim.
\end{proof}

\begin{corollary}\label{cor:general_ns_gap_dependent}
Let the assumptions of Theorem~\ref{thm:general_ns_gap_dependent} be satisfied.
Suppose that $\frac{\alpha}{B K} \leq V_K \leq \frac{K\alpha}{B}$ and $K > 1$.
Then,
\begin{align*}
\regretdyn
&\lesssim \sqrt{\alpha B  V_K K }
   \left(\alpha \beta_K \sum_{i = 1}^n \frac{1}{\Delta_i^{\min}}
   + \beta_K \sum_{i = 1}^n \left(\log \frac{G_i}{\alpha B}\right)_+
   + \frac{n G_{\max}}{\alpha}\right).
\end{align*}
\end{corollary}

\begin{proof}
The upper bound on $V_K$ implies $B V_K \leq \sqrt{\alpha B  V_K K }$.
Thus, the last term in~(\ref{eq:gap_ns_bound}) is absorbed by the leading term.
By the definition of $G_{\max}$, we have $\sum_{i = 1}^n G_i \leq n G_{\max}$.
Finally, since $V_K \geq \frac{\alpha}{B K}$, we obtain $\sqrt{\alpha B  V_K K } \cdot \frac{n G_{\max}}{\alpha} \geq n G_{\max}$ 
and prove the claim. 
\end{proof}

\begin{theorem}[Gap-independent bound]\label{thm:general_ns_gap_independent}
Let the assumptions of Theorem~\ref{thm:master_bound} be satisfied.
Suppose first that $V_K \geq \alpha \sqrt{\frac{n \beta_K}{B K}}$.
Then, choosing
$W = \max\left\{1, \left(\frac{K^2 n \alpha^2 \beta_K}{B V_K^2}\right)^{\frac{1}{3}}\right\}$
yields
\begin{align}
\regretdyn
&\lesssim n G_{\max}
   + (\alpha^2 n B^2)^{\frac{1}{3}} K^{\frac{2}{3}} V_K^{\frac{1}{3}} \beta_K^{\frac{1}{3}} + B V_K 
   + \left(\frac{K B V_K^2}{n \alpha^2 \beta_K}\right)^{\frac{1}{3}}
   \left(\alpha \beta_K \sum_{i = 1}^n \left(\log \frac{G_i}{\alpha B}\right)_+
      + \sum_{i = 1}^n G_i\right). \label{eq:ind_ns_bound}
\end{align}
Otherwise, choosing $W = K$, we obtain
\begin{align}
\regretdyn
&\lesssim n G_{\max} + \alpha \sqrt{n B K \beta_K}
   + \alpha \beta_K \sum_{i = 1}^n \left(\log \frac{G_i}{\alpha B}\right)_+ \label{eq:ind_stationary_bound}
\end{align}
\end{theorem}

\begin{proof}
Setting $\varepsilon = \frac{1}{2}$ in Theorem~\ref{thm:master_bound} and taking the second term
in the minimum, we obtain
\begin{align*}
\regretdyn
\lesssim n G_{\max} + B V_K W + \underbrace{K \theta
   + \frac{K}{W}\frac{n B \alpha^2 \beta_K}{\theta}}_{A_1}
   + \frac{K}{W}\left(
   \alpha \beta_K \sum_{i = 1}^n \left(\log \frac{G_i}{\alpha B}\right)_+
   + \sum_{i = 1}^n G_i\right).
\end{align*}
We choose $\theta^* = \alpha \sqrt{\frac{n B \beta_K}{W}}$ to minimize $A_1$ and get
$A_1\lesssim K \alpha \sqrt{\frac{n B \beta_K}{W}}$.
Therefore,
\begin{align}\label{eq:ind_intermediate_bound}
\regretdyn
&\lesssim n G_{\max} + \underbrace{B V_K W + K \alpha \sqrt{\frac{n B \beta_K}{W}}}_{A_2}
   + \frac{K}{W}\left(\alpha \beta_K \sum_{i = 1}^n \left(\log \frac{G_i}{\alpha B}\right)_+
   + \sum_{i = 1}^n G_i\right).
\end{align}
Suppose that $V_K \geq \alpha \sqrt{\frac{n \beta_K}{B K}}$.
Then, $\left(\frac{K^2 n \alpha^2 \beta_K}{B V_K^2}\right)^{\frac{1}{3}} \leq K$.
The choice of $W$ in the theorem balances the terms in $A_2$. Substituting 
$W = \max\left\{1, \left(\frac{K^2 n \alpha^2 \beta_K}{B V_K^2}\right)^{\frac{1}{3}}\right\}$, we obtain
\begin{align*}
A_2 &\lesssim B V_K
   + (\alpha^2 n B^2)^{\frac{1}{3}} K^{\frac{2}{3}} V_K^{\frac{1}{3}} \beta_K^{\frac{1}{3}}
\quad\text{and}\quad
\frac{K}{W} \leq \left(\frac{K B V_K^2}{n \alpha^2 \beta_K}\right)^{\frac{1}{3}}.
\end{align*}
Substituting these inequalities into~(\ref{eq:ind_intermediate_bound})
proves~(\ref{eq:ind_ns_bound}).
Now, suppose that $V_K < \alpha \sqrt{\frac{n \beta_K}{B K}}$.
Choosing $W = K$, we have $B V_K K < \alpha \sqrt{n B K \beta_K}$ and 
\begin{align*}
\regretdyn
&\lesssim n G_{\max} + \alpha \sqrt{n B K \beta_K}
   + \alpha \beta_K \sum_{i = 1}^n \left(\log \frac{G_i}{\alpha B}\right)_+ + \sum_{i = 1}^n G_i.
\end{align*}
It remains to notice that $\sum_{i = 1}^n G_i \leq n G_{\max}$.
\end{proof}
\begin{corollary}\label{cor:general_ns_gap_independent}
Let the assumptions of Theorem~\ref{thm:general_ns_gap_independent} be satisfied.
Suppose that $\alpha \sqrt{\frac{n \beta_K}{B K}}
\leq V_K \leq \alpha K \sqrt{\frac{n \beta_K}{B}}$.
Then, choosing the window as in Theorem~\ref{thm:general_ns_gap_independent} yields
\begin{align*}
\regretdyn
&\leq \widetilde{\mathcal O}\Bigg(
   \left(\alpha^2 n B^2\right)^{\frac{1}{3}} K^{\frac{2}{3}} V_K^{\frac{1}{3}}
+ \left(\alpha n^2 B\right)^{\frac{1}{3}} K^{\frac{1}{3}} V_K^{\frac{2}{3}} \cdot \frac{G_{\max}}{\alpha}\Bigg).
\end{align*}
If instead $V_K < \alpha \sqrt{\frac{n \beta_K}{B K}}$, choosing $W = K$ yields
\begin{align*}
\regretdyn
&\lesssim \beta_K n G_{\max} + \alpha \sqrt{n B K \beta_K}.
\end{align*}
\end{corollary}

\begin{proof}
Since $(\log x)_+ \leq x$ for $x \geq 0$, $B \geq 1$, and $\beta_K \geq 1$, we have
\begin{align*}
\alpha \beta_K \sum_{i = 1}^n \left(\log \frac{G_i}{\alpha B}\right)_+ + \sum_{i = 1}^n G_i
&\leq \left(1 + \frac{\beta_K}{B}\right)\sum_{i = 1}^n G_i
\leq 2\beta_K n G_{\max}.
\end{align*}
The upper bound on $V_K$ gives $B V_K \leq (\alpha^2 n B^2)^{\frac{1}{3}}
K^{\frac{2}{3}} V_K^{\frac{1}{3}} \beta_K^{\frac{1}{3}}$.
The lower bound gives $\left(\nicefrac{K B V_K^2}{n \alpha^2 \beta_K}\right)^{1/3} \geq 1$.
Substituting these inequalities into~(\ref{eq:ind_ns_bound}) and absorbing $n G_{\max}$, we obtain
\begin{align*}
\regretdyn
&\lesssim \left(\alpha^2 n B^2\right)^{\frac{1}{3}}
   K^{\frac{2}{3}} V_K^{\frac{1}{3}} \beta_K^{\frac{1}{3}}
+ \left(\alpha n^2 B\right)^{\frac{1}{3}}
   K^{\frac{1}{3}} V_K^{\frac{2}{3}} \beta_K^{\frac{2}{3}} \frac{G_{\max}}{\alpha}.
\end{align*}
Hiding logarithmic factors proves the first bound.
For the small-variation regime, apply the same logarithmic bound to~(\ref{eq:ind_stationary_bound})
and use $\beta_K \geq 1$ to prove the second bound.
\end{proof}

\begin{theorem}[Gap-independent bound for means bounded away from zero]
\label{thm:general_ns_positive_mean}
Let the assumptions of Theorem~\ref{thm:master_bound} be satisfied.
Additionally, suppose that $\mu_{i,k} \geq m > 0$ for all $i \in [n]$ and $k \in [K]$.
Suppose first that $V_K \geq \frac{\alpha}{B}\sqrt{\frac{(B + n)\beta_K}{K}}$.
Then, choosing
$W = \max\left\{1, \left(\frac{K^2(B + n)\alpha^2 \beta_K}{B^2 V_K^2}\right)^{\frac{1}{3}}\right\}$
yields
\begin{align}
\regretdyn
&\lesssim n G_{\max}
   + K^{\frac{2}{3}}\left(B V_K\right)^{\frac{1}{3}} \left((B + n)\alpha^2 \beta_K\right)^{\frac{1}{3}}
   + B V_K \nonumber\\
&\quad + \left(
   \frac{K B^2 V_K^2}{(B + n)\alpha^2 \beta_K}\right)^{\frac{1}{3}}
   \left(\frac{\alpha^2 n \beta_K}{m}
   + \alpha \beta_K \sum_{i = 1}^n \left(\log \frac{G_i}{\alpha B}\right)_+
   + \sum_{i = 1}^n G_i\right). \label{eq:positive_ns_bound}
\end{align}
Otherwise, choosing $W = K$, we obtain
\begin{align}
\regretdyn
&\lesssim n G_{\max} + \alpha \sqrt{K(B + n)\beta_K}
   + \frac{\alpha^2 n \beta_K}{m}
   + \alpha \beta_K \sum_{i = 1}^n \left(\log \frac{G_i}{\alpha B}\right)_+. \label{eq:positive_stationary_bound}
\end{align}
\end{theorem}

\begin{proof}
Setting $\varepsilon = \frac{1}{2}$ in Theorem~\ref{thm:master_bound} and taking the third term
in the minimum, we obtain
\begin{align*}
\regretdyn
&\lesssim n G_{\max} + B V_K W + K \theta 
   + \alpha^2 \beta_K
   \sum_{q = 1}^{\tceil{\frac{K}{W}}}
   \left(1 + \frac{L_{\tau_q}^*}{\theta}\right)
   \sum_{i = 1}^n \frac{1}{\mu_{i,\tau_q}}
+ \frac{K}{W}\left(
   \alpha \beta_K \sum_{i = 1}^n \left(\log \frac{G_i}{\alpha B}\right)_+
   + \sum_{i = 1}^n G_i\right).
\end{align*}
Since $\mu_{i,k} \geq m$, we have $\sum_{i = 1}^n \frac{1}{\mu_{i,\tau_q}} \leq \frac{n}{m}$.
Moreover, Lemma~\ref{lemma:number_of_tasks} yields
$\sum_{i = 1}^n \frac{L_{\tau_q}^*}{\mu_{i,\tau_q}} \leq B + n - 1 \leq B + n$.
Therefore,
\begin{align*}
\regretdyn
&\lesssim n G_{\max} + B V_K W + \underbrace{K \theta
   + \frac{K}{W}\frac{(B + n)\alpha^2 \beta_K}{\theta}}_{A_1}
+ \frac{K}{W}\left(\frac{\alpha^2 n \beta_K}{m}
   + \alpha \beta_K \sum_{i = 1}^n \left(\log \frac{G_i}{\alpha B}\right)_+
   + \sum_{i = 1}^n G_i\right).
\end{align*}
As in the proof of Theorem~\ref{thm:general_ns_gap_independent}, 
we choose $\theta^* = \alpha \sqrt{\frac{(B + n)\beta_K}{W}}$ to minimize $A_1$.
Thus, we get
\begin{align*}
K \theta^* + \frac{K}{W}\frac{(B + n)\alpha^2 \beta_K}{\theta^*}
\lesssim K \alpha \sqrt{\frac{(B + n)\beta_K}{W}},
\end{align*}
which implies
\begin{align}\label{eq:positive_intermediate_bound}
\regretdyn
&\lesssim n G_{\max} + \underbrace{B V_K W + K \alpha \sqrt{\frac{(B + n)\beta_K}{W}}}_{A_2}
+ \frac{K}{W}\left(\frac{\alpha^2 n \beta_K}{m}
   + \alpha \beta_K \sum_{i = 1}^n \left(\log \frac{G_i}{\alpha B}\right)_+
   + \sum_{i = 1}^n G_i\right).
\end{align}

Suppose that $V_K \geq \frac{\alpha}{B}\sqrt{\frac{(B + n)\beta_K}{K}}$.
Then, choosing 
$W = \max\left\{1, \left(\frac{K^2(B + n)\alpha^2 \beta_K}{B^2 V_K^2}\right)^{\frac{1}{3}}\right\} \leq K$,
we balance the terms in $A_2$.
$\left(\frac{K^2(B + n)\alpha^2 \beta_K}{B^2 V_K^2}\right)^{\frac{1}{3}} \leq K$.
Therefore, 
$A_2 \lesssim B V_K + K^{\frac{2}{3}}\left(B V_K\right)^{\frac{1}{3}} \left((B + n)\alpha^2 \beta_K\right)^{\frac{1}{3}}$ and 
\begin{align*}
\regretdyn
&\lesssim n G_{\max} + B V_K + K^{\frac{2}{3}}\left(B V_K\right)^{\frac{1}{3}} \left((B + n)\alpha^2 \beta_K\right)^{\frac{1}{3}}
+ \left(\frac{B^2 K V_K^2}{(B + n)\alpha^2 \beta_K}\right)^{\frac{1}{3}} \left(\frac{\alpha^2 n \beta_K}{m}
   + \alpha \beta_K \sum_{i = 1}^n \left(\log \frac{G_i}{\alpha B}\right)_+
   + \sum_{i = 1}^n G_i\right).
\end{align*}

Now, suppose that $V_K < \frac{\alpha}{B}\sqrt{\frac{(B + n)\beta_K}{K}}$.
Choosing $W = K$, we have $B V_K K < \alpha \sqrt{K(B + n)\beta_K}$.
Substituting this inequality proves~(\ref{eq:positive_stationary_bound}).
\end{proof}

\begin{corollary}\label{cor:general_ns_positive_mean}
Let the assumptions of Theorem~\ref{thm:general_ns_positive_mean} be satisfied.
Suppose that $\frac{\alpha}{B}\sqrt{\frac{(B + n)\beta_K}{K}}
\leq V_K \leq \frac{\alpha}{B}\sqrt{(B + n)\beta_K} K$.
Then, choosing the window as in Theorem~\ref{thm:general_ns_positive_mean} yields
\begin{align*}
\regretdyn
&\leq \widetilde{\mathcal O}\Bigg(
   \left(\alpha^2 B(B + n)\right)^{\frac{1}{3}}
   K^{\frac{2}{3}} V_K^{\frac{1}{3}}
+ \left(\frac{\alpha n^3 B^2}{B + n}\right)^{\frac{1}{3}}
   K^{\frac{1}{3}} V_K^{\frac{2}{3}}
   \left(\frac{\alpha}{m\beta_K^{\frac{1}{3}}}
   + \frac{G_{\max}}{\alpha\beta_K^{\frac{4}{3}}}\right)\Bigg).
\end{align*}
If instead $V_K<\frac{\alpha}{B}\sqrt{\frac{(B+n)\beta_K}{K}}$, choosing $W=K$ yields
\begin{align*}
\regretdyn
&\lesssim \left(1+\frac{\beta_K}{B}\right)nG_{\max}+\alpha\sqrt{K(B+n)\beta_K}
   +\frac{\alpha^2 n\beta_K}{m}.
\end{align*}
\end{corollary}

\begin{proof}
Since $(\log x)_+\leq x$ for $x\geq0$, we have
\begin{align*}
\alpha\beta_K\sum_{i=1}^n \left(\log\frac{G_i}{\alpha B}\right)_+ &\leq \frac{\beta_K}{B}\sum_{i=1}^nG_i.
\end{align*}
Therefore, for the chosen window,~(\ref{eq:positive_ns_bound}) yields
\begin{align*}
\regretdyn
&\lesssim nG_{\max}+BV_K
   +K^{\frac{2}{3}}(BV_K)^{\frac{1}{3}}
   \left((B+n)\alpha^2\beta_K\right)^{\frac{1}{3}} \\
&\quad +\left(\frac{KB^2V_K^2}{(B+n)\alpha^2\beta_K}\right)^{\frac{1}{3}}
   \left(\frac{\alpha^2n\beta_K}{m}
   +\left(1+\frac{\beta_K}{B}\right)\sum_{i=1}^nG_i\right).
\end{align*}
The upper bound on $V_K$ gives
$BV_K\leq K^{\frac{2}{3}}(BV_K)^{\frac{1}{3}}
((B+n)\alpha^2\beta_K)^{\frac{1}{3}}$.
The lower bound gives
$\left(\nicefrac{KB^2V_K^2}{(B+n)\alpha^2\beta_K}\right)^{1/3}\geq1$.
Using these inequalities, $\sum_{i=1}^nG_i\leq nG_{\max}$, and
$1+\nicefrac{\beta_K}{B}\leq1+\beta_K$, we obtain
\begin{align*}
\regretdyn
&\lesssim \left(\alpha^2B(B+n)\right)^{\frac{1}{3}}
   K^{\frac{2}{3}}V_K^{\frac{1}{3}}\beta_K^{\frac{1}{3}}
+\left(\frac{\alpha n^3B^2}{B+n}\right)^{\frac{1}{3}}
   K^{\frac{1}{3}}V_K^{\frac{2}{3}}\beta_K(1+\beta_K)
   \left(\frac{\alpha}{m\beta_K^{\frac{1}{3}}}
   +\frac{G_{\max}}{\alpha\beta_K^{\frac{4}{3}}}\right).
\end{align*}
Hiding logarithmic factors proves the first bound.
For the small-variation regime, apply the same logarithmic bound to~(\ref{eq:positive_stationary_bound})
and use $\sum_{i=1}^nG_i\leq nG_{\max}$ to prove the second bound.
\end{proof}

\subsection{Comparison with Baselines}
In this section, we compare our bounds for proxy loss with those in \citep{takemura2021near,nakamura2025multi}
on the common subclass $0 < X_{i,k}^{(u)} \leq 1$.
\begin{theorem}
Let Assumption~\ref{ass:stationary-exp} be satisfied, and suppose that $0<X_{i,k}^{(u)}\leq1$ a.s. for all $i,k,u$
and let $m = \min_{i \in [n]} \mu_i$. 
Then, ATA with $\alpha = 1$ yields gap-dependent regret $n G_{\max} + \alpha^2 \beta_K \sum_{i = 1}^n \nicefrac{1}{\Delta_i}$
and gap-independent regret 
\begin{align*}
\regretdyn
&\lesssim n G_{\max} + \min\left\{\alpha \sqrt{n B K \beta_K}, \alpha \sqrt{K(B + n)\beta_K}
   + \frac{\alpha^2 n \beta_K}{m}\right\}.
\end{align*}
\end{theorem}
\begin{proof}
First, we notice that if $0 < X_{i,k}^{(u)} \leq 1$, then $\|X_{i,k}^{(u)} - \mu_{i, k}\|_{\psi_1} \leq 1$.
Therefore, we can apply Theorem~\ref{thm:general_ns_gap_dependent} with $\alpha = 1$ 
to obtain a gap-dependent bound.
Since $\log \frac{G_i}{\alpha B} \leq \log \frac{\mu_i}{\alpha} \leq \log \mu_i \leq 0$, we get
for $K \geq 1$
\begin{align*}
\regretdyn
&\lesssim n G_{\max} + \alpha^2 \beta_K \sum_{i = 1}^n \frac{1}{\Delta_i}.
\end{align*}
Consider now the gap-independent bound. Similarly, we bound the sub-exponential term by $0$. 
Then, combining Theorems~\ref{thm:general_ns_positive_mean} and~\ref{thm:general_ns_gap_independent}, we obtain
\begin{align*}
\regretdyn
&\lesssim n G_{\max} + \min\left\{\alpha \sqrt{n B K \beta_K}, \alpha \sqrt{K(B + n)\beta_K}
   + \frac{\alpha^2 n \beta_K}{m}\right\}.
\end{align*}
\end{proof}

\paragraph{The bound of \citet{takemura2021near}.}
We use Theorem~4 of \citet{takemura2021near} with an exact optimization oracle.
Their framework considers base-arm means of the form $\mathbf x_j^\top\boldsymbol\theta$, where the feature
vectors $\mathbf x_j\in\mathbb R^d$ are known, $\boldsymbol\theta\in\mathbb R^d$ is unknown, and $q$ distinct
arms are selected in each round. Suppose that the noise is conditionally $\sigma$-sub-Gaussian and the
absolute base-arm means are bounded by $M$. The features and parameter also have bounded Euclidean norms.
The objective $f(I,\mathbf v)$ must be nondecreasing in $\mathbf v$ and satisfy
$|f(I,\mathbf v)-f(I,\mathbf v')| \leq C\sum_{j\in I}|v_j-v_j'|$.
With the parameter choices specified in their theorem, Algorithm~2 satisfies
\begin{align*}
\bar \regretdyn &\leq \widetilde{\mathcal O}\left(C\left(\sigma d\sqrt{qK}+Mdq\right)\right)
\end{align*}
with probability at least $1-\delta$, for any $\delta\in(0,1)$.
Here, $\bar \regretdyn$ denotes the random pseudo-regret as in Section~\ref{sec:setup}.

\begin{theorem}\label{thm:takemura_proxy_bound}
Let Assumption~\ref{ass:stationary-exp} be satisfied, and suppose that $0<X_{i,k}^{(u)}\leq1$ a.s. for all $i,k,u$.
Then, Algorithm~2 of \citet{takemura2021near}, applied through the reduction below, satisfies
\begin{align*}
\regretdyn &\leq \widetilde{\mathcal O}\left(n\sqrt{BK}+nB\right).
\end{align*}
\end{theorem}

\begin{proof}
Introduce $B$ arms indexed by $(i,1),\dots,(i,B)$ for each worker $i$, so the total number of arms is $nB$.
For each arm $(i,u)$, we define the feature vector as $\mathbf e_i\in\mathbb R^n$ and the unknown parameter
as $\boldsymbol\theta=-\bmu$. Selecting arm $(i,u)$ in round $k$, we observe the reward
$-X_{i,k}^{(u)}$.
Choosing $B$ distinct arms forms a set $I$ and corresponds to the allocation $a_i=|\{u:(i,u)\in I\}|$.
Conversely, every feasible allocation can be represented by such a set. Define the objective
\begin{align*}
f(I,\mathbf v) &:= \min_{i\in[n]}\sum_{u:\,(i,u)\in I}v_{i,u},
\end{align*}
where an empty sum equals zero. By definition, the true mean reward of arm $(i,u)$ is
$\langle\mathbf e_i,\boldsymbol\theta\rangle=-\mu_i$. Therefore, at the true base-arm means $v_{i,u}=-\mu_i$,
we have $f(I,\mathbf v)=-\ell(\ba,\bmu)$. Thus, maximizing $f$ is equivalent to minimizing the proxy loss.
Moreover, $f$ is nondecreasing in every coordinate and
\begin{align*}
|f(I,\mathbf v)-f(I,\mathbf v')|
&\leq \max_{i\in[n]}\left|\sum_{u:\,(i,u)\in I}(v_{i,u}-v_{i,u}')\right|
\leq \sum_{(i,u)\in I}|v_{i,u}-v_{i,u}'|.
\end{align*}
Therefore, we can take $C=1$, $d=n$, and $q=B$.
The required norm bounds hold since $\|\mathbf e_i\|_2=1$ and $\|\boldsymbol\theta\|_2\leq\sqrt n$.
Since $-X_{i,k}^{(u)}\in[-1,0]$ and the observations are independent, the noise is conditionally
$\nicefrac{1}{2}$-sub-Gaussian, and the absolute means are at most $1$.
Thus, we have formulated the proxy-loss minimization problem as a contextual combinatorial semi-bandit problem.
Applying their theorem, we obtain $\bar \regretdyn \leq \widetilde{\mathcal O}\left(n\sqrt{BK}+nB\right)$
with probability at least $1-\delta$.
Since the regret in each round is at most $B$, the total pseudo-regret is at most $BK$. Therefore,
\begin{align*}
\regretdyn = \Exp{\bar \regretdyn}
&\leq \widetilde{\mathcal O}\left(n\sqrt{BK}+nB\right) + \delta BK.
\end{align*}
Choosing $\delta=\nicefrac{1}{2K}$ finishes the proof.
\end{proof}

\paragraph{The bound of \citet{nakamura2025multi}.}
Theorem~3.1 of \citet{nakamura2025multi} considers integer actions
$\ba\in\mathcal A\subseteq\mathbb Z_{\geq0}^d$.
Selecting $a_i$ units of arm $i$ provides $a_i$ independent observations
from its fixed distribution on $[0,1]$.
Write $I_{\ba}:=\{i:a_i>0\}$ and $s:=\max_{\ba\in\mathcal A}|I_{\ba}|$.
Their mean-based loss $r$ satisfies
$|r(\ba,\bmu)-r(\ba,\bmu')|
\leq\kappa_r\sum_{i\in I_{\ba}}|\mu_i-\mu_i'|$,
and an exact oracle minimizing $r(\ba,\bmu)$ is available.
For a unique optimal action $\bar\ba$ and $s\geq2$, the leading term of
their GenCTS bound is
\begin{align*}
\regretdyn
&\leq C_0+\mathcal O\left(
   \kappa_r^2\log s\sum_{i=1}^d\frac{\log K}{\Delta_i^{\mathrm{MP}}}\right),
\end{align*}
where
$\Delta_i^{\mathrm{MP}}
:=\min_{\ba\in\mathcal A\setminus\{\bar\ba\}:\,a_i>0}
\{r(\ba,\bmu)-r(\bar\ba,\bmu)\}$.
Here, $C_0$ collects terms independent of $K$.
\begin{theorem}\label{thm:gencts_proxy_bound}
Let Assumption~\ref{ass:stationary-exp} be satisfied, and suppose that $0<X_{i,k}^{(u)}\leq1$ a.s. for all $i,k,u$.
Assume that $n,B\geq2$ and the optimal proxy allocation $\bar\ba$ is unique.
Then, the GenCTS analysis applied to the proxy-loss objective yields
\begin{align*}
\regretdyn
&\leq C_0+\mathcal O\left(
   B^2\log(\min\{n,B\})
   \sum_{i=1}^n\frac{\log K}{\Delta_i^{\mathrm{MP}}}\right),
\end{align*}
where
$\Delta_i^{\mathrm{MP}}
:=\min_{\ba\in\mathcal A\setminus\{\bar\ba\}:\,a_i>0}
\{\ell(\ba,\bmu)-L^*\}$,
and $C_0$ may depend on the instance but not on $K$.
\end{theorem}
\begin{proof}
We use $r(\ba,\bmu)=\ell(\ba,\bmu)$ as the objective of the GenCTS oracle
and measure performance by proxy-loss regret.
The feedback model agrees with ours: assigning $a_i$ tasks to worker $i$
provides $a_i$ independent observations with mean $\mu_i$.
For any $\ba\in\mathcal A$, we have
\begin{align*}
|\ell(\ba,\bmu)-\ell(\ba,\bmu')|
&\leq\max_{i\in I_{\ba}}a_i|\mu_i-\mu_i'|
\leq B\sum_{i\in I_{\ba}}|\mu_i-\mu_i'|.
\end{align*}
Thus, $\kappa_r=B$ is admissible.
Moreover, $d=n$ and $s=\min\{n,B\}$.
We then apply Theorem~\ref{thm:gencts_proxy_bound} with $r=\ell$.
\end{proof}

\section{Knapsack Problem}\label{sec:knapsack}
In this section we slightly generalize the approach from the proof sketch for ATA. We first state the three
steps in a more general form and then apply them to the knapsack problem from~\citep{nakamura2025multi}. The main difference is in Step~1.

\textbf{Step 1.} Assume that the regret on a fixed trajectory, restricted to the selected rounds, satisfies
\begin{align*}
\bar \regretdyn \leq \mathcal O\left(\sum_{i = 1}^n \sum_{k \in \mathcal S_i} a_{i,k}'\Delta_{i,k}\right).
\end{align*}
Here, $\mathcal S_i \subseteq [K]$ are sets of rounds, and
$0 \leq a_{i,k}' \leq a_{i,k}$. For each $i$ with $\mathcal S_i \neq \varnothing$, the confidence levels satisfy
$0 < \Delta_i \leq \Delta_{i,k} \leq 2\conf(i,k)$ for every $k \in \mathcal S_i$.

\textbf{Step 2.} We fix $i$ with $\mathcal S_i \neq \varnothing$ and define $\delta_1 < \cdots < \delta_J$ as the
increasing sequence of distinct values in $\{\Delta_{i,k}\}_{k \in \mathcal S_i}$. Additionally, we set
$\delta_0 = 0$. By Lemma~\ref{lemma:weighted_gap_summation}, we have
\begin{align}\label{eq:knapsack_levels}
\sum_{k \in \mathcal S_i} a_{i,k}'\Delta_{i,k} = \sum_{j = 1}^J (\delta_j - \delta_{j-1})
   \sum_{k \in \mathcal S_i:\,\Delta_{i,k} \geq \delta_j} a_{i,k}'.
\end{align}

\textbf{Step 3.} Since $\delta_j \leq \Delta_{i,k} \leq 2\conf(i,k)$ and $a_{i,k}' \leq a_{i,k}$, we bound
$\sum_{k \in \mathcal S_i:\,\Delta_{i,k} \geq \delta_j} a_{i,k}'$ by counting samples up to the last round with
$\Delta_{i,k} \geq \delta_j$, as in Lemma~\ref{lemma:stable_samples}. For confidence intervals constructed for sub-Gaussian 
variables, this number of samples is $\mathcal O\left(\nicefrac{\log K}{\delta_j^2}\right)$, omitting constant factors.
Additionally, we hide the sub-Gaussian parameter in $\mathcal O\left(\cdot\right)$. Then, 
\begin{align*}
\sum_{k \in \mathcal S_i} a_{i,k}'\Delta_{i,k} \leq O\left(\log K \sum_{j = 1}^J
   \frac{\delta_j - \delta_{j-1}}{\delta_j^2}\right) \leq O\left(\frac{\log K}{\delta_1}\right) \leq O\left(\frac{\log K}{\Delta_i}\right).
\end{align*}
Below, we apply this proof idea to the knapsack problem.

\paragraph{Problem setting.}
We consider $n$ types of items. Each item of type $i \in [n]$ has a known weight $w_i > 0$ and an unknown mean
value $\mu_i > 0$. A knapsack of capacity $W$ can contain multiple copies of each item type. The goal is to
maximize the total mean value without exceeding the capacity:
\begin{align*}
\bar \ba \in \argmax_{\ba \in \mathcal A} \ba^T \bmu, \qquad
\mathcal A := \left\{\ba \in \mathbb Z_{\geq 0}^n:\,\ba^T \bw \leq W\right\}.
\end{align*}
We assume $w_i \leq W$ for every $i$.
In each round $k \in [K]$, the algorithm chooses $\ba_k \in \mathcal A$ before observing the rewards. For every
item type $i$, it then observes $a_{i,k}$ independent values $X_{i,k}^{(1)}, \dots, X_{i,k}^{(a_{i,k})}$.
The observations are positive and mutually independent. For each $i$, they have a fixed distribution with mean
$\mu_i$. We assume that $\|X_{i,k}^{(u)} - \mu_i\|_{\psi_2} \leq \sigma$ for a known $\sigma > 0$, where the
sub-Gaussian norm is defined by
\begin{align*}
\|Y\|_{\psi_2} := \inf\left\{c > 0:\,\Exp{\exp(Y^2/c^2)} \leq 2\right\}.
\end{align*}
In this section, the expected regret is defined by
\begin{align*}
\regretdyn := \Exp{\sum_{k = 1}^K \left(\bar \ba^T \bmu - \ba_k^T \bmu\right)}.
\end{align*}

\paragraph{Algorithm.}
We assume that there exists an oracle $O$ that, for given $\bw$, $\bu$, and $W$, returns an allocation maximizing $\ba^T \bu$ over
$\mathcal A$, as in GenCTS \citep{nakamura2025multi}. 
Since the objective is maximized, we use upper confidence bounds. In the first $\min\{n,K\}$
rounds, we fill the knapsack with copies of one item type at a time. For $k > n$, let
\begin{align*}
N_{i,k} := \sum_{q = 1}^{k-1} a_{i,q}, \qquad
\hat \mu_{i,k} := \frac{1}{N_{i,k}}\sum_{q = 1}^{k-1}\sum_{u = 1}^{a_{i,q}} X_{i,q}^{(u)}.
\end{align*}
Using the same argument as in Lemma~\ref{lemma:sufficient-beta} and the sub-Gaussian assumption, 
we define the confidence radius and the upper confidence bounds (UCBs) by
\begin{align}\label{eq:knapsack_ucb}
\conf(i,k) := 2\sigma\sqrt{\frac{\log\left(2 B k^3\right)}{N_{i,k}}}, \quad
u_{i,k} := \hat \mu_{i,k} + \conf(i,k),
\end{align}
where $B := \tfloor{W/\min_{i \in [n]}w_i}$ is the maximum number of items in a feasible allocation.
We then present Algorithm~\ref{alg:knapsack_ucb}, which is similar to 
Algorithm~\ref{alg:sw_ata} (with $W=K$). The main differences are substituting LCBs with UCBs and RAS with oracle $O$.
\begin{algorithm}[h]
\caption{UCB for the Knapsack Problem}\label{alg:knapsack_ucb}
\begin{algorithmic}[1]
\STATE \textbf{Input:} weights $\bw$, capacity $W$, $\sigma > 0$
\FOR{$k = 1,\dots,K$}
    \IF{$k \leq n$}
        \STATE Set $a_{k,k} = \tfloor{W/w_k}$ and $a_{i,k} = 0$ for all $i \neq k$
    \ELSE
        \STATE Compute UCBs $u_{i,k}$ for all $i \in [n]$ using~(\ref{eq:knapsack_ucb})
        \STATE Set $\ba_k = \operatorname{Oracle}(\bw,\bu_k;W)$
    \ENDIF
    \STATE Observe $X_{i,k}^{(u)}$ for all $i \in [n]$ and $u = 1,\dots,a_{i,k}$
\ENDFOR
\end{algorithmic}
\end{algorithm}

\begin{theorem}\label{thm:knapsack_regret}
Under the assumptions above, let
$M := \max_{i \in [n]}\mu_i$, and $\beta_K := 4\log(2BK^3)$. For a fixed optimal allocation $\bar \ba$, define sub-optimality gaps
\begin{align}\label{eq:sdfjksndfkjnsd}
\Delta_i := \min_{\substack{\ba \in \mathcal A:\,\bar \ba^T \bmu > \ba^T \bmu,\\ a_i > \bar a_i}}
   \frac{\bar \ba^T \bmu - \ba^T \bmu}{2\sum_{j = 1}^n (a_j - \bar a_j)_+},
\end{align}
with the convention that the minimum of an empty set is $+\infty$ and $1/(+\infty) = 0$. Then,
Algorithm~\ref{alg:knapsack_ucb} satisfies
\begin{align*}
\regretdyn \leq \mathcal O\left(n B M + \sigma^2 \beta_K \sum_{i = 1}^n \frac{1}{\Delta_i}\right).
\end{align*}
\end{theorem}

\begin{proof}
We apply the three steps above. As in the proof of Theorem~\ref{thm:master_bound}, we
bound the expected regret as 
\begin{align}\label{eq:ksdflknsdf}
\regretdyn \leq 3n \bar \ba^T \bmu
   + \Exp{\sum_{k \in \mathcal S}\left(\bar \ba^T \bmu - \ba_k^T \bmu\right)},
\end{align}
where $\mathcal S := \{k \in [K]:\,k > n,\ \neg\mathcal E_k,\ \ba_k^T \bmu < \bar \ba^T \bmu\}$.
We now fix a trajectory and bound the sum over $\mathcal S$.

For $k \in \mathcal S$, we have $0 \leq u_{i,k} - \mu_i \leq 2\conf(i,k)$. The oracle guarantees
$\ba_k^T \bu_k \geq \bar \ba^T \bu_k$, so
\begin{align*}
\bar \ba^T \bmu - \ba_k^T \bmu &\leq \left(\ba_k - \bar \ba\right)^T\left(\bu_k - \bmu\right)
\leq \sum_{i = 1}^n \left(a_{i,k} - \bar a_i\right)_+\left(u_{i,k} - \mu_i\right).
\end{align*}
Set $a_{i,k}' := (a_{i,k} - \bar a_i)_+$ and $\Delta_{i,k} := u_{i,k} - \mu_i$. By the definition of the gaps,
\begin{align*}
\sum_{i:\,a_{i,k}' > 0} a_{i,k}'\Delta_i \leq \frac{\bar \ba^T \bmu - \ba_k^T \bmu}{2}.
\end{align*}
Define $\mathcal S_i := \{k \in \mathcal S:\,a_{i,k}' > 0,\ \Delta_{i,k} \geq \Delta_i\}$.
The terms with $\Delta_{i,k} < \Delta_i$ contribute at most half the regret at each round. Therefore,
\begin{align}\label{eq:knapsack_charge}
\sum_{k \in \mathcal S}\left(\bar \ba^T \bmu - \ba_k^T \bmu\right) \leq 2\sum_{i = 1}^n
   \sum_{k \in \mathcal S_i} a_{i,k}'\Delta_{i,k}.
\end{align}
For every $k \in \mathcal S_i$, we have $0 < \Delta_i \leq \Delta_{i,k} \leq 2\conf(i,k)$ and
$0 < a_{i,k}' \leq a_{i,k}$. This completes Step~1.
Fix $i$ with $\mathcal S_i \neq \varnothing$ and the levels
$\delta_j$ from Step~2. Let $k'$ be the last round in $\{k \in \mathcal S_i:\,\Delta_{i,k} \geq \delta_j\}$. Since
$\delta_j \leq \Delta_{i,k'} \leq 2\conf(i,k')$, we obtain $N_{i,k'} \leq \nicefrac{4\sigma^2\beta_K}{\delta_j^2}$.
Due to the first $n$ rounds, we have $N_{i,k'} \geq \tfloor{W/w_i} \geq a_{i,k'}$. As in
Lemma~\ref{lemma:stable_samples}, we count the samples up to this last round and obtain
\begin{align}\label{eq:knapsack_tail}
\sum_{k \in \mathcal S_i:\,\Delta_{i,k} \geq \delta_j} a_{i,k}' &\leq N_{i,k'} + a_{i,k'}
   \leq 2N_{i,k'} \leq \frac{8\sigma^2\beta_K}{\delta_j^2}.
\end{align}
Applying Steps~2 and~3, we deduce
\begin{align*}
\sum_{k \in \mathcal S_i} a_{i,k}'\Delta_{i,k} \leq \frac{16\sigma^2\beta_K}{\Delta_i},
\end{align*}
Summing over $i$, applying~(\ref{eq:knapsack_charge}), and using $\bar \ba^T \bmu \leq BM$ in~(\ref{eq:ksdflknsdf}) proves the claim.
\end{proof}

\paragraph{Comparison with \citet{nakamura2025multi}.}
We compare the bounds on the common class $0 < X_{i,k}^{(u)} \leq 1$, for which we can take $\sigma = 1$ and
$M \leq 1$. Assume that $\bar \ba$ is unique, and let $s := \max_{\ba \in \mathcal A}|\supp(\ba)| \geq 2$.
Theorem~3.1 of \citet{nakamura2025multi}, with the Lipschitz constant $\kappa_r = B$ from their Section~3,
gives the leading term
\begin{align*}
\regretdyn^{\mathrm{GenCTS}} \leq C_0
   + \mathcal O\left(B^2\log s\sum_{i = 1}^n\frac{\log K}{\Delta_i^{\mathrm{MP}}}\right),
\end{align*}
where $C_0$ is independent of $K$ and
\begin{align*}
\Delta_i^{\mathrm{MP}} := \min_{\ba \in \mathcal A\setminus\{\bar \ba\}:\,a_i > 0}
   \left\{\bar \ba^T\bmu - \ba^T\bmu\right\}.
\end{align*}
Our gaps use only allocations with $a_i > \bar a_i$ and have a denominator at most $2B$. Therefore,
$\Delta_i \geq \Delta_i^{\mathrm{MP}}/(2B)$, and Theorem~\ref{thm:knapsack_regret} implies
\begin{align*}
\regretdyn \leq \mathcal O\left(nB + B\beta_K\sum_{i = 1}^n\frac{1}{\Delta_i^{\mathrm{MP}}}\right).
\end{align*}
Thus, when both bounds are written using the same gaps, the coefficient of the leading logarithmic term has
dependence $B$ instead of $B^2\log s$. We compare against GenCTS, which uses the same exact knapsack oracle as our algorithm.
Another algorithm from~\citep{nakamura2025multi} (GenLBINFV)  
solves a regularized problem over $\operatorname{conv}(\mathcal A)$ and uses a convex decomposition to sample an integer allocation.
The authors note that solving this regularized problem exactly is difficult for knapsack
\citep[Section~4]{nakamura2025multi}. Thus, we compare methods with the same optimization step at each iteration.

\section{Connection between Proxy-loss and the True Time}
\PROXCONNECTION*
\begin{proof}
The case $B = 1$ is trivial, since the proxy loss and the expected completion time coincide and 
$C_K = C_K^* + \regretdyn$. 

Assume $B \geq 2$.
We fix $k \in [K]$. Let $\ba$ be a fixed allocation. Define $Y_{i, k}^{(u)} = X_{i, k}^{(u)} - \mu_{i, k}$.
Then, for any $i \in \supp(\ba)$
and all $\left|\lambda\right| \leq \nicefrac{1}{\alpha_{i,k}}$, Assumption~\ref{ass:drift-exp}
and independence give
\begin{align}\label{eq:ksndfnsdfnl}
\Exp{\exp\left(\lambda \sum_{u = 1}^{a_i} Y_{i, k}^{(u)}\right)}
= \prod_{u = 1}^{a_i} \Exp{\exp\left(\lambda Y_{i, k}^{(u)}\right)}
\leq \exp\left(a_i \alpha_{i,k}^2 \lambda^2\right)
= \exp\left((\sqrt{a_i} \alpha_{i,k})^2 \lambda^2\right).
\end{align}
We denote
$\tilde \alpha = \max_{i \in [n]} \sqrt{a_i} \alpha_{i,k}$.
Since $a_i \geq 1$ for $i \in \supp(\ba)$, we have $\tilde\alpha \geq \alpha_{i,k}$ for every such $i$.
Therefore, inequality~(\ref{eq:ksndfnsdfnl}) implies that
\begin{align*}
\Exp{\exp\left(\lambda \sum_{u = 1}^{a_i} Y_{i, k}^{(u)}\right)} \leq \exp\left(\tilde \alpha^2 \lambda^2\right)
\quad \text{for all} \quad
i \in \supp(\ba) \quad \text{and} \quad
\left|\lambda\right| \leq \frac{1}{\tilde \alpha}.
\end{align*}
Thus, the conditions of Lemma~\ref{lemma:mgf_max_bound} are satisfied, and we get
\begin{align}\label{eq:klsdnflksdnfl}
\Exp{\max_{i \in \supp(\ba)} \sum_{u = 1}^{a_i} X_{i, k}^{(u)}}
\leq \max_{i \in [n]} \left\{a_i \mu_{i, k}\right\}
   + \Exp{\max_{i \in \supp(\ba)} \sum_{u = 1}^{a_i} Y_{i, k}^{(u)}}
&\leq \max_{i \in [n]} \left\{a_i \mu_{i, k}\right\} \\
&\quad + \left(1 + \ln \min\left\{n, B\right\}\right) \tilde \alpha. \nonumber
\end{align}
Here, we use Remark~\ref{remark:tighter_max_bound} to obtain the tighter constant
and the inequality $\left|\supp(\ba)\right| \leq \min\left\{n, B\right\}$.
Then, we denote by $\ExpSub{k - 1}{\cdot}$ the conditional expectation given the history prior to the $k$th round
and use~(\ref{eq:klsdnflksdnfl}) to bound the true expected time.
\begin{align}\label{eq:ksladsfssdf}
\sum_{k = 1}^K \Exp{\max_{i \in \supp(\ba_k)} \sum_{u = 1}^{a_{i, k}} X_{i, k}^{(u)}}
&= \Exp{\sum_{k = 1}^K \ExpSub{k - 1}{\max_{i \in \supp(\ba_k)} \sum_{u = 1}^{a_{i, k}} X_{i, k}^{(u)}}} \nonumber \\
&\leq \Exp{\sum_{k = 1}^K \ell_k(\ba_k)}
   + \left(1 + \ln \min\left\{n, B\right\}\right)\sum_{k = 1}^K \Exp{\max_{i \in [n]} \left\{\alpha_{i, k} \sqrt{a_{i, k}}\right\}}.
\end{align}
Consider the last term separately. Note that the following two bounds hold simultaneously:
\begin{align*}
\alpha_{i, k} \sqrt{a_{i, k}} \leq \frac{\alpha_{i, k}}{\sqrt{\mu_{i, k}}} \cdot \sqrt{\ell_k(\ba_k)}
\quad\text{and}\quad
\alpha_{i, k} \sqrt{a_{i, k}} \leq \frac{\alpha_{i, k}}{\mu_{i, k}} \cdot \ell_k(\ba_k).
\end{align*}
Then, we fix $x > 0$ and define
\begin{align*}
\varphi(x) = \max_{i \in [n], k \in [K]}
   \left\{\alpha_{i, k} \cdot \min\left\{\frac{1}{x \sqrt{\mu_{i, k}}},
      \frac{1}{\mu_{i, k}}\right\}\right\}.
\end{align*}
Additionally, we define the sets of fast and slow worker-round pairs, respectively, by
\begin{align*}
\mathcal F = \left\{(i, k):\, i \in [n],\, k \in [K],\,
   \frac{\alpha_{i, k}}{\mu_{i, k}}
      > \frac{\alpha_{i, k}}{\sqrt{\mu_{i, k}} x}\right\}
\quad \text{and} \quad
\mathcal S = \left\{(i, k):\, i \in [n],\, k \in [K]\right\}\setminus\mathcal F.
\end{align*}
In this notation,
\begin{align}\label{eq:lsidhfjsdh}
\varphi(x) =
\max\left\{\max_{(i, k) \in \mathcal F} \left\{\alpha_{i, k} \cdot \frac{1}{x \sqrt{\mu_{i, k}}}\right\},
      \max_{(i, k) \in \mathcal S} \left\{\alpha_{i, k} \cdot \frac{1}{\mu_{i, k}}\right\}\right\}.
\end{align}
Here, maxima over empty sets are taken to be zero.
Let $i_k$ be the smallest index in $\supp(\ba_k)$ such that
$\alpha_{i_k, k} \sqrt{a_{i_k, k}} = \max_{i \in [n]} \left\{\alpha_{i, k} \sqrt{a_{i, k}}\right\}$.
We split the rounds into two groups:
\begin{align*}
\mathcal K_{\mathcal F} = \left\{k \in [K]:\, (i_k, k) \in \mathcal F\right\}
\quad \text{and} \quad
\mathcal K_{\mathcal S} = [K]\setminus\mathcal K_{\mathcal F},
\end{align*}
where we refer to the first set as fast rounds and the second set as slow rounds.
Then, we bound the sum over rounds of the first type.
\begin{align*}
\sum_{k \in \mathcal K_{\mathcal F}} \alpha_{i_k, k} \sqrt{a_{i_k, k}}
\leq \sum_{k \in \mathcal K_{\mathcal F}} \frac{\alpha_{i_k, k}}{\sqrt{\mu_{i_k, k}}} \cdot \sqrt{\ell_k(\ba_k)}
&\leq \max_{(i, k) \in \mathcal F} \left\{\frac{\alpha_{i, k}}{\sqrt{\mu_{i, k}}}\right\}
   \cdot \sum_{k \in \mathcal K_{\mathcal F}} \sqrt{\ell_k(\ba_k)}.
\end{align*}
By Cauchy--Schwarz and Jensen's inequalities, using $\left|\mathcal K_{\mathcal F}\right| \leq K$, we obtain
\begin{align*}
\Exp{\sum_{k \in \mathcal K_{\mathcal F}} \sqrt{\ell_k(\ba_k)}}
\leq \Exp{\sqrt{\left|\mathcal K_{\mathcal F}\right| \sum_{k \in \mathcal K_{\mathcal F}} \ell_k(\ba_k)}}
\leq \sqrt{K \Exp{\sum_{k \in \mathcal K_{\mathcal F}} \ell_k(\ba_k)}}.
\end{align*}
Note that~(\ref{eq:lsidhfjsdh}) implies
$\max_{(i, k) \in \mathcal F} \left\{\frac{\alpha_{i, k}}{\sqrt{\mu_{i, k}}}\right\} \leq \varphi(x) x$.
Therefore,
\begin{align}\label{eq:kldsnfkjlsdnjfsd}
\Exp{\sum_{k \in \mathcal K_{\mathcal F}} \alpha_{i_k, k} \sqrt{a_{i_k, k}}}
\leq \varphi(x) x \cdot \sqrt{K \Exp{\sum_{k \in \mathcal K_{\mathcal F}} \ell_k(\ba_k)}}
\leq \varphi(x) \cdot \left(\frac{x^2 K}{4} + \Exp{\sum_{k \in \mathcal K_{\mathcal F}} \ell_k(\ba_k)}\right).
\end{align}
Here, we use Young's inequality $ab \leq \frac{a^2}{2 \varepsilon} + \frac{\varepsilon b^2}{2}$ with
$\varepsilon = 2$.
On the slow rounds, we get
\begin{align}\label{eq:klsdfsdcgdjfsd}
\Exp{\sum_{k \in \mathcal K_{\mathcal S}} \alpha_{i_k, k} \sqrt{a_{i_k, k}}}
\leq \Exp{\sum_{k \in \mathcal K_{\mathcal S}} \frac{\alpha_{i_k, k}}{\mu_{i_k, k}} \cdot \ell_k(\ba_k)}
&\leq \max_{(i, k) \in \mathcal S} \left\{\frac{\alpha_{i, k}}{\mu_{i, k}}\right\}
   \cdot \Exp{\sum_{k \in \mathcal K_{\mathcal S}} \ell_k(\ba_k)} \nonumber \\
&\leq \varphi(x) \cdot \Exp{\sum_{k \in \mathcal K_{\mathcal S}} \ell_k(\ba_k)}.
\end{align}
By the definition of $\regretdyn$ and Jensen's inequality,
\begin{align}\label{eq:lsdfkjsndfkj}
\sum_{k = 1}^K \Exp{\ell_k(\ba_k)}
= \sum_{k = 1}^K \ell_k(\bar{\ba}_k) + \regretdyn
\leq \sum_{k = 1}^K \ell_k(\ba_k^*) + \regretdyn
\leq C_K^* + \regretdyn.
\end{align}
Combining~(\ref{eq:kldsnfkjlsdnjfsd}),~(\ref{eq:klsdfsdcgdjfsd}), and~(\ref{eq:lsdfkjsndfkj}),
and setting $x = \sqrt{\nicefrac{\left(C_K^* + \regretdyn\right)}{K}}$, so that $\varphi(x) = \varphi$, we get
\begin{align*}
\sum_{k = 1}^K \Exp{\max_{i \in [n]} \left\{\alpha_{i, k} \sqrt{a_{i, k}}\right\}}
&\leq \varphi \cdot \left(\frac{x^2 K}{4} + \Exp{\sum_{k \in \mathcal K_{\mathcal F}} \ell_k(\ba_k)}\right)
   + \varphi \cdot \Exp{\sum_{k \in \mathcal K_{\mathcal S}} \ell_k(\ba_k)} \\
&= \varphi \left(\frac{x^2 K}{4} + \Exp{\sum_{k = 1}^K \ell_k(\ba_k)}\right)
\leq \frac{5}{4} \varphi \left(C_K^* + \regretdyn\right).
\end{align*}
Substituting this bound and~(\ref{eq:lsdfkjsndfkj}) into~(\ref{eq:ksladsfssdf}) gives
\begin{align*}
\sum_{k = 1}^K \Exp{\max_{i \in \supp(\ba_k)} \sum_{u = 1}^{a_{i, k}} X_{i, k}^{(u)}}
&\leq \left(1 + \varphi \left(1 + \ln \min\left\{n, B\right\}\right) \cdot \frac{5}{4}\right)
   \left(C_K^* + \regretdyn\right) \\
&\leq \left(1 + 4 \varphi \ln \min\left\{n, B\right\}\right)
   \left(C_K^* + \regretdyn\right),
\end{align*}
where the last inequality uses $\nicefrac{5}{4}(1 + \ln \min\left\{n, B\right\}) \leq 4 \ln \min\left\{n, B\right\}$ 
   for $n \geq 2$.
\end{proof}

\section{Sub-exponential properties}
\subsection{Confidence Intervals}\label{sec:confidence_interval}
\begin{lemma}\label{lemma:different_sigma}
Let $Y_1,\dots,Y_N$ be random variables adapted to a filtration
$(\mathcal F_j)_{j=0}^N$. Suppose that, for some $\alpha_j>0$,
\begin{align*}
\E\left[Y_j\mid\mathcal F_{j-1}\right]=0,
\qquad
\E\left[\exp(\lambda Y_j)\mid\mathcal F_{j-1}\right]
\leq
\exp\left(\lambda^2\alpha_j^2\right)
\end{align*}
for every $j\in[N]$ and every $|\lambda|\leq1/\alpha_j$.
Let $A_j\in\{0,1\}$ be $\mathcal F_{j-1}$-measurable indicators, and define
\begin{align*}
Z=\sum_{j=1}^N A_jY_j,
\qquad
M=\sum_{j=1}^N A_j,
\qquad
\alpha=\max_{j\in[N]}\alpha_j.
\end{align*}
Then, for every $m\in[N]$ and every $\varepsilon>0$,
\begin{align*}
\Prob{\left|\frac{Z}{m}\right|\geq\varepsilon,\ M=m}
\leq
2\exp\left(
-\frac{m}{4}
\min\left\{
\frac{\varepsilon^2}{\alpha^2},
\frac{2\varepsilon}{\alpha}
\right\}
\right).
\end{align*}
\end{lemma}

\begin{proof}
For every $|\lambda|\leq1/\alpha$, since $A_j$ is
$\mathcal F_{j-1}$-measurable, we have
\begin{align*}
\E\left[\exp(\lambda A_jY_j)\mid\mathcal F_{j-1}\right]
&=
\cI\{A_j=0\}
+
\cI\{A_j=1\}
\E\left[\exp(\lambda Y_j)\mid\mathcal F_{j-1}\right] \\
&\leq
\cI\{A_j=0\}
+
\cI\{A_j=1\}\exp\left(\lambda^2\alpha_j^2\right)
=
\exp\left(\lambda^2\alpha_j^2A_j\right).
\end{align*}
Therefore,
\begin{align*}
\E\left[
\exp\left(
\lambda A_jY_j-\lambda^2\alpha_j^2A_j
\right)
\Bigm|
\mathcal F_{j-1}
\right]
\leq1.
\end{align*}
By repeatedly applying the tower property, we obtain
\begin{align*}
\E\left[
\exp\left(
\lambda Z-\lambda^2\sum_{j=1}^N\alpha_j^2A_j
\right)
\right]
\leq1.
\end{align*}
Since $\sum_{j=1}^N\alpha_j^2A_j\leq\alpha^2M$, it follows that
\begin{align*}
\E\left[
\exp\left(
\lambda Z-\lambda^2\alpha^2M
\right)
\right]
\leq1.
\end{align*}

For $\lambda\in(0,1/\alpha]$, on the event
$\{Z\geq m\varepsilon,\ M=m\}$, we have
\begin{align*}
\lambda Z-\lambda^2\alpha^2M
\geq
\lambda m\varepsilon-\lambda^2\alpha^2m.
\end{align*}
Hence,
\begin{align*}
\Prob{Z\geq m\varepsilon,\ M=m}
\leq
\exp\left(
\lambda^2\alpha^2m-\lambda m\varepsilon
\right).
\end{align*}
If $\varepsilon\leq2\alpha$, setting
$\lambda=\varepsilon/(2\alpha^2)$, we obtain
\begin{align*}
\Prob{Z\geq m\varepsilon,\ M=m}
\leq
\exp\left(
-\frac{m\varepsilon^2}{4\alpha^2}
\right).
\end{align*}
If $\varepsilon>2\alpha$, setting $\lambda=1/\alpha$, we obtain
\begin{align*}
\Prob{Z\geq m\varepsilon,\ M=m}
&\leq
\exp\left(
m-\frac{m\varepsilon}{\alpha}
\right) \\
&\leq
\exp\left(
-\frac{m\varepsilon}{2\alpha}
\right).
\end{align*}
Consequently,
\begin{align*}
\Prob{Z\geq m\varepsilon,\ M=m}
\leq
\exp\left(
-\frac{m}{4}
\min\left\{
\frac{\varepsilon^2}{\alpha^2},
\frac{2\varepsilon}{\alpha}
\right\}
\right).
\end{align*}
Applying the same argument to $-Z$ and taking a union bound over the
upper and lower tails proves the result.
\end{proof}
\begin{lemma}[Confidence intervals]\label{lemma:sufficient-beta}
Under Assumption~\ref{ass:drift-exp}, we have
\begin{align*}
\Prob{\mathcal E_k}\leq\frac{n}{k^2},
\qquad
\sum_{k=1}^K\Prob{\mathcal E_k}\leq2n,
\end{align*}
where the events are defined in~(\ref{eq:bad_events}).
\end{lemma}

\begin{proof}
Fix a worker $i$ and a round $k$. If $\cW_k=\varnothing$, then
$\mathcal E_{i,k}=\varnothing$, and the result is immediate.
Suppose that $\cW_k\neq\varnothing$.

For $q\in\cW_k$ and $u\in[B]$, define
$Y_q^{(u)}=X_{i,q}^{(u)}-\mu_{i,q}$ and
$A_q^{(u)}=\cI\{u\leq a_{i,q}\}$.
Since the allocation $\ba_q$ is chosen before the observations in round $q$,
the indicator $A_q^{(u)}$ is $\mathcal F_{q,u-1}$-measurable. Moreover,
\begin{align*}
N_{i,k}^{(W)}
=
\sum_{q\in\cW_k}\sum_{u=1}^B A_q^{(u)}
\leq BW.
\end{align*}

By Assumption~\ref{ass:drift-exp}, the preceding MGF lemma, and the
independence of $Y_q^{(u)}$ and $\mathcal F_{q,u-1}$, we have
\begin{align*}
\E\left[
\exp\left(\lambda Y_q^{(u)}\right)
\Bigm|
\mathcal F_{q,u-1}
\right]
\leq
\exp\left(\lambda^2\alpha^2\right)
\end{align*}
for every $|\lambda|\leq1/\alpha$.

On the event $N_{i,k}^{(W)}=m$, we have
\begin{align*}
\hat\mu_{i,k}^{(W)}-\bar\mu_{i,k}^{(W)}
=
\frac{1}{m}
\sum_{q\in\cW_k}\sum_{u=1}^B A_q^{(u)}Y_q^{(u)}.
\end{align*}
For every $m\geq1$, define
\begin{align*}
\varepsilon_m
= 2\alpha\max\left\{\sqrt{\frac{\log(2B k^3)}{m}}, \frac{\log(2B k^3)}{m}
\right\}.
\end{align*}
Therefore, using a union bound over all possible values of
$N_{i,k}^{(W)}$, we obtain
\begin{align*}
\Prob{\mathcal E_{i,k}}
&\leq \sum_{m=1}^{B|\cW_k|} \Prob{\left|\frac{1}{m}\sum_{q\in\cW_k}\sum_{u=1}^B A_q^{(u)}Y_q^{(u)}\right|
> \varepsilon_m, \quad N_{i,k}^{(W)}=m}.
\end{align*}
Then, $\min\left\{\frac{\varepsilon_m^2}{\alpha^2}, \frac{2\varepsilon_m}{\alpha}\right\}
\geq \frac{4\log(2B k^3)}{m}$.
Enumerating the pairs $(q,u)$ in chronological order and applying Lemma~\ref{lemma:different_sigma}, we obtain
\begin{align*}
&\Prob{\left|\frac{1}{m} \sum_{q\in\cW_k}\sum_{u=1}^B A_q^{(u)}Y_q^{(u)}\right|
   > \varepsilon_m, \quad N_{i,k}^{(W)}=m}
\leq 2\exp\left(-\frac{m}{4} \min\left\{\frac{\varepsilon_m^2}{\alpha^2},
\frac{2\varepsilon_m}{\alpha}\right\}\right) 
\leq 2\exp\left(-\log(2Bk^3)\right).
\end{align*}
Since $|\cW_k| \leq k - 1$, we get
\begin{align*}
\Prob{\mathcal E_{i,k}}
&\leq
2B|\cW_k|\exp\left(-\log(2Bk^3)\right)
\leq 2B|\cW_k|\exp\left(-\log(2Bk^3)\right)
\leq \frac{1}{k^2}.
\end{align*}
Taking a union bound over all workers, we obtain
\begin{align*}
\Prob{\mathcal E_k}
= \Prob{\bigcup_{i=1}^n\mathcal E_{i,k}}
\leq \sum_{i=1}^n\Prob{\mathcal E_{i,k}} \leq
\frac{n}{k^2}
\quad \text{and} \quad
\sum_{k=1}^K\Prob{\mathcal E_k}
\leq n\sum_{k=1}^{\infty}\frac{1}{k^2}
\leq 2n.
\end{align*}
\end{proof}

\begin{corollary}\label{cor:confidence_interval}
On the good event $\neg\mathcal E_k$, for every worker $i$ such that
$N_{i,k}^{(W)}\geq1$, we have
\begin{align*}
\bar\mu_{i,k}^{(W)}
\in
\left[
s_{i,k},\:
\hat\mu_{i,k}^{(W)}+\conf_W(i,k)
\right].
\end{align*}
\end{corollary}

\begin{proof}
On the event $\neg\mathcal E_k$, we have
\begin{align*}
\hat\mu_{i,k}^{(W)}-\conf_W(i,k)
\leq
\bar\mu_{i,k}^{(W)}
\leq
\hat\mu_{i,k}^{(W)}+\conf_W(i,k).
\end{align*}
Moreover, $\bar\mu_{i,k}^{(W)}>0$ because $\mu_{i,q}>0$ for all $q$. Therefore,
\begin{align*}
s_{i,k}
=
\left(\hat\mu_{i,k}^{(W)}-\conf_W(i,k)\right)_+
\leq
\bar\mu_{i,k}^{(W)}.
\end{align*}
This proves the claim.
\end{proof}

\subsection{General properties}
The lemmas in this section are well known, but we include them to make the paper self-contained.
\begin{lemma}\label{lemma:orlicz_norm_mgf}
Assume that $X$ is a centered random variable and that
there exists an $\alpha > 0$ such that $\|X\|_{\psi_1} \leq \alpha$. 
Then, 
\begin{align*}
\Exp{\exp\left(\lambda X\right)} \leq \exp\left(\alpha^2 \lambda^2\right)
\quad \text{for all} \quad \left|\lambda\right| \leq \frac{1}{\alpha}.
\end{align*}
\end{lemma}
\begin{proof}
Using the Taylor expansion and $\Exp{X} = 0$, we get
\begin{align*}
\Exp{\exp\left(\lambda X\right)}
\leq \Exp{1 + \lambda X + \sum \limits_{p = 2}^{+\infty} \frac{\lambda^p X^p}{p!}} 
\leq \Exp{1 + \sum \limits_{p = 2}^{+\infty} \frac{\left(\alpha \left|\lambda\right|\right)^p \left(\frac{\left|X\right|}{\alpha}\right)^p}{p!}}.
\end{align*}
Consider $\left|\lambda\right| \leq \frac{1}{\alpha}$. Since $\left(\alpha \left|\lambda\right|\right)^p 
   \leq \left(\alpha \left|\lambda\right|\right)^2$
for $p \geq 2$, we deduce
\begin{align*}
\Exp{\exp\left(\lambda X\right)} 
&\leq 1 + \left(\alpha \left|\lambda\right|\right)^2 \Exp{\sum \limits_{p = 2}^{+\infty} 
   \frac{\left(\frac{\left|X\right|}{\alpha}\right)^p}{p!}}
\leq  1 + \left(\alpha \left|\lambda\right|\right)^2 
   \Exp{\exp\left(\frac{\left|X\right|}{\alpha}\right) - 1} \\
&\leq  1 + \alpha^2 \lambda^2 
\leq \exp\left(\alpha^2 \lambda^2\right).
\end{align*}
\end{proof}

\begin{lemma}\label{lemma:mgf_max_bound}
Let $X_1, \dots, X_n$ be a sequence of centered random variables such that for some $\alpha > 0$
\begin{align*}
\Exp{\exp\left(\lambda X_i\right)} \leq \exp\left(\alpha^2 \lambda^2\right)
\quad \text{for all} \quad
i \in [n] \quad \text{and} \quad
\left|\lambda\right| \leq \frac{1}{\alpha}.
\end{align*}
Then, $\Exp{\max_{i \in [n]} X_i} \leq 4 \alpha \ln n$.
\end{lemma}
\begin{proof}
For $n = 1$, the bound is trivial. Assume $n \geq 2$. Then,
\begin{align*}
\Exp{\max_{i \in [n]} X_i}
= \Exp{\frac{1}{\lambda} \ln \exp\left(\lambda \max_{i \in [n]} X_i\right)}
\leq \Exp{\frac{1}{\lambda} \ln\left(\sum_{i = 1}^n \exp\left(\lambda X_i\right)\right)}.
\end{align*}
By Jensen's inequality and the assumption, for all $0 < \lambda \leq \nicefrac{1}{\alpha}$,
we have
\begin{align*}
\Exp{\max_{i \in [n]} X_i}
\leq \frac{1}{\lambda} \ln\left(\Exp{\sum_{i = 1}^n \exp\left(\lambda X_i\right)}\right)
\leq \frac{1}{\lambda} \ln\left(n \exp\left(\alpha^2 \lambda^2\right)\right)
= \frac{\ln n}{\lambda} + \alpha^2 \lambda.
\end{align*}
The function on the right-hand side is convex and decreasing on
$\left(0, \nicefrac{\sqrt{\ln n}}{\alpha}\right]$.
Choosing $\lambda = \min\left\{\nicefrac{1}{\alpha}, \nicefrac{\sqrt{\ln n}}{\alpha}\right\}$, we get
\begin{align}\label{eq:iskldjflkdss}
\Exp{\max_{i \in [n]} X_i} \leq
\begin{cases}
2 \alpha \sqrt{\ln n}, & \ln n \leq 1, \\
\alpha \ln n + \alpha \leq 2 \alpha \ln n, & \ln n > 1.
\end{cases}
\end{align}
Note that for $n = 2$, we have $\sqrt{\ln n} \leq 2 \ln n$. This finishes the proof.
\end{proof}
\begin{remark}\label{remark:tighter_max_bound}
Actually, Lemma~\ref{lemma:mgf_max_bound} can be slightly improved. Using the bound $2 \sqrt{\ln 2} \leq \ln 2 + 1$
in (\ref{eq:iskldjflkdss}), we obtain $\Exp{\max_{i \in [n]} X_i} \leq \alpha \left(\ln n + 1\right)$.
\end{remark}

\begin{lemma}\label{lemma:expectation_via_orlicz}
Let $X$ be a non-negative, sub-exponential random variable. 
Denote $\alpha = \left\|X\right\|_{\psi_1}$ and $\Exp{X} = \mu$. Then
$\mu \leq \alpha\ln 2$.
\end{lemma}
\begin{proof}
Using Jensen's inequality and the non-negativity of $X$, we get
\begin{align*}
2 \geq \Exp{\exp\left(\frac{\left|X\right|}{\alpha}\right)}
\geq \exp\left(\frac{\Exp{\left|X\right|}}{\alpha}\right)
= \exp\left(\frac{\Exp{X}}{\alpha}\right)
= \exp\left(\frac{\mu}{\alpha}\right)
\end{align*}
Therefore, $\ln 2 \geq \nicefrac{\mu}{\alpha}$, which implies 
$\mu \leq \alpha \ln 2$.
\end{proof}
\section{RAS}\label{sec:ras}
In Algorithm~\ref{alg:sw_ata}, whenever $s_{i,k} > 0$ for all $i \in [n]$, we need to solve
$\min_{\ba \in \mathcal A} \max_{i\in[n]} a_i s_{i,k}$.
\citet{maranjyan2025ata} propose the Recursive Allocation Selection (RAS) algorithm for this purpose
and prove its correctness. We employ Algorithm~\ref{alg:ras} in SW-ATA.
\begin{algorithm}[t]
\caption{Recursive Allocation Selection (RAS)}
\label{alg:ras}
\begin{algorithmic}[1]
\STATE \textbf{Input:} LCBs $s_1,\ldots,s_n>0$, number of tasks $B$
\STATE Assume without loss of generality that $s_1 \leq s_2 \leq \cdots \leq s_n$
\IF{$B = 1$}
    \RETURN $(1,0,\ldots,0)$
\ENDIF
\STATE $\ba \leftarrow \operatorname{RAS}(s_1,\ldots,s_n;B-1)$
\IF{$a_n = 0$}
    \STATE $r \leftarrow \min\{i \in [n]: a_i=0\}$
\ELSE
    \STATE $r \leftarrow n$
\ENDIF
\STATE $\mathcal M \leftarrow
\argmin_{j \in [r]} \max_{i \in [n]} (a_i + \Ind{i=j})s_i$
\STATE Choose $j \in \mathcal M$ minimizing
$\left|\argmax_{i \in [r]}(a_i+\Ind{i=j})s_i\right|$
\RETURN $(a_1,\ldots,a_j+1,\ldots,a_n)$
\end{algorithmic}
\end{algorithm}

\section{Lower bounds}

\label{sec:lower_bounds}

We now turn to information-theoretic limits and derive a lower bound on the
regret for the problem at hand. The minimax regret is
batch-dependent: sparse ($B<n$) and large ($B\ge n$) batches exhibit
different dependence on $B$ and $n$ (\cref{thm:lower_bound}).

For a drifting instance $\nu$ of the model of \cref{ass:drift-exp}, with
its deterministic mean path $(\bmu_k)_{k=1}^K$, we write
\begin{equation}
\regretdyn_K(\pi,\nu)
:=
\ExpSub{\nu}{\sum_{k=1}^K \bigl(\ell_k(\ba_k) - L_k^*\bigr)}
\label{eq:proxy_loss_regret_policy}
\end{equation}
for the expected proxy-loss regret of a policy $\pi$. This is the
policy-and-instance form of the regret $\regretdyn$ defined in
\Cref{sec:setup}, with $L_k^*$ the optimal proxy loss of
round $k$. On a stationary instance, where $\mu_{i,k}=\mu_i$ for every
round, it is the expected proxy-loss regret against the fixed optimal
allocation. \Cref{subsec:lower_statement} states the bound and discusses
its relation to the upper bounds, and \Cref{app:lower_bound_proofs}
constructs the hard environments for the two batch regimes and proves the
bound.

\subsection{The environment class and the lower bound}
\label{subsec:lower_statement}

The lower bound is stated for the following class of environments.
\begin{definition}\label{def:lower_class}
For fixed $n$, $B$ and $K$, let $\mathfrak E_{n,B}^{[0,1]}(K,V_K)$ denote the
class of environments whose observations are positive almost surely and take
values in $[0,1]$, and whose mean paths satisfy the variation budget of
\cref{ass:variation_budget}.
\end{definition}
The hard instances constructed in the proof have two-point observations
on $\{\nicefrac12, 1\}$ or deterministic positive delays; in particular,
they satisfy \cref{ass:drift-exp}, so the lower bound below
applies to the same model as the upper bounds. The class is deliberately
restricted: it shows that the hardness already arises on simple bounded
instances. For
the class of \cref{def:lower_class} we obtain the following minimax bound
on the proxy-loss regret \eqref{eq:proxy_loss_regret_policy}.

\begin{theorem}[Batch-dependent minimax lower bound]
\label{thm:lower_bound}
Let $n\ge2$ and $B\ge1$. Then there is a universal constant $c>0$ such
that the following two bounds hold.

\noindent\emph{(i) Sparse batches.} If $B<n$, $K\ge n/B$, and
$V_K\lesssim BK/n$, then
\begin{equation}
\inf_{\pi}
\sup_{\nu\in\mathfrak E_{n,B}^{[0,1]}(K,V_K)}
\regretdyn_K(\pi,\nu)
\ge
c\left(
\sqrt{\frac{nK}{B}}
\ \vee\
\left(\frac nB\right)^{1/3}K^{2/3}V_K^{1/3}
\right).
\label{eq:lower_bound_sparse_batch}
\end{equation}

\noindent\emph{(ii) Large batches.} If $B\ge n$, $K\ge B/n$, and
$V_K\lesssim Kn^2/B^2$, then
\begin{equation}
\inf_{\pi}
\sup_{\nu\in\mathfrak E_{n,B}^{[0,1]}(K,V_K)}
\regretdyn_K(\pi,\nu)
\ge
c\left(
\sqrt{\frac{BK}{n}}
\ \vee\
\left(\frac Bn\right)^{2/3}K^{2/3}V_K^{1/3}
\right).
\label{eq:lower_bound_large_batch}
\end{equation}
\end{theorem}

\paragraph{Proof idea.}
The proof first builds hard stationary instances for the two batch regimes
and, for each regime, shows that the instances a policy must tell apart
have nearly identical distributions of the history at the available number of
observations. For $B<n$, all but one
of the $B$ tasks of a round are absorbed by fast dummy workers, and the
remaining task probes a pool of $m\ge n/2$ candidate workers, exactly
one of which is slightly faster than the others. Identifying the fast
candidate costs $\Omega(\sqrt{nK/B})$. For $B\ge n$, the optimal allocation
is balanced, and perturbing the weight of a single worker's delay law by a
small amount $\delta$ forces regret $\Omega(\sqrt{BK/n})$. The bounds under a variation
budget follow the restart construction of \citet{cheung2019learning}: the
horizon is divided into epochs, a fresh instance of the family is drawn
independently of the past at the beginning of every epoch, and the
near-identical history distributions force every policy to pay regret of
the stationary order over the epoch. The
mean change at each epoch boundary consumes part of the variation budget
$V_K$, and the epoch length is chosen to balance the per-epoch regret
against the number of epochs admissible under the budget.
All constructed instances have independent observations from two-point
distributions on $\{\nicefrac12,1\}$ or deterministic positive values in
$(0,1]$, which are positive almost surely and $\mathcal
O(1)$-sub-exponential.
The full proof is in
\Cref{app:lower_bound_proofs}.

The two conditions of each regime are read directly: the condition on $K$
asks for sufficiently many iterations relative to the batch imbalance, and
the condition on $V_K$ asks that the variation is not too large. Both
come from the construction and are not claimed to be necessary.
The two terms of each bound answer two different costs: the stationary term
is the price of identifying the fast workers under a fixed batch
constraint, and the variation term is the price of tracking them under the
budget. The tuned \algname{SW-ATA} attains the optimal dependence on the horizon
$K$ and on the variation budget $V_K$: the factor $K^{2/3}V_K^{1/3}$ is
common to the lower bound and, for large enough $V_K$, to the guarantee of
\cref{thm:general_ns_gap_independent}. The mismatch is confined to the
batch parameters: the guarantee of \cref{thm:general_ns_gap_independent}
carries the maximal powers of $B$ and $n$ appearing in the lower bound,
namely $B^{2/3}$ of \eqref{eq:lower_bound_large_batch} and $n^{1/3}$ of
\eqref{eq:lower_bound_sparse_batch}, and therefore covers both regimes at
once. Designing an algorithm that matches the lower bound in both
parameters simultaneously is an open question.

\subsection{Hard-instance constructions}
\label{app:lower_bound_proofs}

In all instances constructed below, the observations are positive almost
surely, bounded in
$[0,1]$, and they are independent across workers, tasks, and rounds.
Hence every constructed
instance whose mean path satisfies the variation constraint belongs to
$\mathfrak E_{n,B}^{[0,1]}(K,V_K)$.
Throughout this subsection, $\Lambda(p)$ denotes the two-point
distribution on $\{\nicefrac12,1\}$ with $\Prob{X=1}=p$, so that
$\Exp{X}=(1+p)/2$. Since the KL divergence between two such laws depends
only on the weights,
$\mathrm{KL}(\Lambda(p)\Vert\Lambda(q))
=\mathrm{KL}(\operatorname{Ber}(p)\Vert\operatorname{Ber}(q))$.

An environment $\nu$ is a family of delay distributions
$(\nu_{i,k})_{i\in[n],\,k\in[K]}$, one per worker per round. Together with
a policy $\pi$, it generates a distribution
$P_{\nu,\pi}$ over the history of the run, and the subscript of an
expectation names the law it is taken under. For brevity, the policy is
dropped from the notation: for the sparse family we write $P_\theta$ for
$P_{\nu_\theta,\pi}$, and $\ExpSub{\theta}{\cdot}$ for the expectation
under $P_\theta$, while $P_0$ is the law under the null environment
$\nu_0$ of the sparse family, with expectation
$\ExpSub{0}{\cdot}$. For the large-batch pair, $P_-$ and $P_+$ are the laws
under $\nu^-$ and $\nu^+$, with expectations $\ExpSub{-}{\cdot}$ and
$\ExpSub{+}{\cdot}$. The divergence decomposition for these laws is proved
in \cref{lem:kl_decomposition}.

To prove the lower bound, we construct for each of the two cases a family
of hard environments and bound the regret of every policy averaged over
the family. Since the average over a finite family never exceeds the
maximum, at least one environment of the family attains the same order,
which gives a lower bound on the worst case.

\subsubsection{Divergence decomposition}
\label{app:kl_decomposition}

We now make precise the probability space on which the laws
$P_{\nu,\pi}$ live. A policy $\pi=(\pi_k)_{k=1}^K$ is a sequence of
measurable rules, where $\pi_k$ maps the observed history of the first
$k-1$ rounds and a randomization seed to an allocation
$\ba_k\in\mathcal A$. The canonical space is
\[
\Omega
=
[0,1]\times\prod_{k=1}^K\Bigl(\mathcal A\times[0,1]^{nB}\Bigr),
\]
equipped with the product $\sigma$-algebra, where $\mathcal A$ carries the
counting $\sigma$-algebra and $[0,1]$ the Borel one. The coordinate
$U\in[0,1]$ is the randomization seed of the policy, with the uniform
law in every environment. The interconnection of $\pi$ with an
environment $\nu$ generates the history recursively: at round $k$ the
allocation is $\ba_k=\pi_k(\mathcal T_{k-1},U)$, where $\mathcal T_{k-1}$
is the observed history of the first $k-1$ rounds, the revealed delays
$X_{i,k}^{(1)},\dots,X_{i,k}^{(a_{i,k})}$ of worker $i$ are drawn
independently from $\nu_{i,k}$, and the unrevealed entries
$X_{i,k}^{(u)}$ with $u>a_{i,k}$ are drawn from the point mass at $0$,
identically in every environment. The law of the resulting element of
$\Omega$ is denoted $P_{\nu,\pi}$. The observed
history $\tau_K=(\ba_k,\{X_{i,k}^{(u)}\}_{i\in[n],\,u\le
a_{i,k}})_{k=1}^K$ is a measurable projection of $\Omega$.

\begin{lemma}[Divergence decomposition]
\label{lem:kl_decomposition}
Let $\nu=(\nu_{i,k})$ and $\nu'=(\nu'_{i,k})$ be two environments, and
assume that $\nu_{i,k}$ is absolutely
continuous with respect to $\nu'_{i,k}$ for every $i\in[n]$ and
$k\in[K]$. Fix a policy $\pi$, and let $P_{\nu,\pi}$ and $P_{\nu',\pi}$
be the probability measures on the canonical space induced by the
$K$-round interconnection of $\pi$ with $\nu$ and with $\nu'$. Then
\begin{equation}
\mathrm{KL}(P_{\nu,\pi}\,\|\,P_{\nu',\pi})
=
\sum_{k=1}^K\sum_{i=1}^n
\ExpSub{\nu}{a_{i,k}}\,\mathrm{KL}(\nu_{i,k}\,\|\,\nu'_{i,k}).
\label{eq:kl_decomposition}
\end{equation}
\end{lemma}
\begin{proof}
This is the standard divergence decomposition for bandit histories (see,
e.g., \citet[Chapter~15]{lattimore2019bandit}); we prove it in our notation.
By the chain rule for the KL divergence, applied over the rounds of the
canonical space,
\[
\mathrm{KL}(P_{\nu,\pi}\,\|\,P_{\nu',\pi})
=
\sum_{k=1}^K
\ExpSub{\nu}{
\mathrm{KL}\bigl(
P_{\nu,\pi}(\cdot\mid U,\mathcal T_{k-1})
\,\big\|\,
P_{\nu',\pi}(\cdot\mid U,\mathcal T_{k-1})
\bigr)
},
\]
where the $k$-th term is the conditional divergence of round $k$ given the
seed and the past history.

Conditional on $(U,\mathcal T_{k-1})$, the allocation $\ba_k$ is the same
deterministic function in both environments, so it contributes zero to
the conditional divergence. Given the allocation, the delays of round $k$
are independent across workers and across draws: worker $i$ contributes
$a_{i,k}$ revealed draws from $\nu_{i,k}$ under $\nu$ and from
$\nu'_{i,k}$ under $\nu'$, and the unrevealed entries have the same dummy
law in both environments and contribute zero. The divergence of a product
of independent draws is the sum of the per-draw divergences. Hence
\[
\mathrm{KL}\bigl(
P_{\nu,\pi}(\cdot\mid U,\mathcal T_{k-1})
\,\big\|\,
P_{\nu',\pi}(\cdot\mid U,\mathcal T_{k-1})
\bigr)
=
\sum_{i=1}^n
a_{i,k}\,\mathrm{KL}(\nu_{i,k}\,\|\,\nu'_{i,k}).
\]
The expectation
over the past history then replaces $a_{i,k}$ by
$\ExpSub{\nu}{a_{i,k}}$ by the tower property of expectation, since the
divergences
$\mathrm{KL}(\nu_{i,k}\,\|\,\nu'_{i,k})$ are deterministic, and summing
over the rounds gives \eqref{eq:kl_decomposition}.
\end{proof}
In particular, when the delay laws of the two environments differ
only on a subset $S$ of the workers, the inner sum in
\eqref{eq:kl_decomposition} reduces to $S$.

\subsubsection{Sparse batches ($B<n$)}
Let $B<n$, and put
\[
d:=\min\{B-1,\lfloor n/2\rfloor\},
\qquad
m:=n-d.
\]
When $d>0$, define the vector $\bm c=(c_1,\dots,c_d)$ by
\[
q:=\left\lfloor\frac{B-1}{d}\right\rfloor,
\qquad
r:=(B-1)-dq,
\qquad
c_j:=
\begin{cases}
q+1, & j\le r,\\
q, & j>r.
\end{cases}
\]
Since $d\le B-1$, the entries of $\bm c$ are positive integers; they differ
by at most one and sum to $B-1$. We distinguish two classes of workers. We
call workers $1,\dots,d$ the \emph{dummy workers} and define their delays
from the vector $\bm c$: dummy worker $j$ has the deterministic
delay $\mu_j=\frac{3}{4(c_j+1)}$. We call the remaining $m$ workers
the \emph{candidates}. When $B=1$ the dummy set is empty and all $n=m$
workers are candidates.

We now define the family of environments, indexed by $\theta\in[m]$. In
the environment $\nu_\theta$, candidate $\theta$ has
$\Lambda(1/2-\delta)$ delays,
and every other candidate has $\Lambda(1/2)$ delays, where for a sufficiently small universal $c_0>0$ we take
\begin{equation}
\delta
:=
c_0\min\left\{1,\sqrt{\frac{m}{BK}}\right\}.
\label{eq:lower_sparse_stationary_delta}
\end{equation}
We also write $\nu_0$ for the environment of the same form in which every
candidate has $\Lambda(1/2)$ delays.

The next lemma shows that the environment $\nu_\theta$ has a unique
optimal allocation, and that deviating from it incurs instantaneous regret
at least $\delta/2$.

\begin{lemma}[Hard family for sparse batches]
\label{lem:lower_sparse_construction}
In environment $\nu_\theta$, the unique optimal allocation assigns
$c_j$ tasks to dummy worker $j$ and one task to candidate $\theta$. Its
loss is $3/4-\delta/2$, and every nonoptimal action has instantaneous regret
at least $\delta/2$.
\end{lemma}
\begin{proof}
The construction has $c_j\le2$ for every dummy worker. We distinguish two
cases.

\textbf{Case $B-1\le\lfloor n/2\rfloor$.} Then $d=B-1$; for $d>0$ this
gives $q=1$ and $r=0$, so every $c_j=1$, and for $B=1$ there are no dummy
workers at all.

\textbf{Case $B-1>\lfloor n/2\rfloor$.} Then $d=\lfloor n/2\rfloor$, and the
regime condition $B\le n-1$ gives
$B-1\le n-2<2\lfloor n/2\rfloor=2d$, hence
$c_j\le\lceil(B-1)/d\rceil\le2$.

In both cases $c_j\le2$, and since $\delta\le c_0$ by
\eqref{eq:lower_sparse_stationary_delta}, choosing $c_0\le1/6$ gives
$\delta\le1/6\le 3/(2(c_j+1))$ for every dummy worker.

In $\nu_\theta$, the allocation described in the statement assigns $c_j$
tasks to dummy worker $j$ and one task to candidate $\theta$. The load of
candidate $\theta$ is $3/4-\delta/2$, the load of dummy worker $j$ is
\begin{equation*}
c_j\mu_j
=
\frac{3c_j}{4(c_j+1)}
=
\frac{3}{4}-\frac{3}{4(c_j+1)}
\le
\frac{3}{4}-\frac{\delta}{2},
\end{equation*}
since $\delta\le 3/(2(c_j+1))$ by the choice of $c_0\le1/6$ above, and
every other candidate besides $\theta$ carries no task. The
loss of the allocation is therefore $3/4-\delta/2$.

Any other allocation has loss at least $3/4$. Indeed, the total number of
tasks is fixed at $B=\sum_jc_j+1$, and an allocation that assigns no task
to any candidate other than $\theta$, at most one task to candidate
$\theta$, and at most $c_j$ tasks to every dummy worker $j$ places at most
$B$ tasks, with equality only for the allocation above. Every other
allocation therefore falls into one of three cases: it assigns a task to a
candidate other than $\theta$, it assigns at least two tasks to candidate
$\theta$, or it assigns more than $c_j$ tasks to some dummy worker $j$.
In the first case, that candidate's load is at least $3/4$. In the second
case, the load of candidate $\theta$ is at least $2(3/4-\delta/2)\ge3/4$,
since $\delta\le1/6$. In the third case, the load of dummy worker $j$ is
at least $(c_j+1)\mu_j=3/4$. The optimum is therefore unique, with loss
$3/4-\delta/2$, and every nonoptimal action has instantaneous regret at
least $\delta/2$.
\end{proof}

Recall that
$N_{i,k}$ denotes the number of observations collected from worker $i$ up
to round $k$, as in the stationary analysis, so that $N_{i,K+1}$ is the
total over the horizon. The next lemma shows that the laws of the history
under $\nu_\theta$ and $\nu_0$ are close at the available number of
observations.

\begin{lemma}[KL bound for the sparse family]
\label{lem:lower_sparse_kl}
Retain the sparse family defined above.
Then
\begin{equation}
\mathrm{KL}(P_0\Vert P_\theta)
\le
4\delta^2\ExpSub{0}{N_{\theta,K+1}}
\label{eq:lower_sparse_kl}
\end{equation}
\end{lemma}
\begin{proof}
By the divergence decomposition of \cref{lem:kl_decomposition}, only the
workers whose delay laws differ under the two environments
contribute to $\mathrm{KL}(P_0\Vert P_\theta)$. The dummy workers are
deterministic and identical under $P_0$
and $P_\theta$, so their divergences vanish. Among the candidates, only
candidate $\theta$ has different distributions under the two laws, so a
single term survives:
\[
\mathrm{KL}(P_0\Vert P_\theta)
=
\ExpSub{0}{N_{\theta,K+1}}\,
\mathrm{KL}(\Lambda(1/2)\Vert\Lambda(1/2-\delta)),
\]
The divergence of the two-point laws equals the Bernoulli divergence of
their weights,
$\mathrm{KL}(\operatorname{Ber}(1/2)\Vert\operatorname{Ber}(1/2-\delta))
=-\tfrac12\log(1-4\delta^2)\le 4\delta^2$ for $\delta\le1/4$, and
inserting this estimate gives \eqref{eq:lower_sparse_kl}.
\end{proof}

\begin{theorem}[Stationary lower bound for sparse batches]
\label{thm:lower_sparse_stationary}
Retain the sparse family defined above.
Then, for $\delta$ of \eqref{eq:lower_sparse_stationary_delta} with
$c_0\le1/6$,
\begin{equation}
\inf_\pi\sup_{\nu\in\mathfrak E_{n,B}^{[0,1]}(K,0)}
\regretdyn_K(\pi,\nu)
\ge
c\min\left\{K,\sqrt{\frac{nK}{B}}\right\},
\qquad B<n.
\label{eq:lower_sparse_stationary}
\end{equation}
\end{theorem}
\begin{proof}
Write $C_\theta$ for the number of rounds in which the learner selects the
optimal allocation of the environment $\nu_\theta$, and fix
$\theta\in[m]$. Since $0\le C_\theta\le K$, the difference of the means
is at most $K$ times the total variation distance, and Pinsker's inequality
gives
\[
\ExpSub{\theta}{C_\theta}
\le
\ExpSub{0}{C_\theta}
+K\,\mathrm{TV}(P_0,P_\theta)
\le
\ExpSub{0}{C_\theta}
+K\sqrt{\frac12\mathrm{KL}(P_0\Vert P_\theta)}.
\]
Averaging this inequality over $\theta\in[m]$, we get
\[
\frac1m\sum_{\theta=1}^m\ExpSub{\theta}{C_\theta}
\le
\frac1m\sum_{\theta=1}^m\ExpSub{0}{C_\theta}
+\frac{K}{m}\sum_{\theta=1}^m\sqrt{\frac12\mathrm{KL}(P_0\Vert P_\theta)}.
\]
The
first term is at most $K/m$: the optimal allocations associated with
distinct $\theta$ are distinct, so at most one of them is selected in each
round, which gives $\sum_{\theta=1}^mC_\theta\le K$ pathwise, and hence
$\sum_{\theta=1}^m\ExpSub{0}{C_\theta}\le K$. For the second term, inserting
\eqref{eq:lower_sparse_kl} and then applying Jensen's inequality to the
concave function $x\mapsto\sqrt x$ gives
\[
\frac1m\sum_{\theta=1}^m
\sqrt{\mathrm{KL}(P_0\Vert P_\theta)}
\le
2\delta\,
\frac1m\sum_{\theta=1}^m
\sqrt{\ExpSub{0}{N_{\theta,K+1}}}
\le
2\delta\,
\sqrt{\frac1m\sum_{\theta=1}^m\ExpSub{0}{N_{\theta,K+1}}}
\le
2\delta\,\sqrt{\frac{BK}{m}},
\]
where the last step uses $\sum_{\theta=1}^mN_{\theta,K+1}\le BK$, which
holds pathwise because every observation is collected from a single
candidate and the learner assigns $B$ tasks per round. Combining the two
terms yields
\begin{equation}
\frac1m\sum_{\theta=1}^m\ExpSub{\theta}{C_\theta}
\le
\frac Km
+\sqrt2\,K\delta\sqrt{\frac{BK}{m}}
\le
\frac K2+\frac K4.
\label{eq:lower_sparse_correct_actions}
\end{equation}
Here the first inequality is from the averaging above, and the second
inequality holds term by term: $K/m\le K/2$ because $m\ge2$, and
$\sqrt2\,\delta\sqrt{BK/m}\le\sqrt2\,c_0\le1/4$ by
\eqref{eq:lower_sparse_stationary_delta} with $c_0\le1/6$.
Every round in which the learner does not select the optimal allocation of
$\nu_\theta$ contributes instantaneous regret at least $\delta/2$ by
\cref{lem:lower_sparse_construction}, so averaging uniformly over
$\theta\in[m]$ and inserting \eqref{eq:lower_sparse_correct_actions} yields
an average regret of at least
\[
\frac{\delta}{2m}
\sum_{\theta=1}^m
\left(K-\ExpSub{\theta}{C_\theta}\right)
\ge \frac{K\delta}{8}.
\]
Every environment of the family $\{\nu_\theta\}$ is stationary and hence
belongs to $\mathfrak E_{n,B}^{[0,1]}(K,0)$, as noted at the beginning of
this subsection. The average regret over the family is
at least $K\delta/8$ for every policy, so at least one environment of the
family attains it, which gives
\[
\inf_\pi\sup_{\nu\in\mathfrak E_{n,B}^{[0,1]}(K,0)}
\regretdyn_K(\pi,\nu)
\ge
\frac{K\delta}{8}.
\]
Since $d\le\lfloor n/2\rfloor$, the number of candidates satisfies
$m=n-d\ge n/2$. Substituting this bound and the choice of $\delta$ from
\eqref{eq:lower_sparse_stationary_delta} yields
\[
K\delta
=
c_0K\min\left\{1,\sqrt{\frac{m}{BK}}\right\}
=
c_0\min\left\{K,\sqrt{\frac{mK}{B}}\right\}
\ge
c_0\min\left\{K,\sqrt{\frac{nK}{2B}}\right\},
\]
and absorbing the constant factors into the universal constant gives
\eqref{eq:lower_sparse_stationary}.
\end{proof}

\subsubsection{Large batches ($B\ge n$)}

The hard instances of the large-batch regime are two environments that differ
only in the mean of worker $1$. Let $n\ge2$ and $B\ge n$, and define the balanced
allocation $\bm{b}=(b_1,\dots,b_n)\in\mathcal A$ by
\[
q:=\left\lfloor\frac Bn\right\rfloor,
\qquad
r:=B-nq,
\qquad
b_1:=q,
\qquad
b_i:=
\begin{cases}
q+1, & 2\le i\le r+1,\\
q, & i>r+1.
\end{cases}
\]
For a sufficiently small universal $c_0$,
set
\begin{equation}
\delta
:=
c_0\min\left\{\frac1{b_1},\frac1{\sqrt{b_1K}}\right\}
\label{eq:lower_dense_stationary_delta}
\end{equation}
and put
\[
p:=\frac14,
\qquad
L:=\frac{5b_1}{8},
\qquad
h:=\frac{b_1\delta}{2}.
\]
Here $p$ is the unperturbed weight parameter of the two-point delay law of
worker $1$, whose unperturbed mean is $(1+p)/2=5/8$, while $L=b_1(1+p)/2$
is its load under the allocation $\bm{b}$ at that mean, and $h$ is the gap
by which the two environments below shift the loss of $\bm{b}$.
We now define the pair of environments. Under
$\nu^-$, worker $1$ has $\Lambda(p-\delta)$ delays (of mean
$5/8-\delta/2$), and
under $\nu^+$, it has $\Lambda(p+\delta)$ delays (of mean
$5/8+\delta/2$). Every
worker $i\ge2$ has deterministic delay
\begin{equation}
\mu_i=\frac{L}{b_i+1}
\label{eq:lower_dense_dummy_means}
\end{equation}
under both environments. Write $T_{\bm{b}}$ for the number of rounds in
which the allocation $\bm{b}$ is selected, $R_-$ and $R_+$ for the expected
regrets of the policy under $\nu^-$ and $\nu^+$, and recall that
$N_{1,K+1}$ denotes the total number of observations collected from worker
$1$ over the horizon. The next lemma shows that the two environments have
different optimal allocations, and records the resulting regret bounds.

\begin{lemma}[Hard pair for large batches]
\label{lem:lower_dense_construction}
Under $\nu^-$, the unique optimal allocation is
$\bm{b}$, with loss $L-h$, and every other allocation has gap at least $h$.
Under $\nu^+$, the allocation
$\bm{b}-\mathbf e_1+\mathbf e_2$ has loss at most $L$, whereas $\bm{b}$ has
loss $L+h$, so selecting $\bm{b}$ under $\nu^+$ incurs regret at least $h$.
Consequently,
\begin{equation}
R_-\ge h\left(K-\ExpSub{-}{T_{\bm{b}}}\right),
\qquad
R_+\ge h\ExpSub{+}{T_{\bm{b}}},
\label{eq:lower_dense_testing_regrets}
\end{equation}
and
\begin{equation}
\ExpSub{-}{N_{1,K+1}}\le b_1K+6R_-.
\label{eq:lower_dense_self_bound}
\end{equation}
\end{lemma}
\begin{proof}
Since $\delta\le c_0/b_1$ by \eqref{eq:lower_dense_stationary_delta} and
$b_i\le b_1+1$ for every $i$, choosing $c_0\le1/16$ gives
$\delta\le 1/(16b_1)\le 5/(4(b_i+1))$ for every $i$.

Under $\nu^-$, the load of worker $1$ under $\bm{b}$ is
$b_1(5/8-\delta/2)=L-h$, and the load of worker $i\ge2$ is
\begin{equation*}
b_i\mu_i
=
\frac{5b_1}{8}\,\frac{b_i}{b_i+1}
=
b_1\Bigl(\frac58-\frac{5}{8(b_i+1)}\Bigr)
\le
b_1\Bigl(\frac58-\frac\delta2\Bigr)
=
L-h
\end{equation*}
by the condition $\delta\le 5/(4(b_i+1))$
established above, so $\bm{b}$ has loss $L-h$. If $\ba\ne\bm{b}$, then some
coordinate exceeds its corresponding coordinate in $\bm{b}$. Exceeding
$b_i$ on worker $i\ge2$ produces load $(b_i+1)\mu_i=L$,
while exceeding $b_1$ on worker $1$ produces load
$(b_1+1)(5/8-\delta/2)\ge L$, since $\delta\le 5/(4(b_1+1))$. Hence
every other allocation has loss at least $L$ and gap at least $h$, and the
optimum is unique.

Under $\nu^+$, the load of worker $1$ under $\bm{b}$ is
$b_1(5/8+\delta/2)=L+h$, so $\bm{b}$ has loss $L+h$. Under
$\bm{b}-\mathbf e_1+\mathbf e_2$, worker $1$ carries $b_1-1$ tasks with
load $(b_1-1)(5/8+\delta/2)\le 5b_1/8=L$, since $(b_1-1)\delta\le5/4$
follows from $\delta\le 1/(16b_1)$, and worker
$2$ carries $b_2+1$ tasks with load $(b_2+1)\mu_2=L$, so this allocation
has loss at most $L$.

\eqref{eq:lower_dense_testing_regrets} follows from
the gaps established above. Under $\nu^-$, every round in which $\bm{b}$ is
not selected incurs regret at least $h$, which gives the first inequality.
Under $\nu^+$, every round in which $\bm{b}$ is selected incurs regret at
least $h$, which gives the second.
For \eqref{eq:lower_dense_self_bound}, note that the learner may allocate more than $b_1$
tasks to worker $1$, so $N_{1,K+1}$ cannot be bounded by $b_1K$ directly, and is
bounded through the regret instead. Under $\nu^-$, whenever $a_{1,k}>b_1$,
the load of worker $1$ is
$a_{1,k}(5/8-\delta/2)=(L-h)+(a_{1,k}-b_1)(5/8-\delta/2)$, so the round regret is
at least
\begin{equation*}
(a_{1,k}-b_1)\Bigl(\frac58-\frac\delta2\Bigr)\ \ge\ \frac{a_{1,k}-b_1}{6},
\end{equation*}
since $\delta\le c_0\le1/16$ gives $5/8-\delta/2\ge1/6$. Each round
therefore satisfies
$(a_{1,k}-b_1)_+\le6\bigl(\ell_k(\ba_k)-(L-h)\bigr)$, and summing over the iterations and separating the first $b_1$
tasks gives
\begin{equation*}
N_{1,K+1}
\ =\
\sum_{k=1}^Ka_{1,k}
\ \le\
b_1K+\sum_{k=1}^K(a_{1,k}-b_1)_+
\ \le\
b_1K+6\sum_{k=1}^K\bigl(\ell_k(\ba_k)-(L-h)\bigr),
\end{equation*}
and taking expectations under $\nu^-$ gives \eqref{eq:lower_dense_self_bound}.
\end{proof}

The next lemma shows that the laws of the history under $\nu^-$ and $\nu^+$
are close when worker $1$ is not observed too often.

\begin{lemma}[KL bound for the large-batch pair]
\label{lem:lower_dense_kl}
Retain the large-batch pair defined above.
Then
\begin{equation}
\mathrm{KL}(P_-\Vert P_+)
\le
32\delta^2\left(b_1K+6R_-\right).
\label{eq:lower_dense_kl}
\end{equation}
\end{lemma}
\begin{proof}
By the divergence decomposition of \cref{lem:kl_decomposition}, only the
workers whose delay laws differ under the two environments
contribute to $\mathrm{KL}(P_-\Vert P_+)$. Every worker $i\ge2$ of
\eqref{eq:lower_dense_dummy_means} is
deterministic with the same mean under both environments, so these workers
contribute nothing, and only worker $1$ remains:
\[
\mathrm{KL}(P_-\Vert P_+)
=
\ExpSub{-}{N_{1,K+1}}\,
\mathrm{KL}(\Lambda(p-\delta)\Vert\Lambda(p+\delta)).
\]
The divergence of the two-point laws equals the Bernoulli divergence of
their weights, which is at most the $\chi^2$ divergence,
\[
\mathrm{KL}(\operatorname{Ber}(p-\delta)\Vert\operatorname{Ber}(p+\delta))
\le
\frac{(2\delta)^2}{(p+\delta)(1-p-\delta)}
\le
32\delta^2,
\]
where the last step uses $p=1/4$ and $\delta\le1/8$, which give
$(p+\delta)(1-p-\delta)\ge\frac14\cdot\frac58=\frac5{32}$. The condition
$\delta\le1/8$ holds for $c_0\le1/8$ by
\eqref{eq:lower_dense_stationary_delta}. Inserting
$\ExpSub{-}{N_{1,K+1}}\le b_1K+6R_-$ of \eqref{eq:lower_dense_self_bound} gives
\eqref{eq:lower_dense_kl}.
\end{proof}

We now show that the two environments are hard to distinguish at the
available number of observations, which yields the stationary lower bound.

\begin{theorem}[Stationary lower bound for large batches]
\label{thm:lower_dense_stationary}
Let $n\ge2$, and retain the large-batch pair defined above.
Then, for $\delta$ of \eqref{eq:lower_dense_stationary_delta} with
$c_0\le1/16$,
\begin{equation}
\inf_\pi\sup_{\nu\in\mathfrak E_{n,B}^{[0,1]}(K,0)}
\regretdyn_K(\pi,\nu)
\ge
c\min\left\{K,\sqrt{\frac{BK}{n}}\right\},
\qquad B\ge n.
\label{eq:lower_dense_stationary}
\end{equation}
\end{theorem}
\begin{proof}
If $R_-\ge hK/8$, then
$(R_-+R_+)/2\ge R_-/2\ge hK/16$, and the average regret under the two
environments already has the desired order. If instead $R_-<hK/8$, the
KL estimate \eqref{eq:lower_dense_kl} gives
\[
\mathrm{KL}(P_-\Vert P_+)
\le
32\delta^2\left(b_1K+\frac{3hK}{4}\right)
\le
32c_0^2+24c_0^3
\le
\frac12,
\]
where the first step inserts $R_-<hK/8$ into \eqref{eq:lower_dense_kl},
the second step uses $\delta^2b_1K\le c_0^2$ and
$\delta^2hK=b_1\delta^3K/2\le c_0^3$, which follow from $h=b_1\delta/2$ and
the choice \eqref{eq:lower_dense_stationary_delta} of $\delta$, and the
third step holds by choosing $c_0\le1/16$. Since $0\le T_{\bm{b}}\le K$,
Pinsker's inequality gives
\[
\ExpSub{-}{T_{\bm{b}}}-\ExpSub{+}{T_{\bm{b}}}
\le
K\,\mathrm{TV}(P_-,P_+)
\le
K\sqrt{\mathrm{KL}(P_-\Vert P_+)/2}
\le
\frac K2.
\]
Averaging the two displays
\eqref{eq:lower_dense_testing_regrets} and inserting this bound gives
\[
\frac{R_-+R_+}{2}
\ge
\frac{h}{2}
\left(K-\ExpSub{-}{T_{\bm{b}}}+\ExpSub{+}{T_{\bm{b}}}\right)
\ge
\frac{hK}{4}.
\]
Both environments $\nu^-$ and $\nu^+$ are stationary and hence belong to
$\mathfrak E_{n,B}^{[0,1]}(K,0)$, as noted at the beginning of this
subsection. The average regret over the two
environments is at least $hK/16$ for every policy, so at least one of the
two environments attains it, which gives
\[
\inf_\pi\sup_{\nu\in\mathfrak E_{n,B}^{[0,1]}(K,0)}
\regretdyn_K(\pi,\nu)
\ge
\frac{hK}{16}.
\]
Since $b_1=\lfloor B/n\rfloor=\Theta(B/n)$ and $h=b_1\delta/2$ with $\delta$
from \eqref{eq:lower_dense_stationary_delta},
\[
hK
=
\frac{c_0}{2}K\min\left\{1,\sqrt{\frac{b_1}{K}}\right\}
=
\frac{c_0}{2}\min\left\{K,\sqrt{b_1K}\right\},
\]
and substituting $b_1=\Theta(B/n)$ gives \eqref{eq:lower_dense_stationary}.
\end{proof}

\subsubsection{Drifting environment}

It remains to assemble the proof of \cref{thm:lower_bound}. The stationary
bounds are supplied by the stationary theorem of each regime. Each of the two
restart paragraphs in the proof replays one of the stationary
configurations on independent epochs: the epoch-wise randomization is
drawn first, then the total variation it creates is bounded, and finally
the number of epochs is optimized against the budget $V_K$.

\begin{proof}[Proof of \cref{thm:lower_bound}]
The stationary lower bounds of
\cref{thm:lower_sparse_stationary,thm:lower_dense_stationary} read
\begin{equation*}
\inf_\pi\sup_{\nu\in\mathfrak E_{n,B}^{[0,1]}(K,0)}
\regretdyn_K(\pi,\nu)
\ge
c\begin{cases}
\min\bigl\{K,\sqrt{nK/B}\bigr\}, & B<n,\\[0.3em]
\min\bigl\{K,\sqrt{BK/n}\bigr\}, & B\ge n.
\end{cases}
\end{equation*}
It remains
to restart the stationary constructions on independent epochs under the
variation budget. The rest of the proof runs under the hypotheses of the
theorem: the conditions on $K$ and $V_K$ of the corresponding regime.

\paragraph{Independent restarts when $B<n$.}
We construct the non-stationary environment by restarting the family of
\cref{lem:lower_sparse_construction} on independent epochs.
Partition the horizon into $M$ epochs of equal length
$\ell:=\lfloor K/M\rfloor$, leaving a final incomplete block of the
remaining rounds. The proxy-loss regret is nonnegative in every round, and the final block
covers less than half of the horizon, so ignoring it costs only a
constant factor. At the
beginning of every epoch, draw the \emph{fast candidate} of the epoch ---
the index of the candidate whose delay law is $\Lambda(1/2-\delta_M)$
(whose mean is $3/4-\delta_M/2$) in that epoch ---
independently and uniformly from the $m$ candidates of
\cref{lem:lower_sparse_construction}, keep all
means fixed inside the epoch, and use the common gap
\begin{equation}
\delta_M
:=
c_0\min\left\{1,\sqrt{\frac{mM}{BK}}\right\}.
\label{eq:lower_sparse_epoch_delta}
\end{equation}
These draws define a distribution over deterministic mean paths:
conditioned on any realization, the whole mean path is fixed before the
policy is run. The fast candidate of an epoch is drawn independently of the history
before the epoch, so the past carries no information about which
environment of the family runs during the epoch. The gap $\delta_M$ of \eqref{eq:lower_sparse_epoch_delta} does not
exceed the gap that \cref{lem:lower_sparse_construction} prescribes for an
epoch of length $\ell$, since $\ell\le K/M$ gives
$\sqrt{m/(B\ell)}\ge\sqrt{mM/(BK)}$. The stationary lower bound of \cref{thm:lower_sparse_stationary}
applies conditionally in every epoch with the same constant $c_0$ and gives an
expected regret of at least $\ell\delta_M/8$ per epoch, averaged
over the choice of its fast candidate.
Summing over the epochs, the expected regret, averaged over the draws of
the fast candidate, is at least
\begin{equation}
\Exp{\regretdyn_K(\pi,\nu)}
\ \ge\
\frac18\,M\Bigl\lfloor\frac KM\Bigr\rfloor\,\delta_M
\ \ge\
c\,K\delta_M
\ \ge\
c\min\left\{K,\sqrt{\frac{mKM}{B}}\right\},
\label{eq:lower_sparse_epoch_regret}
\end{equation}
where the second step uses $M\lfloor K/M\rfloor\ge K-M\ge K/2$ for
$M\le K/2$ and absorbs the factor $1/8$, and the last step inserts
\eqref{eq:lower_sparse_epoch_delta} and multiplies both its branches by $K$.
At an epoch boundary, changing the fast candidate changes the mean vector
by at most $\delta_M$ in infinity norm. Hence every realized mean path has
total variation at most
\begin{equation}
M\delta_M.
\label{eq:lower_sparse_epoch_variation}
\end{equation}
Whenever this quantity is at most $V_K$, every realization is therefore a
deterministic path of the instance class. The regret of
\eqref{eq:lower_sparse_epoch_regret} is averaged over this distribution,
so at least one admissible deterministic path attains it.
It remains to optimize the epoch count $M$ subject to
\eqref{eq:lower_sparse_epoch_variation} being at most $V_K$. Put
$A=m/B$, so that $A\ge1/2$, and abbreviate
\[
M_0:=\left(\frac{V_K^2K}{A}\right)^{1/3},
\]
the epoch count at which the variation
\eqref{eq:lower_sparse_epoch_variation} balances the budget. The choice is
$M=\min\bigl\{\Theta\bigl(\max\{M_0,1\}\bigr),\,\lfloor K/2\rfloor\bigr\}$:
the cap keeps $\ell=\lfloor K/M\rfloor\ge2$ and the final block below half
of the horizon, and the implicit constant is small
enough that $M\le M_0$ when $M_0\gtrsim1$ and that it absorbs the factor
$c_0$ in the variation check below. When the cap binds, $M_0\gtrsim K$,
which under the hypothesis $V_K\lesssim K/A$ forces $A=\Theta(1)$ and
$V_K=\Theta(K)$; the target is then already of order $K$, and the capped
choice attains it, since \eqref{eq:lower_sparse_epoch_regret} gives
$\Theta(K\delta_M)=\Theta\bigl(K\sqrt{A}\bigr)=\Theta(K)$ there while the
variation consumed is $\Theta(c_0K)\le V_K$.

Consider first the case $M_0\lesssim1$, equivalently
$V_K\lesssim\sqrt{A/K}$. The choice gives $M=\Theta(1)$, and the
stationary bound \eqref{eq:lower_sparse_stationary}, valid for every
budget, dominates the non-stationary target. The case $A>K$ falls here as well:
the hypothesis $V_K\lesssim BK/n$ gives $M_0\lesssim K/A<1$ there, since
$A=\Theta(n/B)$.

Suppose henceforth that $M_0\gtrsim1$. Since $A=\Theta(n/B)$, the
hypothesis $V_K\lesssim BK/n$ reads $V_K\lesssim K/A$, which gives
$M_0\lesssim K/A$, hence $AM\lesssim K$. The gap of
\eqref{eq:lower_sparse_epoch_delta} therefore sits in its second branch,
$\delta_M=c_0\sqrt{AM/K}$, and the minimum in
\eqref{eq:lower_sparse_epoch_regret} attains its second entry.

For admissibility, inserting the definition of $M_0$ gives
\[
M\delta_M
=
c_0\sqrt{\frac AK}\,M^{3/2}
\le
c_0\sqrt{\frac AK}\,M_0^{3/2}
=
c_0V_K
\le
V_K,
\]
where the second step uses $M\le M_0$, and the last holds because the
implicit constant of the choice absorbs $c_0$. Hence every realized mean
path is admissible. If the choice returns $M=1$, which happens only when
$M_0=\Theta(1)$, admissibility follows directly, since
$M\delta_M=\delta_1=c_0\min\{1,\sqrt{A/K}\}$ and $M_0\gtrsim1$ gives
$V_K\gtrsim\sqrt{A/K}$.

For the regret evaluation, inserting the definition of $M_0$ into the
second entry of \eqref{eq:lower_sparse_epoch_regret} gives
\[
\sqrt{AKM}
=
\Theta\left(\sqrt{AKM_0}\right)
=
\Theta\bigl(A^{1/3}K^{2/3}V_K^{1/3}\bigr).
\]

Combining the regimes, and replacing $A^{1/3}$ by
$\Theta\bigl((n/B)^{1/3}\bigr)$ via $A=m/B$ and $m\ge n/2$, gives
\begin{equation}
\inf_\pi\sup_{\nu\in\mathfrak E_{n,B}^{[0,1]}(K,V_K)}
\regretdyn_K(\pi,\nu)
\ge
c\left(\frac nB\right)^{1/3}K^{2/3}V_K^{1/3},
\qquad B<n.
\label{eq:lower_sparse_nonstationary}
\end{equation}
The cap at $K$ is absent: under the hypothesis
$V_K\lesssim BK/n$ of the regime, this term does not exceed $K$.

\paragraph{Independent restarts when $B\ge n$.}
We do the same for the pair of \cref{lem:lower_dense_construction}.
Use $M$ epochs of equal length $\ell:=\lfloor K/M\rfloor$, with the
final incomplete block ignored as above, at the same absorbed constant
factor. At the beginning of every epoch, one of the two environments of
\cref{lem:lower_dense_construction} is chosen uniformly at random and
independently of the past: the delay law of worker $1$ is
$\Lambda(p-\delta_M)$ under one of them and $\Lambda(p+\delta_M)$ under
the other (of means $5/8\mp\delta_M/2$), with the common gap
\begin{equation}
\delta_M
:=
c_0\min\left\{
\frac1{b_1},
\sqrt{\frac{M}{b_1K}}
\right\}.
\label{eq:lower_dense_epoch_delta}
\end{equation}
As above, these choices specify a distribution over deterministic mean
paths.
Applying the stationary lower bound of \cref{thm:lower_dense_stationary}
conditionally in each epoch
(as in the sparse case, $\delta_M$ does not exceed the per-epoch choice
with horizon $\ell$) gives, for the expected regret averaged over the
epoch draws and with the same rounding step as above,
\begin{equation}
\Exp{\regretdyn_K(\pi,\nu)}
\ \ge\
c\,M\Bigl\lfloor\frac KM\Bigr\rfloor\,b_1\delta_M
\ \ge\
c\,Kb_1\delta_M
\ \ge\
c\min\left\{K,\sqrt{b_1KM}\right\},
\label{eq:lower_dense_epoch_regret}
\end{equation}
the last step inserting \eqref{eq:lower_dense_epoch_delta} and multiplying
both its branches by $b_1K$.
Only the mean of worker $1$ changes, by at most $2\delta_M$ at each
boundary, so the total variation is at most
\begin{equation}
2M\delta_M.
\label{eq:lower_dense_epoch_variation}
\end{equation}
Thus, whenever the total variation is at most $V_K$, every realization is
an admissible deterministic path, and at least one realization attains
the average bound.
It remains to choose the epoch count $M$. Abbreviate
\[
M_0:=\left(V_K^2b_1K\right)^{1/3},
\]
the epoch count at which the variation
\eqref{eq:lower_dense_epoch_variation} balances the budget. The choice is
$M=\min\bigl\{\Theta\bigl(\max\{M_0,1\}\bigr),\,\lfloor K/2\rfloor\bigr\}$:
the cap keeps $\ell=\lfloor K/M\rfloor\ge2$ and the final block below half
of the horizon, and the implicit constant is small
enough that $M\le M_0$ when $M_0\gtrsim1$ and that it absorbs the factor
$c_0$ in the variation check below. When the cap binds, $M_0\gtrsim K$,
which under the hypothesis $V_K\lesssim Kn^2/B^2$ forces $b_1=\Theta(1)$
and $V_K=\Theta(K)$; the target is then already of order $K$, and the
capped choice attains it, since \eqref{eq:lower_dense_epoch_regret} gives
$\Theta(Kb_1\delta_M)=\Theta(K)$ there while the variation consumed is
$\Theta(c_0K)\le V_K$.

Consider first the case $M_0\lesssim1$, equivalently
$V_K\lesssim1/\sqrt{b_1K}$. The choice gives $M=\Theta(1)$, and the
stationary bound \eqref{eq:lower_dense_stationary}, valid for every
budget, dominates. The case $b_1\ge K$ falls here as well: the hypothesis
$V_K\lesssim Kn^2/B^2$ gives $M_0\lesssim K/b_1\le1$ there, since
$b_1=\Theta(B/n)$.

Suppose henceforth that $M_0\gtrsim1$. Since $b_1=\Theta(B/n)$, the
hypothesis $V_K\lesssim Kn^2/B^2$ reads $V_K\lesssim K/b_1^2$, which gives
$M_0\lesssim K/b_1$, hence $b_1M\lesssim K$. The gap of
\eqref{eq:lower_dense_epoch_delta} therefore sits in its second branch,
$\delta_M=c_0\sqrt{M/(b_1K)}$, and the minimum in
\eqref{eq:lower_dense_epoch_regret} attains its second entry.

For admissibility, inserting the definition of $M_0$ gives
\[
2M\delta_M
=
2c_0\,\frac{M^{3/2}}{\sqrt{b_1K}}
\le
2c_0\,\frac{M_0^{3/2}}{\sqrt{b_1K}}
=
2c_0V_K
\le
V_K,
\]
where the second step uses $M\le M_0$, and the last holds because the
implicit constant of the choice absorbs $c_0$. Hence every realized mean
path is admissible. If the choice returns $M=1$, which happens only when
$M_0=\Theta(1)$, admissibility follows directly, since
$2M\delta_M=2\delta_1=2c_0\min\{1/b_1,1/\sqrt{b_1K}\}$ and $M_0\gtrsim1$
gives $V_K\gtrsim1/\sqrt{b_1K}$.

For the regret evaluation, inserting the definition of $M_0$ into the
second entry of \eqref{eq:lower_dense_epoch_regret} gives
\[
\sqrt{b_1KM}
=
\Theta\left(\sqrt{b_1KM_0}\right)
=
\Theta\bigl(b_1^{2/3}K^{2/3}V_K^{1/3}\bigr).
\]

Since $b_1=\Theta(B/n)$, combining the regimes gives
\begin{equation}
\inf_\pi\sup_{\nu\in\mathfrak E_{n,B}^{[0,1]}(K,V_K)}
\regretdyn_K(\pi,\nu)
\ge
c\left(\frac Bn\right)^{2/3}K^{2/3}V_K^{1/3},
\qquad B\ge n.
\label{eq:lower_dense_nonstationary}
\end{equation}
As above, there is no cap at $K$: under the hypothesis
$V_K\lesssim Kn^2/B^2$ of the regime, this term does not exceed $K$.

Finally, the stationary and nonstationary lower bounds of the two regimes
combine, up to universal constants, into the two displays of
\cref{thm:lower_bound}: in each regime the condition on $K$ bounds the
stationary terms of \eqref{eq:lower_sparse_stationary} and
\eqref{eq:lower_dense_stationary} by $K$, and the condition on $V_K$
removes the cap at $K$ from the variation term, as noted after
\eqref{eq:lower_sparse_nonstationary} and
\eqref{eq:lower_dense_nonstationary}. This completes the proof.
\end{proof}

\section{Experimental details}
\label{app:paper_extra}

This appendix collects the details of the experiment of
\Cref{sec:experiments}: the evaluation metrics
(\Cref{app:paper_experiments:setup}), the parameter selection
(\Cref{app:paper_param_selection}), and the full results at the remaining
pool sizes (\Cref{app:paper_extra:drift}).

\subsection{Evaluation metrics}
\label{app:paper_experiments:setup}

The run is evaluated on the quantities of \Cref{sec:setup}, most of which
already appear in \Cref{sec:experiments}: the suboptimality
$f(x_k) - f^*$ of the optimization run, the partial proxy-loss regret
$\bar\regretdyn_k := \sum_{q=1}^k \bigl(\ell_q(\ba_q) - L_q^*\bigr)$ up to
iteration $k$, where $L_q^*$ is the loss of the exact \algname{RAS}
allocation of the round's true means, the runtime
$C_k := \sum_{q=1}^k C_q(\ba_q)$, and the total worker time
$\sum_{i=1}^n T_{i,k}$, where
$T_{i,k}$ is the total time worked by worker $i$ up to iteration $k$. For
the simplex methods, $T_{i,k}$ accumulates the realized delays of the
assigned tasks,
\begin{equation}
T_{i,k}
=
\sum_{q=1}^k\sum_{u=1}^{a_{i,q}} X_{i,q}^{(u)},
\label{eq:paper_worker_time}
\end{equation}
and for \algname{Rennala} every worker stays busy for the whole round, so
$T_{i,k} = C_k$. The figures report these quantities directly or averaged
per iteration, and \cref{tab:app_metrics_map} summarizes where each of
them appears. Since \algname{Rennala} is not a simplex allocation, it
has no proxy-loss regret and is absent from the regret panels.

\begin{table}[h]
\centering
\small
\caption{The evaluation metrics, their definitions, and the figures and
tables where each of them is reported.}
\label{tab:app_metrics_map}
\begin{tabular}{lll}
\toprule
Metric & Definition & Reported in \\
\midrule
Suboptimality & $f(x_k) - f^*$ & \cref{fig:paper_sgd_drift_n20_mini}, \cref{fig:paper_sgd_drift_time_app} \\
Runtime & $C_k = \sum_{q=1}^k C_q(\ba_q)$ & \cref{fig:paper_sgd_drift_time_app} (top), \cref{tab:paper_to_target} \\
Total worker time & $\sum_{i=1}^n T_{i,k}$, \eqref{eq:paper_worker_time} & \cref{fig:paper_sgd_drift_time_app} (bottom), \cref{tab:paper_to_target} \\
Avg. iteration time & $C_k / k$ & \cref{fig:paper_sgd_drift_iters_app} (top) \\
Avg. worker time & $\frac{1}{k}\sum_{i=1}^n T_{i,k}$ & \cref{fig:paper_sgd_drift_iters_app} (middle) \\
Avg. proxy-loss regret & $\bar\regretdyn_k / k$ & \cref{fig:paper_sgd_drift_iters_app} (bottom) \\
\bottomrule
\end{tabular}
\end{table}

\subsection{Parameter selection}
\label{app:paper_param_selection}

The confidence schedule uses $\alpha = 1/2$ and $\beta_K = 4\log(2BK^3)$
throughout. The sub-exponential constant $\alpha$ enters both the
confidence radii and the choice of the window below. For the
shifted-exponential family $\mu/2 + \mathrm{Exp}(\mu/2)$ of
\Cref{sec:experiments} the centering removes the deterministic shift, so
the Orlicz constant of \cref{ass:drift-exp} is the norm of the centered
exponential part: solving
$\Exp{\exp\bigl(|E - \lambda|/t\bigr)} = 2$ for
$E \sim \mathrm{Exp}(\lambda)$ gives
$\alpha_{i,k} = \|E - \mu_{i,k}/2\|_{\psi_1} \approx 1.53\,\mu_{i,k}/2
\approx 0.77\,\mu_{i,k}$, and the prescribed value is
$\alpha \approx 0.77\max_{i,k}\mu_{i,k} \approx 8.4$. This value is
conservative in practice, and the
methods are run with the smaller constant $\alpha = 1/2$: a sweep over the
grid $\{2^{-3},\dots,2^{3}\}$ on the instances of \Cref{sec:experiments}
places the empirical optimum in $[1/4, 1]$, growing mildly with the pool
size, and the final regret at $\alpha = 1/2$ is within $12\%$ of the best
value over the grid on every instance.

The window of \algname{SW-ATA} follows the prescription of
\cref{thm:general_ns_gap_independent}:
\begin{equation}
W =
\begin{cases}
\left\lceil \max\Bigl\{1,\; \Bigl(\dfrac{K^2 n \alpha^2 \beta_K}
{B V_K^2}\Bigr)^{1/3}\Bigr\} \right\rceil,
& V_K \ge \alpha\sqrt{\dfrac{n \beta_K}{B K}},\\[2mm]
K, & \text{otherwise (including } V_K = 0\text{)},
\end{cases}
\label{eq:paper_window_app}
\end{equation}
and the confidence radius of \algname{ATA} of \citet{maranjyan2025ata} is
used as is.

The window \eqref{eq:paper_window_app} requires the variation
budget $V_K$ as input, which is unknown in practice. We therefore evaluate
it at the reference budget $V_K = \sqrt K$, and select the
multiplicative constant of the window once from the grid
$\{1/2, 1, 2, 4\}$ on the drifting instance of \Cref{sec:experiments}. The
final regret is flat for the multipliers in $[1, 2]$ at every pool size
and grows beyond the multiplier $4$, and the multiplier $1$ is within
$2\%$ of the empirical optimum at $n = 20$ and $n = 40$ and within $19\%$
at $n = 10$. The window is therefore used with the multiplier $1$,
exactly as prescribed. The reported runs thus
use \eqref{eq:paper_window_app} with $V_K = \sqrt K$, which gives
$W = 47$, $59$, and $74$ at $n = 10$, $20$, and $40$.

\subsection{Drifting delays in the optimization loop}
\label{app:paper_extra:drift}

\paragraph{Overview.}
While \Cref{sec:experiments} focused on the suboptimality of the runs at
$n = 20$, this section reports the same experiment in full: all three
pool sizes $n = 10, 20, 40$ and all the metrics of
\cref{tab:app_metrics_map}. To give a complete picture of the behavior of
the methods under drift, we examine them at both the global and the
per-iteration scale: \Cref{fig:paper_sgd_drift_time_app} shows the
suboptimality against the runtime and the total worker time,
\cref{tab:paper_to_target} reports when each method reaches a target
suboptimality, and \Cref{fig:paper_sgd_drift_iters_app} resolves the
same runs per iteration.

\paragraph{Suboptimality against time.}
\Cref{fig:paper_sgd_drift_time_app} plots the suboptimality gap
$f(x_k) - f^*$ of the objective against the runtime (top row) and against
the total worker time (bottom row). The curves of the
simplex methods share one shape and differ essentially by a horizontal
shift. In runtime, \algname{Rennala} has the smallest value at every pool
size, the oracle \algname{OFTA} is next, and \algname{SW-ATA} has the
smallest runtime among the learning methods: it essentially coincides
with the oracle at $n = 10$, and its runtime is within a factor of about
one and a half of the oracle's at $n = 40$. \algname{ATA} and the
non-adaptive \algname{UTA} are the slowest methods and nearly overlap:
under a persistent drift the full-history estimates of \algname{ATA} lag
the drift, so its allocation approaches the uniform one. The gap between
the oracle and this pair widens as the pool grows, since the oracle
concentrates the fixed batch on the fast part of the pool, while the
allocations of \algname{ATA} and \algname{UTA} do not follow the drift
fast enough.

On worker time (bottom row) the ordering partly reverses:
\algname{Rennala}, the fastest method
in runtime, has a high worker time that grows with the pool, while
\algname{SW-ATA} has the smallest worker time among the learning methods:
it closely tracks the oracle and even improves on it at $n = 10$ and
$n = 20$.

\paragraph{Target suboptimality.}
For a concrete reference point we consider the target suboptimality
$10^{-5}$: \cref{tab:paper_to_target} reports
the runtime and the total worker time at which every method first reaches
it and confirms the trends of \Cref{fig:paper_sgd_drift_time_app}: \algname{Rennala} reaches the
target first at every pool size, \algname{SW-ATA} essentially matches the
oracle, and \algname{UTA} is the slowest by both measures.

\paragraph{Per-iteration dynamics.}
\Cref{fig:paper_sgd_drift_iters_app} resolves the same runs along the
iteration axis. The average iteration time $C_k/k$ (top row) splits the
methods into three tiers: \algname{UTA} on top; \algname{ATA} approaches
it as the drift accumulates in its full-history estimates, until the two
nearly coincide at $n = 40$; and the drift-aware methods at the bottom,
where \algname{SW-ATA} is the fastest learning method at every pool
size: it matches the oracle at $n = 10$ and stays within a factor of
about one and a half of it at $n = 40$. The average worker
time per iteration $\frac1k\sum_{i=1}^n T_{i,k}$ (middle row) orders the
methods differently: the uniform baseline has the highest value,
\algname{Rennala} spends more than the oracle --- about three times more
at $n = 40$ --- and \algname{SW-ATA} stays at or below the oracle's
level. The oracle's allocation minimizes the round's proxy loss, not
the worker time spent on it, so its worker time is not a lower bound: at
$n = 10$ \algname{SW-ATA} has slightly less than the oracle itself. The
average regret $\bar\regretdyn_k/k$ (bottom row) shows the ranking inside
the optimization loop: \algname{SW-ATA} is the best learner at every pool
size, and \algname{ATA} stays at the level of the non-adaptive
baseline. At $n = 40$, where the batch is smaller than the pool, the
uniform allocation itself becomes less wasteful, and the regret of every
method contracts towards the oracle.

\begin{figure}[t]
\centering
\includegraphics[width=0.78\linewidth]{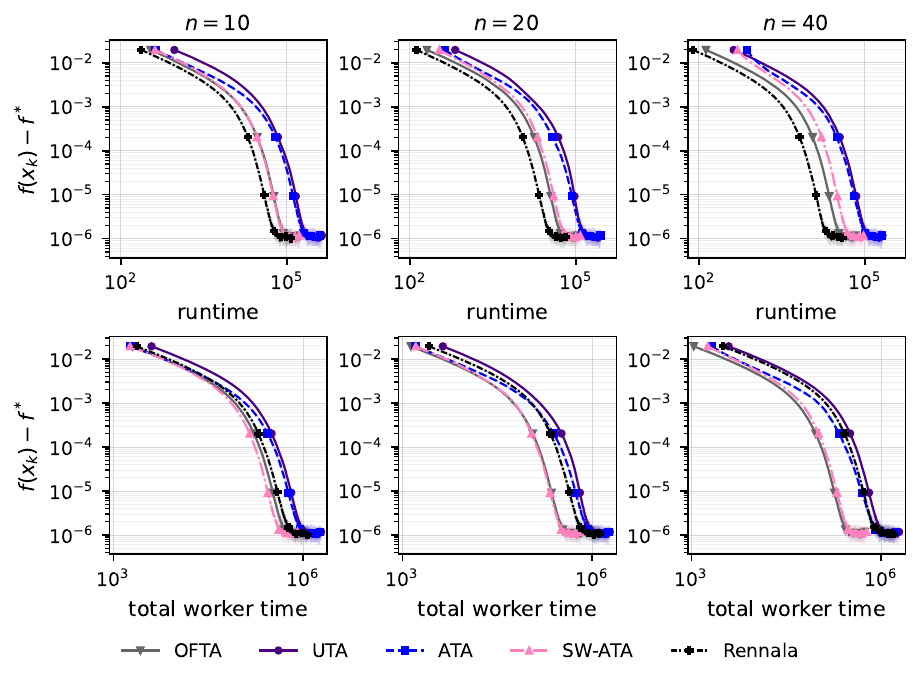}
\caption{Each column corresponds to one pool size, $n = 10, 20, 40$ from
left to right. The top row shows the suboptimality $f(x_k) - f^*$ against
the runtime $C_k$, and the bottom row shows the same against the
\textit{total worker time} $\sum_{i=1}^n T_{i,k}$ of
\eqref{eq:paper_worker_time}. The curves of the simplex methods differ essentially by a
horizontal shift. In runtime, \algname{Rennala} has the
smallest value and \algname{SW-ATA} the smallest value among the learning
methods, essentially coinciding with the oracle at $n = 10$;
\algname{ATA} overlaps with the non-adaptive \algname{UTA} on the right. On
worker time the ordering partly reverses: the worker time of
\algname{Rennala} grows with the pool, and \algname{SW-ATA} is the most
economical learning method, below the oracle at $n = 10$ and $n = 20$.}
\label{fig:paper_sgd_drift_time_app}
\end{figure}

\begin{table*}[t]
\centering
\caption{The runtime and the total worker time at which every method first
reaches the target suboptimality $10^{-5}$ on the drifting instance, for
the three pool sizes. Median over $15$ trajectories; every trajectory
reaches the target in every cell. The best value in every column is in
bold.}
\label{tab:paper_to_target}
\begin{tabular}{lrrrrrr}
\toprule
 & \multicolumn{2}{c}{$n=10$} & \multicolumn{2}{c}{$n=20$} & \multicolumn{2}{c}{$n=40$} \\
\cmidrule(lr){2-3}\cmidrule(lr){4-5}\cmidrule(lr){6-7}
method & runtime & worker time & runtime & worker time & runtime & worker time \\
\midrule
\algname{OFTA} & $5.67\times10^{4}$ & $3.27\times10^{5}$ & $3.44\times10^{4}$ & $2.24\times10^{5}$ & $2.21\times10^{4}$ & $\mathbf{1.78\times10^{5}}$ \\
\algname{UTA} & $1.43\times10^{5}$ & $6.35\times10^{5}$ & $9.33\times10^{4}$ & $6.30\times10^{5}$ & $6.84\times10^{4}$ & $6.29\times10^{5}$ \\
\algname{ATA} & $1.28\times10^{5}$ & $5.73\times10^{5}$ & $8.33\times10^{4}$ & $5.40\times10^{5}$ & $6.34\times10^{4}$ & $4.84\times10^{5}$ \\
\algname{SW-ATA} & $5.58\times10^{4}$ & $\mathbf{2.74\times10^{5}}$ & $3.91\times10^{4}$ & $\mathbf{2.13\times10^{5}}$ & $3.13\times10^{4}$ & $1.97\times10^{5}$ \\
\algname{Rennala} & $\mathbf{3.87\times10^{4}}$ & $3.87\times10^{5}$ & $\mathbf{2.15\times10^{4}}$ & $4.29\times10^{5}$ & $\mathbf{1.30\times10^{4}}$ & $5.18\times10^{5}$ \\
\bottomrule
\end{tabular}
\end{table*}

\begin{figure}[t]
\centering
\includegraphics[width=0.8\linewidth]{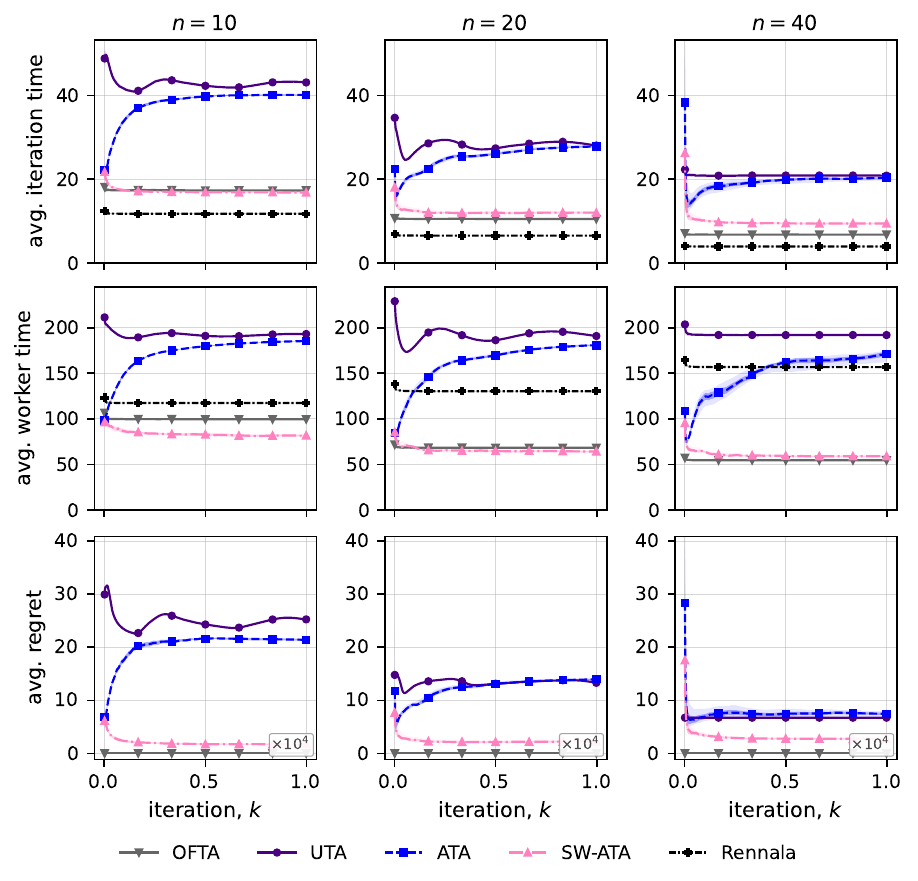}
\caption{Each column corresponds to one pool size, $n = 10, 20, 40$ from
left to right. The top row shows the average iteration time $C_k/k$
against the iteration $k$, the middle row shows the average worker time
per iteration $\frac1k\sum_{i=1}^n T_{i,k}$ of
\eqref{eq:paper_worker_time}, and the bottom row shows the average regret
$\bar\regretdyn_k/k$; the iteration axis is in units of $10^4$
iterations. \algname{Rennala} has no proxy-loss regret and is absent from the
bottom row.}
\label{fig:paper_sgd_drift_iters_app}
\end{figure}

\subsection{Protocol and reproducibility}
\label{app:paper_experiments:repro}

Every figure of this appendix is produced by the script of its experiment
from the stored per-trajectory series, so the appendix and the main text
always show the same run, and the variation budget and the growth exponent
of each regime are read from the recorded metadata, not quoted from the
construction. Reported values are aggregated over the $15$ trajectories as
the median with a $95\%$ confidence band. The code that reproduces every
figure and table of the experimental part is provided in the supplementary
material. All experiments are CPU-only: the full experiment (three pool
sizes, $15$ trajectories, $K = 10^4$ iterations) runs in about one minute
on a 112-core Intel Xeon Gold 6348 machine, and the tuning sweeps take a
few minutes each.

\end{document}

%% file: math_commands.tex
\usepackage{amsmath,amsfonts,bm}

\def\eqref#1{equation~\ref{#1}}

\def\1{\bm{1}}

\DeclareMathAlphabet{\mathsfit}{\encodingdefault}{\sfdefault}{m}{sl}
\SetMathAlphabet{\mathsfit}{bold}{\encodingdefault}{\sfdefault}{bx}{n}

\newcommand{\E}{\mathbb{E}}

\newcommand{\R}{\mathbb{R}}

\DeclareMathOperator*{\argmax}{arg\,max}
\DeclareMathOperator*{\argmin}{arg\,min}

%% file: sample_paper.bib
@article{vershynin2019high,
  title={High-dimensional probability},
  author={Vershynin, Roman},
  journal={Cambridge Series in Statistical and Probabilistic Mathematics},
  volume={47},
  year={2019}
}

@inproceedings{mitliagkas2016asynchrony,
  title={Asynchrony begets momentum, with an application to deep learning},
  author={Mitliagkas, Ioannis and Zhang, Ce and Hadjis, Stefan and R{\'e}, Christopher},
  booktitle={2016 54th Annual Allerton Conference on Communication, Control, and Computing (Allerton)},
  pages={997--1004},
  year={2016},
  organization={IEEE}
}

@inproceedings{kveton2015tight,
  title={Tight regret bounds for stochastic combinatorial semi-bandits},
  author={Kveton, Branislav and Wen, Zheng and Ashkan, Azin and Szepesvari, Csaba},
  booktitle={Artificial Intelligence and Statistics},
  pages={535--543},
  year={2015},
  organization={PMLR}
}

@inproceedings{takemura2021near,
  title={Near-optimal regret bounds for contextual combinatorial semi-bandits with linear payoff functions},
  author={Takemura, Kei and Ito, Shinji and Hatano, Daisuke and Sumita, Hanna and Fukunaga, Takuro and Kakimura, Naonori and Kawarabayashi, Ken-ichi},
  booktitle={Proceedings of the AAAI Conference on Artificial Intelligence},
  volume={35},
  number={11},
  pages={9791--9798},
  year={2021}
}

@article{maranjyan2025ringmaster,
  title={Ringmaster ASGD: The first Asynchronous SGD with optimal time complexity},
  author={Maranjyan, Artavazd and Tyurin, Alexander and Richt{\'a}rik, Peter},
  journal={arXiv preprint arXiv:2501.16168},
  year={2025}
}

@inproceedings{peng2018optimus,
  title={Optimus: an efficient dynamic resource scheduler for deep learning clusters},
  author={Peng, Yanghua and Bao, Yixin and Chen, Yangrui and Wu, Chuan and Guo, Chuanxiong},
  booktitle={Proceedings of the Thirteenth EuroSys Conference},
  pages={1--14},
  year={2018}
}

@article{dean2013tail,
  title={The tail at scale},
  author={Dean, Jeffrey and Barroso, Luiz Andr{\'e}},
  journal={Communications of the ACM},
  volume={56},
  number={2},
  pages={74--80},
  year={2013},
  publisher={ACM New York, NY, USA}
}

@inproceedings{rajbhandari2020zero,
  title={Zero: Memory optimizations toward training trillion parameter models},
  author={Rajbhandari, Samyam and Rasley, Jeff and Ruwase, Olatunji and He, Yuxiong},
  booktitle={SC20: international conference for high performance computing, networking, storage and analysis},
  pages={1--16},
  year={2020},
  organization={IEEE}
}

@article{shoeybi2019megatron,
  title={Megatron-lm: Training multi-billion parameter language models using model parallelism},
  author={Shoeybi, Mohammad and Patwary, Mostofa and Puri, Raul and LeGresley, Patrick and Casper, Jared and Catanzaro, Bryan},
  journal={arXiv preprint arXiv:1909.08053},
  year={2019}
}

@article{sergeev2018horovod,
  title={Horovod: fast and easy distributed deep learning in TensorFlow},
  author={Sergeev, Alexander and Del Balso, Mike},
  journal={arXiv preprint arXiv:1802.05799},
  year={2018}
}

@inproceedings{islamov2024asgrad,
  title={AsGrad: A sharp unified analysis of asynchronous-SGD algorithms},
  author={Islamov, Rustem and Safaryan, Mher and Alistarh, Dan},
  booktitle={International Conference on Artificial Intelligence and Statistics},
  pages={649--657},
  year={2024},
  organization={PMLR}
}

@inproceedings{dutta2018slow,
  title={Slow and stale gradients can win the race: Error-runtime trade-offs in distributed SGD},
  author={Dutta, Sanghamitra and Joshi, Gauri and Ghosh, Soumyadip and Dube, Parijat and Nagpurkar, Priya},
  booktitle={International conference on artificial intelligence and statistics},
  pages={803--812},
  year={2018},
  organization={PMLR}
}

@inproceedings{chen2021combinatorial,
  title={Combinatorial semi-bandit in the non-stationary environment},
  author={Chen, Wei and Wang, Liwei and Zhao, Haoyu and Zheng, Kai},
  booktitle={Uncertainty in Artificial Intelligence},
  pages={865--875},
  year={2021},
  organization={PMLR}
}

@inproceedings{garivier2011upper,
  title={On upper-confidence bound policies for switching bandit problems},
  author={Garivier, Aur{\'e}lien and Moulines, Eric},
  booktitle={International conference on algorithmic learning theory},
  pages={174--188},
  year={2011},
  organization={Springer}
}

@article{nakamura2025multi,
  title={Multi-Play Combinatorial Semi-Bandit Problem},
  author={Nakamura, Shintaro and Kuroki, Yuko and Chen, Wei},
  journal={arXiv preprint arXiv:2509.09933},
  year={2025}
}

@inproceedings{chen2013combinatorial,
  title={Combinatorial multi-armed bandit: General framework and applications},
  author={Chen, Wei and Wang, Yajun and Yuan, Yang},
  booktitle={International conference on machine learning},
  pages={151--159},
  year={2013},
  organization={PMLR}
}

@article{slivkins2019introduction,
  title={Introduction to multi-armed bandits},
  author={Slivkins, Aleksandrs},
  journal={Foundations and Trends{\textregistered} in Machine Learning},
  volume={12},
  number={1-2},
  pages={1--286},
  year={2019},
  publisher={Emerald Publishing Limited}
}

@article{lattimore2019bandit,
  title={Bandit algorithms},
  author={Lattimore, Tor and Szepesv{\'a}ri, Csaba},
  journal={Cambridge University},
  volume={355},
  year={2019}
}

@article{besbes2014stochastic,
  title={Stochastic multi-armed-bandit problem with non-stationary rewards},
  author={Besbes, Omar and Gur, Yonatan and Zeevi, Assaf},
  journal={Advances in neural information processing systems},
  volume={27},
  year={2014}
}

@article{tyurin2023optimal,
  title={Optimal time complexities of parallel stochastic optimization methods under a fixed computation model},
  author={Tyurin, Alexander and Richt{\'a}rik, Peter},
  journal={Advances in Neural Information Processing Systems},
  volume={36},
  pages={16515--16577},
  year={2023}
}

@article{maranjyan2025ata,
  title={ATA: Adaptive task allocation for efficient resource management in distributed machine learning},
  author={Maranjyan, Artavazd and Saad, El Mehdi and Richt{\'a}rik, Peter and Orabona, Francesco},
  journal={arXiv preprint arXiv:2502.00775},
  year={2025}
}

@article{chen2016combinatorial,
  title={Combinatorial multi-armed bandit with general reward functions},
  author={Chen, Wei and Hu, Wei and Li, Fu and Li, Jian and Liu, Yu and Lu, Pinyan},
  journal={Advances in Neural Information Processing Systems},
  volume={29},
  year={2016}
}

@article{auer2002finite,
  title={Finite-time analysis of the multiarmed bandit problem},
  author={Auer, Peter and Cesa-Bianchi, Nicolo and Fischer, Paul},
  journal={Machine learning},
  volume={47},
  number={2},
  pages={235--256},
  year={2002},
  publisher={Springer}
}

@inproceedings{cheung2019learning,
  title={Learning to optimize under non-stationarity},
  author={Cheung, Wang Chi and Simchi-Levi, David and Zhu, Ruihao},
  booktitle={The 22nd International Conference on Artificial Intelligence and Statistics},
  pages={1079--1087},
  year={2019},
  organization={PMLR}
}
